\documentclass[12pt]{article}
\newcommand{\nihil}[1]{}
\nihil{
	\evensidemargin 0.05in \oddsidemargin 0.05in \textwidth 6.4in
	\topmargin - 0.7 in \overfullrule = 0pt
	\textheight 8.8 in
}

\usepackage[normalem]{ulem}
\usepackage[dvipsnames]{xcolor}
\usepackage{graphicx}
\usepackage{caption}
\usepackage{subcaption}
\usepackage{tcolorbox}
\usepackage{import} 
\usepackage{amsfonts}
\usepackage{authblk}
\usepackage{dsfont}
\usepackage{amsmath}
\usepackage{enumerate}
\usepackage{amstext}
\usepackage{amsopn}
\usepackage{amsbsy}
\usepackage{amscd}
\usepackage{amsxtra}
\usepackage{bbm}
\usepackage{amsthm}
\usepackage[numbers,sort]{natbib} 
\usepackage{amssymb}
\usepackage{esint}
\usepackage[margin=.75in]{geometry}
\usepackage{abstract,lipsum,mathrsfs}
\usepackage[new]{old-arrows}
\usepackage{environ,etoolbox}
\newcommand{\h}{\hspace}
\newcommand{\p}{\partial}
\numberwithin{equation}{section}
\usepackage[hidelinks]{hyperref}
\usepackage[makeroom]{cancel}
\usepackage{float}
\usepackage[title,titletoc]{appendix}
\newcommand{\mycomment}[1]{} 
\usepackage{cleveref}

\makeatother
\makeatletter
\newcommand\Const[3]{%
	\@ifundefined{#1-#2}%
	{\stepcounter{#3}\expandafter\xdef\csname #1-#2\endcsname{\arabic{#3}}}%
	{}%
	\ifnum\pdfstrcmp{#1}{eps}=0 
	\varepsilon_{\csname #1-#2\endcsname}%
	\else
	#1_{\csname #1-#2\endcsname}%
	\fi
}
\makeatother
\newcommand\C[1]{\Const{C}{#1}{Ccnt}} 

\newcommand{\pcm}[1]{{\color{purple}#1}}

\allowdisplaybreaks

\usepackage{scalerel,stackengine}
\newcommand\pig[1]{\scalerel*[5.5pt]{\big#1}{ 
		\ensurestackMath{\addstackgap[1.15pt]{\big#1}}}}
\newcommand\pigl[1]{\mathopen{\pig{#1}}}
\newcommand\pigr[1]{\mathclose{\pig{#1}}}

\makeatletter 
\newcommand{\customlabel}[2]{%
	\protected@write \@auxout {}{\string \newlabel {#1}{{#2}{\thepage}{#2}{#1}{}} }%
	\hypertarget{#1}{}
}
\makeatother

\newcommand\specialcomment[3]{%
	\newtoggle{#1}\toggletrue{#1}
	\NewEnviron{#1}{\iftoggle{#1}{\BODY}{}}
}
\newcommand{\excludecomment}[1]{\togglefalse{#1}}

\specialcomment{com}{}{} 
\excludecomment{com} 

\newtheorem{theorem}{Theorem}[section]
\newtheorem{lemma}{Lemma}[section]
\newtheorem{proposition}{Proposition}[section]
\newtheorem{corollary}{Corollary}[section]
\newtheorem{definition}{Definition}[section]
\newtheorem{remark}{Remark}[section]

\newtheorem{assumption}{Assumption}[section]

\def\XXint#1#2#3{{\setbox0=\hbox{$#1{#2#3}{\int}$}
		\vcenter{\hbox{$#2#3$}}\kern-.5\wd0}}

\title{Consumption and Investment in Incomplete Markets with Epstein–Zin Preferences in Infinite Horizon}
\date{\today}
\author[,1]{\normalsize Ho Man Tai\footnote{E-mail: homan.tai@sydney.edu.au}}
\affil[1]{\small\it School of Mathematics and Statistics, The University of Sydney, Sydney, Australia}

\begin{document}
	\maketitle
	
	\vspace{-30pt}
	
	\begin{abstract}
		{\bf Abstract:}	We solve the optimal consumption–investment problem in incomplete markets, where the investor aims to maximise an Epstein--Zin type stochastic differential utility from consumption over an infinite time horizon. We verify that the optimal strategies can be characterised by the limit of a sequence of solutions of the HJB equation in bounded domains with carefully designed boundary conditions. 
		We also conduct numerical experiments for three market models. Explicit solutions are constructed under some parameter regimes.
	\end{abstract}
	
	\noindent\textbf{Keywords:} Incomplete markets, Epstein-Zin preferences, portfolio choice, stochastic differential
	utility, consumption
	
	\noindent\textbf{Mathematics Subject Classification (2020):} 49L12, 49L20, 93E20, 91G10, 91G80 
	
	\noindent\textbf{JEL classification:} C61, D52, G11
	
	\tableofcontents
	
	
	\section{Introduction}  
	Under time-additive constant relative risk aversion (CRRA) utility, the same parameter determines both relative risk aversion and the elasticity of intertemporal substitution (EIS). This restriction prevents the investor's risk exposure from being separated from the allocation of consumption across time. This motivated the development of recursive preferences. Kreps and Porteus \cite{KrepsPorteus1978} first introduced a discrete-time framework that allows preferences over the timing of the resolution of uncertainty, resolving the asset pricing puzzles. Building on this idea, Epstein and Zin \cite{epstein1989substitution} developed a homothetic recursive utility in which relative risk aversion and the EIS
	can be chosen independently. Duffie and Epstein \cite{DuffieEpstein1992a,DuffieEpstein1992} subsequently formulated the continuous-time analogue through the theory of stochastic differential utility. This framework allows attitudes towards risk and intertemporal substitution to affect consumption and portfolio choice through distinct channels. When investment opportunities vary over time, it also influences the intertemporal hedging component of the optimal portfolio.\smallskip

	
	In this paper, we study the consumption and investment problems over an infinite time horizon under the Epstein--Zin type preference, in incomplete markets. We adopt the formulation of Epstein--Zin stochastic
	differential utility developed by Herdegen, Hobson and Jerome
	\cite{herdegen2023infinite_I,herdegen2023infinite_II,
		herdegen2025proper}. Letting \(\gamma\) denote relative risk aversion and
	\(\psi\) the elasticity of intertemporal substitution, we work under the condition $
	\vartheta
	:=
	\dfrac{1-\gamma}{1-\psi^{-1}}>0$
	which ensures that continuation
	utility and the consumption utility have compatible signs in the
	discounted infinite-horizon equation. The main contributions of this paper are twofold: we establish the well-posedness of the control problem and provide a reliable characterisation of the solutions that can be computed efficiently and accurately.   

	
	\subsection{Related literature}

	Epstein--Zin preferences have been used extensively in asset pricing because
	they separate aversion to long-run risk from the willingness to substitute
	consumption across time. This feature is central to the long-run-risk model of
	Bansal and Yaron \cite{BansalYaron2004}, where persistent movements in expected
	consumption growth and economic uncertainty help explain the equity premium,
	the low risk-free rate, market volatility, and return predictability. Related
	recursive-utility models have been used by Benzoni, Collin-Dufresne and
	Goldstein \cite{BenzoniCollinDufresneGoldstein2011} to study option prices and
	crash-related excess volatility; by Bhamra, Kuehn and Strebulaev
	\cite{BhamraKuehnStrebulaev2010} to examine levered equity premia and credit
	spreads; by Wachter \cite{Wachter2013} to analyse stock-market volatility under
	time-varying disaster risk; and by Ai and Kiku \cite{AiKiku2013} to explain the
	value premium in production economies with growth options.
	
	\smallskip
	
	These asset-pricing applications also motivate a rigorous characterisation of
	the superdifferential of indirect utility. Its supergradient is the
	marginal-value process supporting the investor's optimal consumption plan and,
	after normalisation, serves as a state-price density, or stochastic discount
	factor, for marginal claims. Duffie and Skiadas
	\cite{DuffieSkiadas1994} develop the continuous-time utility-gradient
	approach, building on the stochastic differential utility and asset-pricing
	framework of Duffie and Epstein
	\cite{DuffieEpstein1992,DuffieEpstein1992a}. 
	El Karoui, Peng and Quenez \cite{ElKarouiPengQuenez2001} establish a stochastic
	maximum principle for recursive utilities.\smallskip 
	
	There is a substantial body of finite-horizon literature treating portfolio choice under Epstein--Zin preference. Schroder and Skiadas \cite{SCHRODER2003155} use the utility-gradient
	approach to derive necessary and sufficient first-order conditions, forming a constrained FBSDE. In \cite{SCHRODER20051}, they extend to nontradeable income and constraints on portfolio values by introducing a
	class of translation-invariant recursive preferences. Kraft, Seifried and Steffensen
	\cite{kraft2013consumption} obtain solutions under a parameter
	restriction that linearises the HJB equation. Kraft, Seiferling and Seifried
	\cite{kraft2017optimal} remove this restriction under boundedness
	and ellipticity assumptions and construct the solution by a fixed-point method.
	Xing \cite{xing2017consumption} derives the superdifferential of indirect utility through a
	quadratic-BSDE approach for \(\gamma,\psi>1\) that permits an unbounded market price of risk, while
	Matoussi and Xing \cite{MatoussiXing} study the corresponding convex
	duality. Feng, Tian and Zheng \cite{feng2025consumption} extend BSDE methods to
	non-Markovian markets with unbounded coefficients. Aurand and Huang \cite{AurandHuang_random_horizon} solve the problem in an incomplete market with a possibly
	unbounded random horizon, characterising the optimal strategies through
	superlinear BSDEs. \smallskip

	The infinite-horizon formulation of the portfolio problem is substantially more delicate than its finite-horizon counterpart. Over an infinite horizon,
	there is no terminal date at which the recursion can be anchored. One must
	therefore determine which utility process, if any, should be assigned to a given
	consumption before the optimisation problem can be formulated. However, literature in this direction is limited and focuses mostly on constant markets. \smallskip

	The foundations of infinite-horizon Epstein--Zin utility are developed by Herdegen, Hobson and Jerome
	\cite{herdegen2023infinite_I,herdegen2023infinite_II,
		herdegen2025proper}. For \(0<\vartheta<1\), they
	construct a unique generalised utility process for every admissible consumption stream and
	verify the Merton type solution in a constant-parameter Black--Scholes--Merton
	market. For \(\vartheta>1\), they establish that nonuniqueness is intrinsic,
	introduce proper utility processes, and solve the constant-parameter problem
	over consumption streams with a unique proper utility process. The case \(\vartheta<0\) has recently been considered by Shigeta \cite{Shigeta2026} which shows that the utility bubble identified in
	\cite{herdegen2023infinite_II} can be removed by an order-equivalent
	transformation of the utility index, and constructs a maximal
	infinite-horizon utility as a limit of finite-horizon utilities. Melnyk, Muhle‐Karbe and Seifried \cite{EZ_transaction} solve the small-transaction-cost Epstein–Zin consumption–portfolio problem through a free-boundary variational inequality, proving optimality and deriving arbitrary-order asymptotic expansions; for general markets and dynamic preferences, they provide only a formal leading-order characterisation. Kang, Gou and Huang \cite{KANG2025108675} study the problem in stochastic-volatility markets, deriving exact exponential–quadratic HJB solutions for unit EIS and log-linear approximations. Both papers adopt the traditional difference form of Epstein–Zin utility and work within specified parameter, well-posedness, and verification conditions. We also refer to the insightful works of Guasoni and Wang \cite{guasoni2020consumption}, Gutekunst, Herdegen, and Hobson \cite{gutekunst2025optimal}, and Guasoni, Lawless, and Tai \cite{guasoni2025variational} on infinite-horizon problems with CRRA utility in general market models. \smallskip
	
	
	As this work was being finalised, we became aware of the closely related and insightful work of Bayraktar and Lawless
	\cite{bayraktar2026infinite}, who study the same problem when \(0<\vartheta<1\). They use the variational approach developed by Guasoni, Lawless and Tai \cite{guasoni2025variational} to characterise the value function as the minimiser of an energy functional. In contrast, our current study uses a different approach in which we construct the value function directly from the HJB equation, and also includes the results for $\vartheta>1$. \smallskip
	
	\subsection{Main results} 
	We study an infinite-horizon consumption--investment problem in a
	continuous-time market consisting of one safe asset and \(n\) risky
	assets described in \eqref{eq:assets}. The short rate, risk premia, and volatilities depend on a state
	variable following the dynamics in \eqref{eq:state_process}. Shocks to the state variable may be correlated with asset-return
	shocks. The investor chooses consumption and portfolio allocations, subject to
	admissibility, to maximise Epstein--Zin stochastic differential utility.
	Because the appropriate solution concept for the infinite-horizon utility
	recursion depends on the preference parameters (see \cite{herdegen2023infinite_I,herdegen2023infinite_II,herdegen2025proper}), we defer the precise
	definition of the utility process and the corresponding optimisation
	domain to \Cref{sec:investor_decisions}. Roughly speaking, this paper makes two main contributions:
	\begin{enumerate}
		\item[(1).] We prove the verification theorems and show the well-posedness of the portfolio problems for  $\vartheta\in (0,1)$ and $\vartheta>1$ separately, in incomplete markets. See \Cref{prop:classical_generalised_value_equality} and \Cref{thm:verification_theta_greater_one} respectively; \smallskip
		\item[(2).] We construct the solutions through a sequence of boundary-value problems on bounded domains. See \Cref{prop:pass_R_to_infty E=R,prop:pass_R_to_infty}. This construction is both analytically useful and computationally tractable.
	\end{enumerate} 
	
	Our first main result is a set of verification theorems for problems in incomplete markets. The corresponding problem in the constant-parameter markets was studied in
	\cite{herdegen2023infinite_I,herdegen2023infinite_II,herdegen2025proper}.
	In the present model, investment opportunities vary with a state variable,
	whose shocks need not be fully spanned by the traded assets. The optimal
	consumption--wealth ratio and portfolio weights are therefore
	state-dependent, and the value function is determined by a nonlinear
	equation on the state space. We show that the bounded positive solution
	\(h\) of the reduced HJB equation \eqref{eq: reduced HJB} generates admissible feedback controls and the corresponding value function. Verification
	arguments in two parameter regimes are fundamentally different.\smallskip
	
	When \(0<\vartheta<1\), we build on the generalised utility
	theory of \cite{herdegen2023infinite_I,herdegen2023infinite_II}.
	In particular, \cite[Theorems~6.5, 6.7 and 6.9]
	{herdegen2023infinite_II} constructs a unique, possibly
	extended-valued, generalised Epstein--Zin utility process for every
	consumption stream, while
	\cite[Theorem~8.1]{herdegen2023infinite_II} proves optimality over
	all attainable consumption streams in the constant-parameter markets. We prove the verification result on the full admissible
	class \eqref{def. A(x,y)} for incomplete markets. The candidate controls and consumption--wealth ratio are now state-dependent. We justify the candidate-induced change of measure by proving that its
	density process is a true martingale, and derive  the uniform-integrability and
	transversality estimates needed to remove localisation. When
	\(\gamma>1\), the geometric-Brownian domination used in
	\cite[Theorem~8.1]{herdegen2023infinite_II} is no longer available;
	we instead compare the perturbed HJB supersolution directly with
	the extremal generalised utility process. As a consequence,
	\Cref{prop:classical_generalised_value_equality} shows that, in the
	present incomplete factor model, the value over the full admissible
	class under generalised utility agrees with the value over the
	classically evaluable strategies.\smallskip

	The regime \(\vartheta>1\) presents a different difficulty: the
	infinite-horizon utility equation may have multiple solutions. Following
	\cite{herdegen2025proper}, we select the economically relevant solution
	through properness, see \Cref{def:proper_utility_theta_greater_one}. We first prove that the candidate consumption admits a unique proper utility process (candidate process) in \Cref{lem:candidate_proper_theta_greater_one} and establish an exponential terminal estimate for that process. We then prove the verification result
	in \Cref{thm:verification_theta_greater_one}. When \(\gamma<1\), the
	result applies to the full class of right-continuous consumption streams
	with a unique proper utility process, see \eqref{eq:verification_class_theta_greater_one}. When \(\gamma>1\), multiplication
	of the perturbed HJB supersolution by \(1-\gamma<0\) reverses the
	inequality and produces a subsolution of the transformed utility
	equation. The usual comparison
	argument therefore no longer shows that this subsolution is
	bounded above by the candidate process. We introduce
	\(\mathcal A^\dagger_{\mathrm{ver}}\) in \eqref{eq:verification_class_theta_greater_one} to impose sufficient conditions
	for this required comparison.\smallskip

	Several new ingredients are required. We prove that the stochastic
	exponential arising from the candidate dynamics is a true martingale by
	establishing conservativeness of the state diffusion under
	the candidate-induced change of measure. We then use
	the resulting change of measure to prove the terminal estimate and the
	unique properness of the candidate. For comparison, we introduce the
	random clock $\{\Lambda_t\}_{t\geq0}$ in \eqref{eq:candidate_discount_clock}
	which converts the state-dependent decay rate into exponential decay in
	the new time scale. This allows us to adapt the comparison theorem \cite[Proposition~B.6]{herdegen2025proper} to the present factor model. We also construct
	maximal proper solutions for consumption streams dominated by the
	candidate and establish their monotonicity, which makes it possible to
	pass from truncated consumption streams to a general uniquely proper
	stream. These arguments address difficulties created by the state-dependent and potentially unbounded coefficients.\smallskip
	
	Our second contribution is the construction of solutions to the reduced HJB equation \eqref{eq: reduced HJB}. Since the state space is unbounded, the equation
	need not have a unique solution, and its behaviour at the
	boundary is not prescribed a priori. We therefore solve suitable
	boundary-value problems on bounded intervals. For \(E=\mathbb R\), we impose the Dirichlet boundary condition on bounded intervals; for
	\(E=\mathbb R_+\), we use Neumann--Dirichlet problems. Uniform barriers and interior regularity
	estimates yield locally convergent subsequences whose limits are bounded
	positive solutions of the original HJB equation and satisfy the
	requirements of the verification theorems. \smallskip
	
	More importantly, solutions can be computed efficiently using this bounded-domain
	construction. We implement the resulting numerical scheme in
	mean-reverting risk-premium, Heston-type stochastic-volatility, and
	Cox--Ingersoll--Ross short-rate models. In tractable parameter regimes,
	the reduced equation becomes linear and admits explicit representations.
	Using these formulas as benchmarks, we find that the numerical
	approximations have small relative errors.\smallskip

	\paragraph{Organisation.} The remainder of the paper is organised as follows. \Cref{sec. Problem formulation} introduces the optimisation problems that the investors face. \Cref{sec. HJB Equation and construction of solution} derives the reduced HJB equation and constructs bounded positive solutions on
	\(\mathbb R\) and \(\mathbb R_+\). \Cref{sec. Verification theorem} proves the verification results for
	\(\vartheta\in(0,1)\) and \(\vartheta>1\). \Cref{sec. Examples and numerical results} studies the three factor models. We also presents the explicit solutions and numerical results.

	\section{Problem formulation}\label{sec. Problem formulation}
	
	In this section, we introduce the problem that the investor faces. Let $n \in \mathbb{N}$ be the number of risky assets and let $E$ be either $\mathbb{R}$ or $\mathbb{R}_+ := (0,\infty)$.
	
	\subsection{Financial market}\label{sec. Financial market}
	We first set the probability space and the market structure. The market has one safe asset and $n$ risky assets, whose investment
	opportunities are driven by an $E$-valued state variable $\{Y_t\}_{t \geq 0}$. This process satisfies
	\begin{equation}
		dY_t = b(Y_t)\,dt + a(Y_t)\,dW_t,
		\label{eq:state_process}
	\end{equation}
	where $\{W_t\}_{t \geq 0}$ is a one-dimensional Brownian motion and
	$b, a : E \to \mathbb{R}$ are some functions. The cumulative excess return of $i$-th risky asset evolves as
	\begin{equation}
		dR_t^i
		= \mu^i(Y_t)\,dt + \sum_{j=1}^n \sigma_{ij}(Y_t)\,dZ_t^j,
		\qquad \text{for }i = 1,2, \ldots, n,
		\label{eq. R dynamics}
	\end{equation}
	where $Z = \{Z_t^1,Z_t^2,\ldots,Z_t^n\}_{t \geq 0}$ is an $n$-dimensional
	Brownian motion, $\mu : E \to \mathbb{R}^n$ is the vector of risk premia,
	and $\sigma : E \to \mathbb{R}^{n \times n}$ is the volatility matrix.
	Asset prices satisfy
	\begin{equation}
		\frac{dS_t^0}{S_t^0} = r(Y_t)\,dt,\qquad
		\frac{dS_t^i}{S_t^i}=r(Y_t)\,dt+dR^i_t
		=
		\bigl[r(Y_t)+\mu^i(Y_t)\bigr]\,dt
		+
		\sum_{j=1}^{n}\sigma_{ij}(Y_t)\,dZ_t^j,
		\label{eq:assets}
	\end{equation} 
	for $i = 1,2, \ldots, n$, with $r : E \to \mathbb{R}$ the risk-free rate. The Brownian motions have a constant correlation structure: with $\rho \in \mathbb{R}^n$ and
	$|\rho| \leq 1$,
	\begin{equation}
		d\langle Z^i, W\rangle_t = \rho^i\,dt,
		\qquad \text{for }i = 1,\ldots,n.
		\label{eq:correlation}
	\end{equation}
	If $|\rho| = 1$,
	$\sigma(y)$ has full rank and \(a(y)\neq 0\) on $E$, then the factor shock of $Y$ is entirely spanned by traded
	assets, so state risk is perfectly hedgeable. Any $|\rho| < 1$, by contrast, may introduce an orthogonal source of uncertainty that cannot be traded away; in that case, the market is incomplete.\smallskip

	We must ensure the existence of processes $R$ and $Y$ in \eqref{eq:state_process}-\eqref{eq. R dynamics} with Brownian motions satisfying \eqref{eq:correlation}, under mild assumptions. Such existence has been established in \cite[Section 1.1]{PaoloScott12} and \cite[Section 2.1]{guasoni2020consumption}, we include the argument here for the reader’s convenience. Let $\Omega:=C([0,\infty);\mathbb R^n\times E)$
	be the canonical path space equipped with the topology of locally uniform convergence and $\mathcal{B}$ be its Borel $\sigma$-algebra.  We write $
	G_t(\omega)=\omega(t)$ for the canonical coordinate process and impose the following well-posedness assumption. 
	\begin{assumption}[\bf Well-posedness of the market]\label{ass.market well-posed} We assume that  
		\begin{enumerate}[(a).] 
			\item There is $\alpha \in (0,1)$ such that 
			$r \in C^{1,\alpha}(E;\mathbb{R})$, 
			$\mu \in C^{1,\alpha}(E;\mathbb{R}^n)$, 
			$b \in C^{1,\alpha}(E;\mathbb{R})$, 
			$a^2 \in C^{2,\alpha}(E;\mathbb{R})$, 
			$\Sigma:=\sigma\sigma^\top \in C^{2,\alpha}(E;\mathbb{R}^{n \times n})$ and 
			$\Upsilon:=a\h{.5pt}\sigma\rho \in C^{2,\alpha}(E;\mathbb{R}^n)$. We assume
			that \(a\) and \(\sigma\) are Borel measurable. Moreover, $\rho \in \mathbb{R}^n$ is a constant such that $|\rho|\in[0,1]$. 
			\item For any $y \in E$, $\Sigma(y)$ is strictly positive definite and $a(y)^2>0$.
			\item Let $(w_0,y) \in \mathbb{R}^n\times E$. There exists a unique solution $\mathbb{P}^{(w_0,y)}$ with
			$\mathbb{P}^{(w_0,y)}\pig(G_0 = (w_0,y)\pig) = 1$ and $\mathbb{P}^{(w_0,y)}\pig(G_t \in \mathbb{R}^n\times E \,\text{for all $t\geq0$} \pig) = 1$ solving the martingale problem for the
			generator 
			\begin{align*}
				\widehat{L} &:= \frac{1}{2} \sum_{i,j=1}^{n+1} \widehat{A}_{ij}(\cdot) \partial_{ij} + \sum_{i=1}^{n+1} \widehat{b}_i(\cdot) \partial_i,\quad\text{where } \widehat{A} := \begin{pmatrix} 
					\Sigma & \Upsilon \\
					\Upsilon^\top & a^2
				\end{pmatrix}, \quad \widehat{b} := \begin{pmatrix} 
					\mu\\
					b
				\end{pmatrix}.
			\end{align*} 
		\end{enumerate} 
	\end{assumption} 
	Let
	\[
	\mathcal B_t:=\sigma(G_s:0\le s\le t),
	\qquad
	\mathcal F_t:=\bigcap_{s>t}\mathcal B_s,
	\qquad
	\mathbb F=\{\mathcal F_t\}_{t\ge0}.
	\]
	We work with the \(\mathbb{P}^{(w_0,y)}\)-completion of \(\mathcal B\) and the usual \(\mathbb{P}^{(w_0,y)}\)-augmentation of \(\mathbb F\), without changing notation. 
	Moreover, the filtration is continuous. \Cref{ass.market well-posed} implies that the process $\{M^f_t\}_{t\geq 0}$ defined by
	\begin{align}\label{206}
		M^f_t := f(G_t)-f(G_0)-\int^t_0 \widehat{L} f(G_s) ds
	\end{align}
	is a $(\{\mathcal{B}_t\}_{t \geq0},\mathbb{P}^{(w_0,y)})$-martingale, for every $f \in C^2_c(\mathbb{R}^n \times E;\mathbb{R})$. As the process is continuous, it is also a $(\mathbb{F},\mathbb{P}^{(w_0,y)})$-martingale by \cite[Chapter II, Theorem 2.8]{revuz2013continuous}. Standard localisation/cutoff arguments applied to \eqref{206} give the local martingale $M_t:=G_t-G_0-\int^t_0 \widehat{b}(G_s)ds.$ We have $\langle M^i,M^j \rangle_t =\int^t_0 \widehat{A}_{ij}(G_s)ds$ by \eqref{206}. Levy’s characterisation then identifies
	\[Z_t=\int^t_0 \sigma(Y_s)^{-1} d(\Pi_1(M))_s\quad \text{and} \quad W_t=\int^t_0 \dfrac{1}{a(Y_s)} d(\Pi_2(M))_s\]
	\sloppy as the desired Brownian motions with the corresponding correlation, where $\Pi_1(x)=(x^1,x^2,\ldots,x^n)^\top$ and $\Pi_2(x)=x^{n+1}$ for $x\in \mathbb{R}^{n+1}$. Therefore, $(R,Y)=(\Pi_1(G),\Pi_2(G))$ solves  \eqref{eq:state_process}-\eqref{eq. R dynamics} in the probability space $(\Omega,\mathcal{B},\mathbb{F},\mathbb{P}^{(w_0,y)})$. Throughout the article, we keep $w_0 \in \mathbb{R}^n$ fixed and simply write $\mathbb{P}^{(w_0,y)}=\mathbb{P}$. The corresponding expectation is denoted by $\mathbb{E}$.

	\subsection{Investor's decisions}\label{sec:investor_decisions}
	
	An investor chooses an \(\mathbb F\)-progressively measurable portfolio process $\{\pi_t\}_{t\geq 0}=\{\pi_t^1,\pi_t^2,\ldots,\pi_t^n\}_{t\geq 0}$ and an \(\mathbb F\)-progressively measurable consumption-wealth ratio \(\{l_t\}_{t\geq 0}\). Here \(\pi^i\) is the fraction of wealth invested in the \(i\)-th risky asset, and \(1-\sum_{i=1}^n\pi^i\) is invested in the safe asset. For \(x>0\) and \(y\in E\), we set $\xi=(x,y)$. The wealth process $X^{\h{.5pt}\xi,\pi,l}$ satisfies
	\begin{align}
		\frac{dX_t^{\h{.5pt}\xi,\pi,l}}{X_t^{\h{.5pt}\xi,\pi,l}}
		=
		\bigl[
		r(Y_t)
		+
		\pi_t^\top\mu(Y_t)
		-
		l_t
		\bigr]dt
		+
		\pi_t^\top\sigma(Y_t)\,dZ_t,\quad \text{and $(X^{\h{.5pt}\xi,\pi,l}_0,Y_0)=\xi$.}
		\label{eq. wealth eq}
	\end{align}
	We denote by \(\mathcal A(x,y)\) the set of all admissible controls from initial wealth \(x\) and initial market state \((w_0,y)\):
	\begin{align}
		\mathcal A(x,y):=\left\{\!(\pi,l)\!:\!\begin{aligned}\;&
			\pi,\,l \text{ are } \mathbb F\text{-progressively measurable},\,	l_t\geq0,\,dt\otimes d\mathbb{P}\text{-a.e.}\\
			&
			\text{and}\,\,\int_0^T \!\Big[ |\pi_t^\top \mu(Y_t)|
			+\pi_t^\top \Sigma(Y_t)\pi_t
			+l_t\Big]dt < \infty, \text{ for every } T<\infty,\,\mathbb P\text{-a.s.} \end{aligned}
		\right\}.\label{def. A(x,y)}
	\end{align}
	Thus, for any \((\pi,l) \in \mathcal A(x,y)\), the wealth process $X^{\h{.5pt}\xi,\pi,l}$ in \eqref{eq. wealth eq} exists by the Doléans--Dade exponential formula \cite[Chapter II, Section 8]{Protter1990} and is positive for all time, $\mathbb{P}$-a.s.\smallskip
	
	We consider the continuous-time stochastic differential utility of Epstein--Zin type. Let $\delta$, $\beta$, $\psi$ and $\gamma$ be real numbers. We denote
	\[
	\qquad
	\vartheta:=\frac{1-\gamma}{1-1/\psi},
	\qquad
	q:=\frac{1/\psi-\gamma}{1-\gamma}=1-\dfrac{1}{\vartheta}.
	\] 
	Motivated by \cite{herdegen2023infinite_I}, we consider the discounted Epstein--Zin aggregator in the following form
	\begin{equation}
		g_{EZ}(t,c,v)
		:=
		\beta e^{-\delta t}
		\frac{c^{1-1/\psi}}{1-1/\psi}
		\bigl[(1-\gamma)v\bigr]^q 
		\qquad\text{for $c>0$ and $(1-\gamma)v>0$.}
		\label{eq:discounted_EZ_aggregator}
	\end{equation} 
	Here $\gamma \in(0,\infty)\setminus \{1\}$ represents the investor’s relative risk aversion, $\psi \in(0,\infty)\setminus \{1\}$ is the elasticity of intertemporal substitution (EIS) and $\delta \in \mathbb{R}$ is the impatience rate. The scaling factor $\beta>0$ has no effect on preferences as long as it is positive such that the aggregator $g_{EZ}$ is increasing in $c$. The arguments $c$ and $v$ of $g_{EZ}$ represent the consumption and continuation utility variable respectively.\smallskip
	
	\begin{definition}[\bf Epstein--Zin evaluable strategies]\label{def Epstein--Zin evaluable strategies}
		An admissible control \((\pi,l) \in \mathcal{A}(x,y)\) is called Epstein--Zin evaluable if there exists a unique (up to indistinguishability) $\mathbb{F}$-adapted càdlàg process
		\(J^{\,\xi,\pi,l}\), taking values in $(1-\gamma)\mathbb R_+=\mathbb R_+$ when $\gamma<1$ and $(1-\gamma)\mathbb R_+=\mathbb R_-:=(-\infty,0)$ when $\gamma>1$, such that
		\(g_{EZ}(s,l_sX_s^{\h{.5pt}\xi,\pi,l},J_s^{\h{.5pt}\xi,\pi,l})\) is well-defined
		\(dt\otimes d\mathbb P\)-a.e.,
		\begin{align}
			\mathbb E\!
			\left[
			\int_0^\infty
			\pig|
			g_{EZ}\bigl(s,l_s\,X_s^{\h{.5pt}\xi,\pi,l},J_s^{\h{.5pt}\xi,\pi,l}\bigr)
			\pig|ds
			\right]\!
			<\infty,\text{ and }
			J_t^{\,\xi,\pi,l}
			\!=
			\mathbb E_t^y\!
			\left[
			\int_t^\infty\!
			g_{EZ}\bigl(s,l_s\,X_s^{\h{.5pt}\xi,\pi,l},J_s^{\h{.5pt}\xi,\pi,l}\bigr)ds
			\right],\text{ for }t\ge0.
			\label{eq:EZ_utility eq}
		\end{align}
		Here $\mathbb E_t^y(\cdot):=
		\mathbb E(\,\cdot\mid\mathcal F_t)$ denotes the conditional expectation. The process $\{J_t^{\,\xi,\pi,l}\}_{t\geq0}$ is called the Epstein--Zin utility process or simply utility process associated with consumption $l\,X^{\h{.5pt}\xi,\pi,l}$. We denote by \(\mathcal A_{EZ}(x,y)\) the set of admissible controls that are Epstein--Zin evaluable. 
	\end{definition}
	\begin{remark}[\bf Comparison with classical Epstein--Zin utility]
		The expression in \eqref{eq:discounted_EZ_aggregator} is different from the classical Epstein--Zin aggregator used in continuous setting in the literature. For example, in \cite{xing2017consumption,MatoussiXing,kraft2017optimal,feng2025consumption}, the authors use the difference form
		
		$$
		\frac{\delta}{1-1/\psi}
		c^{1-1/\psi}
		\bigl[(1-\gamma)v\bigr]^q
		-
		\frac{\delta}{1-1/\psi}
		(1-\gamma)v.
		$$
		This difference form subtracts a second term, so the positive and negative parts may both be infinite. One way to avoid this ill-defined expectation is to restrict the admissible
		class so that the difference-form integral is well defined. Any verification theorem
		proved under such a restriction then applies only to that restricted class, whereas the discounted form \eqref{eq:discounted_EZ_aggregator} is designed to make the infinite-horizon problem meaningful on a larger and more natural domain. Moreover, when \(\psi=1/\gamma\), the discounted form reduces directly to the
		standard discounted CRRA utility. By contrast, the difference-form equation
		is equivalent to the standard infinite-horizon CRRA formulation only after imposing
		the appropriate transversality condition. In the infinite-horizon setting, the difference-form formulation is usually justified through finite-horizon equations together with a transversality condition. If this transversality condition is not economically appropriate, the resulting utility process may be driven by an artificial terminal continuation value rather than by the consumption flow itself. This leads to a bubble-type utility process, where the value assigned to consumption is not supported by the accumulated flow utility. We refer \cite[Remark 4.3]{herdegen2023infinite_I} for more discussions. 
	\end{remark}

	To ensure that the set \(\mathcal A_{EZ}(x,y)\) is meaningful, the parameters in the discounted Epstein--Zin aggregator $g_{EZ}$ are required to satisfy:
	\begin{assumption}[\bf Parameter constraints]\label{ass. Parameters constraint} We assume that
		$$\delta \in \mathbb{R},\quad
		\beta>0,\quad
		\psi,\gamma\in(0,\infty)\setminus \{1\},\quad \text{and}\quad \vartheta:= \dfrac{1-\gamma}{1-1/\psi} \in(0,\infty).$$
	\end{assumption}
	\textbf{Throughout the article, we assume without further mention that Assumptions \ref{ass.market well-posed} and \ref{ass. Parameters constraint} hold.}
	\begin{remark}[\bf On the restriction \(\vartheta>0\)] A discussion of these parameter restrictions can be found in \cite[Remark 4.2]{herdegen2023infinite_I}. The last condition that $\vartheta>0$ deserves more discussions. It is necessary for the infinite-horizon discounted
		Epstein--Zin utility equation to have the correct sign structure and an economically
		meaningful solution, see \cite[Theorem 4.4, Remark 4.5]{herdegen2023infinite_I}. This is very different from the finite horizon problems \cite{feng2025consumption,xing2017consumption} where they allow $\vartheta<0$. Indeed, the continuation utility is naturally required to take values in $(1-\gamma)\mathbb R_+$. On the other hand, the running consumption term in the aggregator contains \[ \frac{c^{1-1/\psi}}{1-1/\psi}, \] whose sign is determined by \(1-1/\psi\). Therefore, in the absence of a terminal payoff, the two sides of the infinite-horizon recursive utility equation \eqref{eq:EZ_utility eq} have compatible signs only when $\operatorname{sgn}(1-\gamma) = \operatorname{sgn}\!\left(1-1/\psi\right)$, which is precisely \(\vartheta>0\). This issue is specific to the infinite-horizon formulation. In a finite-horizon problem, one may sometimes compensate for a sign mismatch by adding a bequest term at the terminal time. Thus, following \cite[Theorem~4.4 and Remark~4.5]{herdegen2023infinite_I}, we impose \(\vartheta>0\). The same restriction is also consistent with the Hamilton--Jacobi--Bellman (HJB) viewpoint in our work. 
		\end{remark}

	However, strategies may be excluded from \(\mathcal A_{EZ}(x,y)\) not because they are economically suboptimal, but simply because their associated consumption cannot be evaluated by the utility equation \eqref{eq:EZ_utility eq}. Optimizing only over \(\mathcal A_{EZ}(x,y)\) may lead to an artificially small admissible set. For example, we suppose \(1/\psi>1\). The consumption ratio \(l_t\equiv0\) is financially admissible, but the Epstein--Zin utility equation \eqref{eq:EZ_utility eq} is not directly evaluable because \(c^{1-1/\psi}\) is singular at \(c=0\). Hence the absolute-integrability condition in \eqref{eq:EZ_utility eq} fails. Economically, however, zero consumption should not be removed from the admissible set; rather, when \(1/\psi>1\), it should be assigned the extended value \(-\infty\).\smallskip

	\subsubsection{The regime \texorpdfstring{$\vartheta\in(0,1)$}{theta in (0,1)}}
	For \(\vartheta\in(0,1)\), this difficulty can be avoided by using the generalised Epstein--Zin utility process proposed by \cite[Sections 4-6]{herdegen2023infinite_II}. The generalised formulation assigns a well-defined (by the monotone approximation procedure in \Cref{def:generalised_EZ_solution}) utility process to every admissible control in \(\mathcal A(x,y)\). To see how it works, we first define the generalised discounted Epstein--Zin aggregator:
	
	\begin{definition}[\bf Generalised discounted Epstein--Zin aggregator]
		\label{def:EZ_aggregator}
		Assume $\vartheta\in(0,1)$ (equivalently $q<0$). We adopt the
		conventions $
		0^{q} := \infty$ and $\infty^{q} := 0$. For $c\in[0,\infty)$ and $(1-\gamma)v\in [0,\infty]$, the generalised discounted Epstein--Zin aggregator (simply generalised aggregator) $\widetilde{g}_{EZ}$ is defined by
		\begin{equation*}
			\widetilde{g}_{EZ}(t,c,v)
			:=
			\begin{cases}
				\displaystyle
				\beta e^{-\delta t}
				\frac{c^{1- 1/\psi}}{1-1/\psi}\,
				\big[(1-\gamma) v \big]^{q},
				&c^{1-1/\psi}\in(0,\infty),\quad (1-\gamma)v \in(0,\infty),
				\\[1.0em]
				\displaystyle
				\frac{\beta e^{-\delta t}}{1-\gamma}\,
				\big[(1-\gamma) v \big]^{q},
				&c^{1-1/\psi}\in(0,\infty),\quad (1-\gamma)v =\text{$0$ or $\infty$},
				\\[1.0em]
				\displaystyle
				\beta e^{-\delta t}
				\frac{c^{1-1/\psi}}{1-1/\psi},
				&c^{1-1/\psi}=\text{$0$ or $\infty$},\quad (1-\gamma)v \in[0,\infty].
			\end{cases}
		\end{equation*} 
	\end{definition}

	Let $\mathscr{P}([0,\infty))$ denote the set of all nonnegative progressively measurable processes. Then we define the generalised Epstein--Zin utility process:

	\begin{definition}[\bf Generalised Epstein--Zin utility process]
		\label{def:generalised_EZ_solution}
		Assume that \(\vartheta\in(0,1)\). Let
		\(\{C_t\}_{t\ge0} \in \mathscr{P}([0,\infty))\). The generalised Epstein--Zin utility process $\widetilde J^C$ associated with \(C\) is the
		adapted càdlàg process obtained by the following monotone approximation: If \(\gamma<1\) (resp. \(\gamma>1\)), let
		\(\{C^n\}_{n\in\mathbb N}\) be any nondecreasing (resp. nonincreasing) sequence in $\mathscr{P}([0,\infty))$ such that
		\begin{enumerate}[(a).]
			\item $C^n \to C$, $dt\otimes d\mathbb P$ -a.e.;
			\item for each \(n\in\mathbb N\), there exists a unique utility process \(J^n\) associated with \(C^n\) such that $J_t^n
			=
			\mathbb E_t^y
			\left[
			\int_t^\infty
			\widetilde g_{EZ}
			\bigl(s,C_s^n,J_s^n\bigr)\,ds
			\right],$ for $t\ge0,$
			\item $\mathbb E
			\left[
			\int_0^\infty
			\left|
			\widetilde g_{EZ}
			\bigl(s,C_s^n,J_s^n\bigr)
			\right|ds
			\right]
			<\infty.$		\end{enumerate} 
		Then define $
		\widetilde J_t^C
		:=
		\displaystyle\lim_{n\to\infty}J_t^n
		\in[0,\infty]$ (resp. $\displaystyle\lim_{n\to\infty}J_t^n
		\in[-\infty,0]$) for $t\ge0.$ If \(C=lX^{\h{.5pt}\xi,\pi,l}\) for
		\((\pi,l)\in\mathcal A(x,y)\), then we also write $
		\widetilde J^{\,\xi,\pi,l}:=\widetilde J^C.$
	\end{definition}
	
	From \cite[Theorems 6.5, 6.9]{herdegen2023infinite_II}, such monotone approximations exist and the limit \(\widetilde J^C\) is independent of the chosen approximating sequence. If \(\vartheta\in(0,1)\), then the generalised utility process associated with $l X^{\h{.5pt}\xi,\pi,l}$ exists and is unique for each control in $\mathcal{A}(x,y)$:

	\begin{lemma}[\bf Existence and uniqueness of generalised Epstein--Zin utility]\label{lem:generalised_EZ_existence_comparison}
		Assume \(\vartheta\in(0,1)\). For every \((\pi,l)\in\mathcal A(x,y)\), there exists a unique (up to indistinguishability) generalised utility process $\widetilde J^{\,\xi,\pi,l}$ associated with $l X^{\h{.5pt}\xi,\pi,l}$, satisfying \Cref{def:generalised_EZ_solution}.
	\end{lemma}
	\begin{proof} 
		
		\sloppy Each \((\pi,l)\in\mathcal A(x,y)\) generates the consumption $C^{\h{.5pt}\xi,\pi,l}:=\bigl\{l_t X_t^{\h{.5pt}\xi,\pi,l}\bigr\}_{t\geq0}.$ Since \(l\) and \(X^{\h{.5pt}\xi,\pi,l} \) are in $\mathscr{P}([0,\infty))$, it follows that $C^{\h{.5pt}\xi,\pi,l} \in \mathscr{P}([0,\infty))$. 
		We equivalently absorb the deterministic factor into consumption by setting $
		\overline C_t^{\h{.8pt}\xi,\pi,l}
		:=
		\bigl(\beta e^{-\delta t}\bigr)^{\frac{1}{1-1/\psi}}
		C_t^{\h{.5pt}\xi,\pi,l}$, so \cite[Theorem~6.9]{herdegen2023infinite_II} applies. 
	\end{proof}

	Consequently, under \(\vartheta\in(0,1)\), no additional Epstein--Zin
	evaluability restriction is needed: every admissible control in
	\(\mathcal A(x,y)\) has a well-defined and unique generalised utility process. Therefore, it is safe to consider the problem $	\displaystyle\sup_{(\pi,l)\in\mathcal{A}(x,y)}
	\widetilde J_0^{\,\xi,\pi,l}$.

	\subsubsection{The regime \texorpdfstring{$\vartheta > 1$}{theta > 1}}
	
	In this regime, a consumption induced from the admissible set $\mathcal{A}(x,y)$ may admit more than one utility process. Following \cite{herdegen2025proper}, we therefore use properness to select the economically relevant solution. At the boundary of the domain, we use the following continuous and extended-valued conventions:
	\[
	g_{EZ}(t,c,v)
	=
	\begin{cases}
		0, & c>0,\ v=0 \,\, or\,\, c=0,\ \gamma<1,\\[2pt] 
		-\infty, & c=0,\ \gamma>1.
	\end{cases}
	\]
	Thus, when \(\gamma>1\), a finite-valued utility process can be associated
	only with a consumption stream that is strictly positive
	\(dt\otimes d\mathbb P\)-almost everywhere.
	
	\begin{definition}[\bf Proper utility process]
		\label{def:proper_utility_theta_greater_one}
		A utility process $J^C$ associated with $C \in \mathscr{P}([0,\infty))$ is called proper if it satisfies \eqref{eq:EZ_utility eq} and, for every $t\ge0$,
		\begin{equation}
			\mathbb E_t\left[\int_t^\infty
			e^{-\delta\vartheta s}C_s^{1-\gamma}\,ds\right]>0
			\quad\Longrightarrow\quad
			(1-\gamma)J^C_t>0,	\qquad \mathbb P\text{-a.s.}
			\label{eq:proper_utility_theta_greater_one}
		\end{equation}
		When $\gamma>1$, we use the convention $0^{1-\gamma}=+\infty$.
	\end{definition}
	\noindent The deterministic discount factor in
	\eqref{eq:proper_utility_theta_greater_one} does not change the event on
	which the conditional expectation is positive, so this definition is
	equivalent to \cite[Definition~4.1]{herdegen2025proper}. Let $\mathcal A^\dagger(x,y)$ be the subset of $\mathcal A(x,y)$ for which $C^{\xi,\pi,l}$ is right-continuous and admits a unique proper utility process. This is a nonempty set by \Cref{lem:candidate_proper_theta_greater_one}.  Moreover, \cite[Theorem~4.9]{herdegen2025proper} gives a broad
	sufficient condition for membership in the class $\mathcal A^\dagger(x,y)$. Examples include the geometric-Brownian and factor-dependent
	consumption processes considered in \cite[Propositions~4.10 and 4.13]
	{herdegen2025proper}. \smallskip

	When \(\gamma<1\), a suitable solution \(h\) of the HJB equation can address the problem $\displaystyle
	\sup_{(\pi,l)\in\mathcal A^\dagger(x,y)} J_0^{\,\xi,\pi,l}$ and the supremum is attained by the feedback controls associated with \(h\), see \Cref{thm:verification_theta_greater_one}. 
	The situation is different when \(\gamma>1\). Multiplication by \(1-\gamma<0\) reverses the HJB inequality. Hence the comparison principle cannot be applied to show that the investor’s proper utility is smaller than the HJB candidate. We therefore restrict attention to the subclass defined in \eqref{eq:verification_class_theta_greater_one}, for which the required
	domination of the HJB subsolution by the proper utility process can be
	established.
	\smallskip

	\section{HJB Equation and construction of solution}\label{sec. HJB Equation and construction of solution}
	In this section, we construct bounded positive solutions of the HJB equation that satisfy the conditions of the verification theorem. The idea is to reduce the HJB equation to a semilinear ODE by a power transformation. Next, we construct sequences of solutions to this ODE on bounded domains separately for $E=\mathbb{R}$ and $E=\mathbb{R}_+$. The limits of these sequences are then verified to solve the original equation on $E$.
	
	\subsection{Derivation of the reduced HJB equation}

	Formally, we denote by
	\(V(t,x,y)\) the value function of the problem $	\displaystyle\sup_{(\pi,l)\in\mathcal A_{EZ}(x,y)}
	J_0^{\,\xi,\pi,l}$ when the starting time is \(t>0\).
	If the value function is sufficiently smooth, then the HJB equation is
	\begin{align}
		0
		=
		\p_t V(t,x,y)
		+
		\sup_{\pi\in \mathbb{R}^n,\,l>0}
		\Bigg\{
		&
		g_{EZ}\bigl(t,lx,V(t,x,y)\bigr)
		+
		x\left[r(y)+\pi^\top\mu(y)-l\right]\p_x V(t,x,y)\notag\\
		&+
		b(y)\p_y V(t,x,y) 
		+
		\frac12 x^2
		\pi^\top\Sigma(y)\pi
		\p_{xx} V(t,x,y)\notag\\
		&+
		x\,a(y)\,
		\pi^\top\sigma(y)\rho
		\p_{xy} V(t,x,y)
		+
		\frac12 a(y)^2\p_{yy} V(t,x,y)
		\Bigg\}.
		\label{eq:time_dependent_HJB_EZ}
	\end{align}
	We look for a value function of the form
	\begin{equation}
		V(t,x,y)
		=
		e^{-\delta\vartheta t}V(x,y).
		\label{eq:time_factorization}
	\end{equation}
	
	For $l>0$, we substitute \eqref{eq:time_factorization} into the aggregator and then use \(1+\vartheta q=\vartheta\) to obtain
	\[
	\begin{aligned}
		g_{EZ}\bigl(t,lx,V(t,x,y)\bigr)
		=
		\beta e^{-\delta t}
		\frac{(lx)^{1-1/\psi}}{1-1/\psi}
		\bigl[(1-\gamma)e^{-\delta\vartheta t}V(x,y)\bigr]^q  
		=
		\beta e^{-\delta\vartheta t}
		\frac{(lx)^{1-1/\psi}}{1-1/\psi}
		\bigl[(1-\gamma)V(x,y)\bigr]^q .
	\end{aligned}
	\]
	Using the above relation and dividing \eqref{eq:time_dependent_HJB_EZ} by \(e^{-\delta\vartheta t}\), we obtain
	the time-homogeneous equation
	\begin{align}
		0
		=
		\sup_{\pi\in \mathbb{R}^n,\,l>0}
		\Bigg\{
		&
		\overline{g}_{EZ}\bigl(lx,V(x,y)\bigr)
		+
		x\left[r(y)+\pi^\top\mu(y)-l\right]\p_x V(x,y)
		+
		b(y)\p_y V(x,y)
		\notag
		\\
		&
		+
		\frac12
		x^2
		\pi^\top\sigma(y)\sigma(y)^\top\pi
		\p_{xx} V(x,y)
		+
		x\,a(y)\,
		\pi^\top\sigma(y)\rho
		\p_{xy} V(x,y)
		+
		\frac12 a(y)^2\p_{yy} V(x,y)
		\Bigg\},
		\label{eq:HJB_EZ}
	\end{align}
	where the reduced aggregator is $\overline{g}_{EZ}(c,v)
	:=
	\beta
	\dfrac{c^{1-1/\psi}}{1-1/\psi}
	\bigl[(1-\gamma)v\bigr]^q
	-
	\delta\vartheta v.$\smallskip

	We then exploit the spatial homogeneity of the problem. For any
	$\lambda > 0$, the linearity of the wealth dynamics in \eqref{eq. wealth eq} gives
	$X^{(\lambda x,y),\pi,l} = \lambda X^{(x,y),\pi,l}$, and the aggregator
	satisfies
	$g_{EZ}(t,\lambda c, \lambda^{1-\gamma}v)
	= \lambda^{1-\gamma}g_{EZ}(t,c,v)$.
	It follows that $J^{\,(\lambda x,y),\pi,l} = \lambda^{1-\gamma}J^{\,(x,y),\pi,l}$,
	and hence the value function is homogeneous of degree $1-\gamma$ in wealth:
	$V(\lambda x,y)=\lambda^{1-\gamma}V(x,y)$.  This motivates the ansatz
	\begin{equation}
		V(x,y)
		=
		\frac{x^{1-\gamma}}{1-\gamma}\,h(y)^{\chi}, 
		\label{eq:homothetic_value_delta}
	\end{equation}
	where $h:E\to \mathbb{R}_+$ is some function and $\chi\neq 0$ is a constant to be determined.  The partial derivatives of $V$
	are
	\begin{align*}
		\p_x V    &= x^{-\gamma}h^{\chi},
		&
		\p_{xx} V &= -\gamma x^{-\gamma-1}h^{\chi},
		\notag \\
		\p_y V    &= \frac{x^{1-\gamma}}{1-\gamma}\,\chi h^{\chi-1}h',
		&
		\p_{xy} V &= x^{-\gamma}\chi h^{\chi-1}h',
		\notag \\
		\p_{yy} V &= \frac{x^{1-\gamma}}{1-\gamma}
		\Bigl[\chi h^{\chi-1}h''
		+\chi(\chi-1)h^{\chi-2}(h')^{2}\Bigr],
	\end{align*}
	where all functions are evaluated at $(x,y)$ or $y$ unless stated otherwise. Substituting~\eqref{eq:homothetic_value_delta} into the reduced aggregator
	and setting $c=lx$, we obtain
	\begin{align*}
		\overline{g}_{EZ}\bigl(lx,\,V(x,y)\bigr)
		=
		\beta\,\frac{(lx)^{1-1/\psi}}{1-1/\psi}\,
		\bigl(x^{1-\gamma}h^{\chi}\bigr)^{q}
		-
		\delta\vartheta\,\frac{x^{1-\gamma}}{1-\gamma}\,h^{\chi}
		=
		x^{1-\gamma}
		\left[
		\beta\,\frac{l^{1-1/\psi}}{1-1/\psi}\,h^{\chi q}
		-
		\frac{\delta\vartheta}{1-\gamma}\,h^{\chi}
		\right].
	\end{align*}
	Since $q =1-\dfrac{1}{\vartheta}$, we have
	$h^{\chi q}=h^{\chi}\,h^{-\frac{\chi}{1-\gamma}(1-1/\psi)}$. Assuming formally that $h>0$ (which will be shown later), the first-order condition with respect to $l$ yields the candidate for the optimal consumption--wealth ratio
	\begin{equation}
		l^{*}(y)
		=
		\beta^{\psi}\,
		h(y)^{-\frac{\chi(\psi-1)}{1-\gamma}}.
		\label{eq:l_star_chi}
	\end{equation}
	As $\Sigma(y)$ is invertible, the first-order condition with respect to $\pi$ yields the candidate of optimal portfolio
	\begin{equation}
		\pi^{*}(y)
		=
		\frac{1}{\gamma}\,\Sigma(y)^{-1}
		\!\left[
		\mu(y)
		+\chi\,\frac{h'(y)}{h(y)}\,\Upsilon(y)
		\right].
		\label{eq:pi_star_chi}
	\end{equation}
	Substituting $l^{*}$ and $\pi^{*}$ back into the HJB equation \eqref{eq:HJB_EZ} and
	multiplying through by $\frac{1-\gamma}{\chi h(y)^{\chi-1}}$, we arrive at
	\begin{align*}
		0
		={}&
		\frac{a(y)^{2}}{2}\,h''(y)
		+
		\left[
		b(y)
		-\Bigl(1-\tfrac{1}{\gamma}\Bigr)
		\mu(y)^{\top}\Sigma(y)^{-1}\Upsilon(y)
		\right]h'(y)
		\notag\\
		&+
		\left[
		\frac{a(y)^{2}}{2}(\chi-1)
		+
		\frac{\chi(1-\gamma)}{2\gamma}\,
		\Upsilon(y)^{\top}\Sigma(y)^{-1}\Upsilon(y)
		\right]
		\frac{(h'(y))^{2}}{h(y)}
		\notag\\
		&+
		\frac{1-\gamma}{\chi}
		\left[
		r(y)
		-\frac{\delta\psi}{\psi-1}
		+\frac{1}{2\gamma}\,\mu(y)^{\top}\Sigma(y)^{-1}\mu(y)
		\right]h(y)
		+
		\frac{\beta^{\psi}(1-\gamma)}{\chi(\psi-1)}\,
		h(y)^{1-\frac{\chi(\psi-1)}{1-\gamma}}.
	\end{align*}
	The coefficient of $\tfrac{(h')^{2}}{h}$ vanishes if and only if we choose
	\begin{equation}
		\chi
		:=
		\frac{\gamma}{
			\gamma+(1-\gamma)\,\rho^{\top}\rho
		}
		=
		\frac{\gamma}{
			\gamma(1-\rho^{\top}\rho)+\rho^{\top}\rho
		}
		>0.
		\label{eq:chi_constant}
	\end{equation}
	With this choice,
	the HJB equation reduces to
	\begin{align}
		0
		={}&
		\frac{a(y)^{2}}{2}\,h''(y)
		+\widetilde{b}(y)\,h'(y)
		-\kappa(y)\,h(y)
		+\eta
		h(y)^p,\quad\text{for $y \in E$,}
		\label{eq: reduced HJB}
	\end{align}
	where
	\begin{equation}\label{eq:p_eta_def}
		p
		:=
		1 - \frac{\chi(\psi-1)}{1-\gamma},
		\qquad
		\eta
		:=
		\frac{\beta^{\psi}(1-\gamma)}{\chi(\psi-1)},\qquad
		\widetilde{b}(y)\!
		:=\!
		b(y)
		\!-\!\Bigl(1-\tfrac{1}{\gamma}\Bigr)\,
		\mu(y)^{\top}\Sigma(y)^{-1}\Upsilon(y),
	\end{equation}
	$$\kappa(y)\!
	:=\!
	\frac{1-\gamma}{\chi(\psi-1)}
	\left[
	\delta\psi
	+(1-\psi)\left( r(y)
	+\frac{1}{2\gamma}\,\mu(y)^{\top}\Sigma(y)^{-1}\mu(y)\right) 
	\right]. $$
	Since $\beta > 0$, one can verify directly that
	\begin{equation}\label{eq:eta_p_relation}
		\eta(p-1) = -\beta^{\psi} < 0,
	\end{equation}
	therefore, we have $\eta > 0$ and $p < 1$ as $\vartheta > 0$. 
	\subsection{Dirichlet boundary problem for \texorpdfstring{$E = \mathbb{R}$}{E=R} }\label{sec. Dirichlet boundary problem}
	
	Throughout this section, we take $E = \mathbb{R}$ and approximate the solution of the reduced HJB equation \eqref{eq: reduced HJB} in a bounded domain with boundary conditions. For any $R>0$, we consider
	\begin{equation}\label{eq:BVP_h}
		\left\{
		\begin{aligned}
			&\frac{a(y)^2}{2}\,h''(y)
			+
			\widetilde{b}(y)\,h'(y)
			-
			\kappa(y)\,h(y)
			+\eta\,
			h(y)^p
			= 0,
			\qquad y \in (-R,R), \\[4pt]
			&h(\pm R)
			=
			\dfrac{1}{\kappa(\pm R)}.
		\end{aligned}
		\right.
	\end{equation}
	Next, we impose the following assumption:
	
	\begin{assumption}[\bf Positive myopic consumption]\label{ass:kappa}
		\sloppy	There is $\underline{\kappa} >0$ such that $
		\displaystyle\inf_{y\in E}\kappa(y) \geq \underline{\kappa} > 0$.
	\end{assumption}
	\begin{remark}
		This is a generalisation of the assumption in \cite[Corollary 8.2]{herdegen2023infinite_II}. This assumption also appears in \cite[Assumption 3.3]{guasoni2025variational} when the utility is CRRA. It is the key to showing the upper bound of the solution to the ODEs.
	\end{remark}

	\begin{lemma}[\bf Existence]\label{prop:existence}
		Suppose that Assumption~\ref{ass:kappa} holds and $R>0$. The boundary value problem
		\eqref{eq:BVP_h} admits at least one positive classical solution
		$h \in C^{2,\alpha}([-R,R])$. Moreover, any classical solution $h$ of \eqref{eq:BVP_h} satisfies $h>0$ on $[-R,R]$.
	\end{lemma}
	
	\begin{proof}
		Define the differential operator
		\begin{equation}\label{eq:nonlinear_operator}
			\mathcal{F}[u](y)
			:=
			\frac{a(y)^2}{2}\,u''(y)
			+
			\widetilde{b}(y)\,u'(y)
			-
			\kappa(y)\,u(y)
			+
			\eta\,u(y)^{p}.
		\end{equation}
		
		\noindent Set $\overline{\kappa}_R := \displaystyle\max_{y \in [-R,R]}\kappa(y)$. Choose $\varepsilon > 0$
		small enough that
		\begin{equation}\label{eq:epsilon_choice}
			\varepsilon
			\leq
			\min\!\left\{
			\frac{1}{\kappa(-R)},\;
			\frac{1}{\kappa(R)},\;
			\left(\frac{\eta}{\overline{\kappa}_R}\right)^{\!\frac{1}{1-p}}
			\right\},
		\end{equation}
		and let $\underline{h} \equiv \varepsilon$. The boundary inequality
		$\underline{h}(\pm R) \leq 1/\kappa(\pm R)$ holds by \eqref{eq:epsilon_choice}.
		Since $p < 1$, we have
		$\eta\varepsilon^{p-1} \geq \overline{\kappa}_R$, and therefore $
		\mathcal{F}[\h{1.3pt}\underline{h}\h{1.3pt}](y)
		=
		\bigl[-\kappa(y) + \eta\varepsilon^{p-1}\bigr]\varepsilon
		\geq
		(-\overline{\kappa}_R + \eta\varepsilon^{p-1})\varepsilon
		\geq 0.$ Choose $M > 0$ large enough that
		\begin{equation*} 
			M
			\geq
			\max\!\left\{
			\frac{1}{\kappa(-R)},\;
			\frac{1}{\kappa(R)},\;
			\left(\frac{\eta}{\underline{\kappa}}\right)^{\!\frac{1}{1-p}}
			\right\},
		\end{equation*}
		and let $\overline{h} \equiv M$. The boundary inequality
		$\overline{h}(\pm R) \geq 1/\kappa(\pm R)$ holds.
		Since $M \geq (\eta/\underline{\kappa})^{1/(1-p)}$, we have
		$\eta M^{p-1} \leq \underline{\kappa}$, and therefore $
		\mathcal{F}[\h{2pt}\overline{h}\h{2pt}](y)
		=
		\bigl[-\kappa(y) + \eta M^{p-1}\bigr]M
		\leq
		(-\underline{\kappa} + \eta M^{p-1})M
		\leq 0.$ Since $\varepsilon \leq M$, we have $\underline{h}\leq \overline{h}$.\smallskip

		The classical sub- and supersolution theorem \cite[Theorem~2.1]{SATTINGER} then yields a classical solution $h$ such that
		\[
		0<\underline h\leq h(y)\leq \overline h,
		\qquad \text{ for } y\in[-R,R].
		\]

		Suppose $h$ is any classical solution. The strong maximum principle \cite[Theorem 3.5]{GT77} then implies that
		$h$ cannot attain an interior minimum of zero unless $h \equiv 0$. Since
		$h(\pm R) = 1/\kappa(\pm R) > 0$, the solution cannot be identically zero, and
		hence $h > 0$ on $[-R,R]$. The positivity of $h$ and the Hölder continuity of the coefficients $a^2$,
		$\widetilde{b}$, and $\kappa$ from Assumption~\ref{ass.market well-posed} ensure that the nonlinear term $\eta h^p$ is
		Hölder continuous. Standard Schauder estimates \cite{GT77} then yield $h \in C^{2,\alpha}([-R,R])$. 
	\end{proof}

	\begin{lemma}[\bf Comparison principle and uniqueness]\label{prop:uniqueness and max p} Recall $\mathcal F$ defined in \eqref{eq:nonlinear_operator}. Suppose that Assumption~\ref{ass:kappa} holds and $R>0$. Let \(u,v\in C^2((-R,R))\cap C([-R,R])\) be functions such that
		\[u,\,v>0\quad\text{in }[-R,R],\quad
		\mathcal{F}[u]\geq 0,\,\,
		\mathcal F[v]\leq 0
		\quad\text{in }(-R,R),\quad \text{ and \(\quad u(y)\leq v(y)\) at \(y=\pm R\)}.
		\] 
		Then $
		u\leq v$ in $[-R,R]$. In particular, \eqref{eq:BVP_h} admits a unique classical solution and it is positive.
	\end{lemma}

	\begin{proof}
		Suppose, to the contrary, that $
		\displaystyle\max_{[-R,R]}\frac{u}{v}>1.$
		Since \(u\leq v\) at \(y=\pm R\), the maximum of \(u/v\) is attained at some
		interior point \(y_0\in(-R,R)\). Set
		\[
		M:=\frac{u(y_0)}{v(y_0)}>1.
		\]
		At \(y_0\), we have $u(y_0)=Mv(y_0)$, $u'(y_0)=Mv'(y_0)$, and $u''(y_0)\leq Mv''(y_0)$, where the last inequality follows from \((u/v)''(y_0)\leq0\) and
		\(v(y_0)>0\). Using \( \mathcal{F}[u](y_0)\geq0\) and
		\(\frac{a(y_0)^2}{2}\geq0\), we obtain
		\begin{align*}
			0
			\leq
			\mathcal{F}[u](y_0)
			\leq
			M\left[
			\frac{a(y_0)^2}{2}v''(y_0)
			+\widetilde b(y_0)v'(y_0)
			-\kappa(y_0)v(y_0)
			\right]
			+\eta M^p v(y_0)^p .
		\end{align*}
		Since \(\mathcal F[v](y_0)\leq0\), we have 
		\[
		0
		\leq
		\eta v(y_0)^p(M^p-M).
		\]
		But  \(M>1\) and \(\eta(p-1)<0\) from \eqref{eq:eta_p_relation}. Hence $
		\eta(M^p-M)<0$, which is a contradiction since \(v(y_0)>0\). Thus \(u\leq v\) on \([-R,R]\).
	\end{proof}

	To derive uniform (in $R$) bounds for solutions of \eqref{eq:BVP_h}, we need to impose the following assumption:
	
	\begin{assumption}[\bf Polynomial growth of coefficients]\label{ass:polynomial_growth}
		There exist constants $C_* > 0$ and
		$\alpha_* > 0$ such that $
		|a(y)| + |\widetilde{b}(y)| + \kappa(y)
		\leq C_*(1 + |y|^{\alpha_*})$ for all $y \in \mathbb{R}$.  
	\end{assumption}

	\begin{lemma}[\bf Bounds of $h_{R}$]\label{lem:barrier_bounds}
		Suppose that Assumptions \ref{ass:kappa} and \ref{ass:polynomial_growth} hold. Then the classical solution $h_{R}$ of
		\eqref{eq:BVP_h} satisfies
		\begin{equation}\label{eq:barrier_bounds}
			\underline{h}(y) \leq h_{R}(y) \leq \overline{h}(y),
			\qquad \text{for $y \in [-R,R]$ and $R>1$},
		\end{equation}
		where
		\begin{equation}\label{eq:upper_barrier_def}
			\overline{h}(y) =M
			:=
			\max\!\left\{
			\frac{1}{\underline{\kappa}},\,
			\left(\frac{\eta}{\underline{\kappa}}\right)^{\!\frac{1}{1-p}}
			\right\},\qquad
			\underline{h}(y)
			:=
			c_{0}\,(1 + y^{2})^{-\alpha^\dagger/2}.
		\end{equation} 
		Here $\alpha^\dagger > \max\bigl\{\alpha_*,\, 2\alpha_*/(1-p)\bigr\}$ and $c_{0} > 0$ is
		a sufficiently small constant specified in \eqref{eq:c0_choice}.
	\end{lemma}

	
	\begin{proof}
		The choice \eqref{eq:upper_barrier_def} gives $\eta M^{p-1} \leq \underline{\kappa}$, so 
		\begin{equation*}
			\mathcal{F}[\h{2pt}\overline{h}\h{2pt}](y)
			=
			\bigl[-\kappa(y) + \eta M^{p-1}\bigr]M
			\leq
			(-\underline{\kappa} + \eta M^{p-1})M
			\leq 0.
		\end{equation*}
		Moreover, $\overline{h}(\pm R)=M \geq 1/\underline{\kappa} \geq 1/\kappa(\pm R)$.\smallskip
		
		Let $\phi(y) := (1 + y^{2})^{-\alpha^\dagger/2}$ with $\alpha^\dagger > \max\{\alpha_*, 2\alpha_*/(1-p)\}$.
		A direct computation using \Cref{ass:polynomial_growth} shows that there
		exists a constant $\C{1}=\C{1}(C_*,\alpha_*,\alpha^\dagger)> 0$ such that
		\begin{equation}\label{eq:phi_linear_bound}
			\left|\frac{a(y)^{2}}{2}\,\phi''(y) + \widetilde{b}(y)\,\phi'(y)
			- \kappa(y)\,\phi(y)\right|
			\leq
			\C{1}\,\phi(y)\,(1 + |y|)^{2\alpha_*}.
		\end{equation}
		Since $\alpha^\dagger(1-p)/2 > \alpha_*$, we also have
		\begin{equation}\label{eq:phi_power_bound}
			\phi(y)^{p-1}
			=
			(1 + y^{2})^{\alpha^\dagger(1-p)/2}
			\geq
			c_1\,(1 + |y|)^{2\alpha_*}
		\end{equation}
		for some constant $c_1=c_1(\alpha^\dagger,\alpha_*) > 0$. Choose
		\begin{equation}\label{eq:c0_choice}
			c_{0}
			=
			\min\!\left\{
			1,\;
			\left(\frac{\eta\,c_1}{\C{1}}\right)^{\!\frac{1}{1-p}}, \; \left[ C_*\sup_{r\geq0} (1+r^2)^{-\alpha^\dagger/2}(1+r^{\alpha_*}) \right]^{-1}
			\right\},
		\end{equation}
		so that $\eta \,c_{0}^{\h{1pt} \h{1pt}p-1} c_1 \geq \C{1}$, and set
		$\underline{h} := c_{0}\phi$. Using \eqref{eq:phi_linear_bound} and then \eqref{eq:phi_power_bound}, together with $\eta c_{0}^{\h{1pt} p-1}c_1 \geq \C{1}$, we obtain
		\begin{align*}
			\mathcal{F}[\h{1.3pt}\underline{h}\h{1.3pt}](y)
			&=
			c_{0}\!\left[
			\frac{a(y)^{2}}{2}\,\phi''(y) + \widetilde{b}(y)\,\phi'(y)
			- \kappa(y)\,\phi(y)
			\right]
			+ \eta\,c_{0}^{\h{1pt} \h{1pt}p}\,\phi(y)^{\h{1pt} p}
			\notag\\ 
			&\geq
			c_{0}\,\phi(y)
			\left[
			-\C{1}(1 + |y|)^{2\alpha_*}
			+ \eta\,c_{0}^{\h{1pt} p-1}\,c_1(1 + |y|)^{2\alpha_*}
			\right]\notag\\
			&\geq 0.
		\end{align*}
		Since $\phi$ decays as $|y|^{-\alpha^\dagger}$ and $\kappa(y) \leq C_*(1 + |y|^{\alpha_*})$
		with $\alpha^\dagger > \alpha_*$, we obtain
		\begin{equation*}
			\underline{h}(\pm R)
			=
			c_{0}(1 + R^{2})^{-\alpha^\dagger/2}
			\leq
			\frac{1}{\kappa(\pm R)},
			\qquad \text{for } R > 0.
		\end{equation*} 
		
		By \Cref{prop:uniqueness and max p}, applied with $u=\underline h$ and $v=h_R$, we obtain \(\underline h\leq h_R\) on \([-R,R]\). Similarly, applying \Cref{prop:uniqueness and max p} with $
		u=h_R$ and $v=\overline h$, we obtain \(h_R\leq\overline h\) on \([-R,R]\).
	\end{proof}

	\begin{lemma}[\bf Passage to the whole line]\label{prop:pass_R_to_infty E=R} Suppose that Assumptions \ref{ass:kappa} and \ref{ass:polynomial_growth} hold. There exist a sequence \(R_n\to\infty\) and a positive function $
		h_*\in C^{2,\alpha}(\mathbb R)$ such that $
		h_{R_n}\to h_*$ in $C^2(\mathbb R)$ locally and \(h_*\) solves \eqref{eq: reduced HJB}. Moreover, it holds that
		$$
		c_{0}\,(1 + y^{2})^{-\alpha^\dagger/2} \leq h_*(y) \leq \max\!\left\{
		\frac{1}{\underline{\kappa}},\,
		\left(\frac{\eta}{\underline{\kappa}}\right)^{\!\frac{1}{1-p}}
		\right\},
		\qquad \text{for } y \in \mathbb{R},
		$$
		where $c_0$ and $\alpha^\dagger$ are given in \Cref{lem:barrier_bounds}.
	\end{lemma}
	
	\begin{proof}
		Let \(\mathcal K'\subset\subset \mathcal K\subset\subset\mathbb R\) be compact sets and take \(R>0\) sufficiently large such that \(\mathcal K\subset\subset(-R,R)\). By \Cref{lem:barrier_bounds}, there exists
		constant \(m_{\mathcal K}>0\), independent of \(R\), such that
		\[
		0<m_{\mathcal K}\leq h_R(y)\leq M,
		\qquad \text{for } y\in \mathcal K .
		\]
		Hence \(h_R^p\) is uniformly (in $R$) bounded on $\mathcal K$. Bootstrapping arguments and interior Schauder estimates \cite[Theorem 6.6]{GT77} imply that $\|h_R\|_{C^{2,\alpha}(\mathcal K')}
		\leq C_{\mathcal K',\mathcal K}$, where \(C_{\mathcal K',\mathcal K}\) is independent of \(R\). Taking \(R\to\infty\) along some sequence $\{R_n\}_{n\in \mathbb{N}}$ and using a
		diagonal compactness argument, we obtain a subsequence, still denoted by
		\(h_{R_n}\), and a function \(h_*\in C^{2,\alpha}(\mathbb R)\)
		such that $h_{R_n}\to h_*$ in $C^2(\mathbb R)$ locally. Passing to the
		limit in the equation gives the equation on $\mathbb{R}$. The bounds in \eqref{eq:barrier_bounds} also pass to the limit, so the proof is completed.
	\end{proof}

	\subsection{Neumann-Dirichlet boundary problem for \texorpdfstring{$E = \mathbb{R}_+$}{E=R+}}\label{sec. Robin-Dirichlet boundary problem}
	
	Throughout this section, we take \(E=\mathbb R_+\). This case differs from \(E=\mathbb R\) because the left endpoint \(0\) is a genuine boundary of the state space. Therefore, imposing a Dirichlet condition near \(0\) would add an artificial boundary value not present in the original half-line problem. We impose a Dirichlet condition only at the artificial right boundary \(R\). Near the finite boundary \(0\), we impose a homogeneous Neumann condition which does not prescribe an
	artificial value of \(h\). It is compatible with the uniform lower and upper
	barriers, and disappears from the interior limiting equation as
	\(\varepsilon\downarrow0\).\smallskip 
	
	We impose the following condition on the behaviour of coefficients near $y = 0$ and infinity:
	\begin{assumption}[\bf Growth of coefficients]\label{ass poly growth R_+}
		There exist constants $C_*>0$ and $\alpha_1,\alpha_2\ge 0$ such that $
		|a(y)|+|\widetilde b(y)|+|\kappa(y)|
		\le C_*\pigl(1+\max\{y^{-\alpha_1},y^{\alpha_2}\}\pigr)$ for all $y>0$.
	\end{assumption}
	For $0<\varepsilon< R$, we study the HJB equation on a bounded interval with mixed Neumann--Dirichlet boundary conditions:

	\begin{align}
		\begin{cases}
			&\dfrac{a(y)^2}{2}h''(y)+\widetilde b(y)h'(y)-\kappa(y)h(y)+\eta h(y)^p=0,
			\qquad y\in(\varepsilon,R),\\[0.8em]
			& h'(\varepsilon) =0,\qquad
			h(R)=\dfrac1{\kappa(R)}.
		\end{cases}
		\label{eq:eps_BVP}
	\end{align}

	\begin{lemma}[\bf Existence]
		\label{prop:existence on bdd domain when E=R_+}
		Suppose that Assumptions \ref{ass:kappa} and \ref{ass poly growth R_+} hold. Then there exist a constant $\varepsilon_0\in(0,1)$ such that, for every $R>1$ and every $\varepsilon \in (0,\varepsilon_0 )$, the problem \eqref{eq:eps_BVP} admits a positive classical solution $h_{\varepsilon,R}\in C^{2,\alpha}((\varepsilon,R))\cap C^1([\varepsilon,R]).$ Moreover,
		\begin{align}
			\delta_0\exp(-y-y^{-1})\le h_{\varepsilon,R}(y)\le \overline h	:=\max\left\{\frac1{\underline\kappa},
			\left(\frac{\eta}{\underline\kappa}\right)^{1/(1-p)}\right\},
			\qquad \text{for } y\in[\varepsilon,R],
			\label{2881}
		\end{align}
		where $\delta_0>0$ is independent of $R$ and $\varepsilon$. Moreover, any classical solution of \eqref{eq:eps_BVP} is positive on $[\varepsilon,R]$.
	\end{lemma}
	
	\begin{proof}  	We first construct a subsolution. Set
		\[
		\phi(y):=\delta_0 \exp(-y-y^{-1}),\qquad\text{for } y>0.
		\]
		Write $\ell(y):=\log\phi(y)=\log\delta_0-y-y^{-1}$. Then
		\[
		\ell'(y)=-1+y^{-2},
		\qquad
		\ell''(y)=-2y^{-3},\qquad
		\phi'=\phi\ell',
		\qquad
		\phi''=\phi\bigl[(\ell')^2+\ell''\bigr].
		\]
		Hence
		\begin{align}
			\mathcal{F}[\phi](y)
			=
			\phi(y)
			\left\{
			\frac{a(y)^2}{2}\bigl[(\ell'(y))^2+\ell''(y)\bigr]
			+\widetilde b(y)\ell'(y)-\kappa(y)
			+\eta\phi(y)^{p-1}
			\right\}.
			\label{2908}
		\end{align}
		By Assumption \ref{ass poly growth R_+}, there exist constants $C_2>0$ and $N_1,N_2>0$ such that
		\begin{align}
			\left|
			\frac{a(y)^2}{2}\bigl[(\ell'(y))^2+\ell''(y)\bigr]
			+\widetilde b(y)\ell'(y)-\kappa(y)
			\right|
			\le
			C_2\left(1+y^{-N_1}+y^{N_2}\right),
			\qquad \text{for } y>0.
			\label{2919}
		\end{align}
		Since \(p<1\), the exponential factor in $\phi(y)^{p-1}$ in $\mathcal{F}[\phi]$ dominates both as \(y\to 0^+\) and
		as \(y\to\infty\). Thus there exists \(c_2>0\) such that
		\[
		\phi(y)^{p-1}
		=
		\delta_0^{p-1}\exp\bigl[(1-p)(y+y^{-1})\bigr]
		\ge
		c_2\delta_0^{p-1}\left(1+y^{-N_1}+y^{N_2}\right),
		\qquad \text{for } y>0.
		\]
		Since \(p-1<0\), we can always choose $\delta_0>0$ small enough that $
		\eta \delta_0^{p-1}c_2\ge C_2.$ Then $
		\mathcal{F}[\phi](y)\ge0$ for $y>0$, by \eqref{2908} and \eqref{2919}. \smallskip
		
		We next check the boundary inequalities. At the right boundary, Assumptions \ref{ass:kappa} and \ref{ass poly growth R_+} show that we can choose $\delta_0$ small enough, independent of $R>1$, such that
		\[
		\phi(R)=\delta_0 e^{-R-R^{-1}}
		\le \delta_0 e^{-R}
		\le \frac1{C_*(1+R^{\alpha_2})}
		\le \frac1{\kappa(R)},\qquad \text{for } R>1.
		\]
		At the left boundary, we have $\phi'(\varepsilon)=\delta_0(-1+\varepsilon^{-2}) \exp(-\varepsilon-\varepsilon^{-1})>0$ for $0<\varepsilon<\varepsilon_0$, by taking
		$\varepsilon_0\in(0,1)$ smaller. \smallskip
		
		Now set
		\[
		\overline h
		=
		\max\left\{
		\frac1{\underline\kappa},
		\left(\frac{\eta}{\underline\kappa}\right)^{1/(1-p)}
		\right\}.
		\]
		Then $ \phi\le\delta_0\leq \overline h$ after decreasing $\delta_0$ once more. Moreover,
		\[
		\mathcal{F}[\,\overline h\,](y)
		=
		-\kappa(y)\overline h+\eta\overline h^p
		=
		\overline h\bigl(-\kappa(y)+\eta\overline h^{p-1}\bigr)
		\le
		\overline h\bigl(-\underline\kappa+\eta\overline h^{p-1}\bigr)
		\le0.
		\]
		Also, $
		\overline h\ge \frac1{\underline\kappa}\ge\frac1{\kappa(R)} $ and $\overline h'=0$. Thus $\phi$ and $\overline h$ are, respectively, ordered sub- and supersolutions of \eqref{eq:eps_BVP}.\smallskip
		
		On the compact interval $[\varepsilon,R]$ the equation \eqref{eq:eps_BVP} is uniformly elliptic, because
		$a(y)^2>0$ and $a^2$ is continuous there. By the upper--lower solution theorem \cite[Theorem 6.3]{SCHMITT1978263} or \cite[Theorem~1]{MawhinSchmitt1984}, there exists a classical solution $h_{\varepsilon,R}$ satisfying \eqref{2881}.\smallskip
		
		Any classical solution $h_{\varepsilon,R}$ is positive on $[\varepsilon,R]$ by the maximum principle and Hopf's lemma when $p\in[0,1)$. If \(p<0\), the finiteness of \(h_{\varepsilon,R}^p\) implies that
		\(h_{\varepsilon,R}\) cannot vanish on \((\varepsilon,R]\). If \(h(\varepsilon)=0\), then \(h\) attains its minimum at the left endpoint. Hopf's lemma gives $
		h'(\varepsilon)>0$ since $h$ is not constant, contradicting the Neumann condition \(h'(\varepsilon)=0\). Hence
		\(h>0\) on \([\varepsilon,R]\). 
	\end{proof}
	
	\begin{lemma}[\bf Comparison principle and uniqueness]
		\label{prop:comparison_uniqueness R_+}
		Suppose that Assumptions~\ref{ass:kappa} and~\ref{ass poly growth R_+} hold. Recall $\mathcal F$ defined in \eqref{eq:nonlinear_operator}.
		Let $u,v\in C^2((\varepsilon,R))\cap C^1([\varepsilon,R])$ be functions such that
		\begin{equation}
			\label{eq:subsolution_ineq}
			u,\,v>0\quad\text{in }[\varepsilon,R],\quad
			\mathcal{F}[u]\geq 0,\,\,
			\mathcal F[v]\leq 0 \quad\text{in }(\varepsilon,R)
		\end{equation}
		with the boundary inequalities
		\begin{equation}
			\label{eq:sub_super_boundary_ineq}
			u(R)\le v(R),\,\, 
			u'(\varepsilon)
			\ge 0,\,\, 
			v'(\varepsilon)
			\le 0.
		\end{equation}
		Then $u\le v$ on $[\varepsilon,R]$. Consequently, the boundary value problem~\eqref{eq:eps_BVP} admits a unique classical solution in $C^2((\varepsilon,R))\cap C^1([\varepsilon,R])$ and it is positive.
	\end{lemma}
	\begin{proof} Set 
		\[ \mathbbm{u}:=\log u, \qquad \mathbbm{v}:=\log v, \qquad w:=\mathbbm{u}-\mathbbm{v} . \] 
		Then $\mathbbm{u},\mathbbm{v},w\in C^2((\varepsilon,R))\cap C^1([\varepsilon,R])$. Inequalities \eqref{eq:subsolution_ineq} imply 
		\begin{equation} \label{eq:U_sub_ineq} \frac{a(y)^2}{2}\bigl[\mathbbm{u}''(y)+(\mathbbm{u}'(y))^2\bigr] +\widetilde b(y)\mathbbm{u}'(y) -\kappa(y) +\eta e^{(p-1)\mathbbm{u}(y)} \ge 0, 
		\end{equation} 
		and 
		\begin{equation} \label{eq:V_super_ineq} \frac{a(y)^2}{2}\bigl[\mathbbm{v}''(y)+(\mathbbm{v}'(y))^2\bigr] +\widetilde b(y)\mathbbm{v}'(y) -\kappa(y) +\eta e^{(p-1)\mathbbm{v}(y)} \le 0. 
		\end{equation} 
		Subtracting \eqref{eq:V_super_ineq} from \eqref{eq:U_sub_ineq} gives 
		\begin{equation} \label{eq:w_comparison_ineq} \frac{a(y)^2}{2}w''(y) + \left\{ \widetilde b(y) +\frac{a(y)^2}{2}\bigl[\mathbbm{u}'(y)+\mathbbm{v}'(y)\bigr] \right\}w'(y) 
			+ \left[ \eta(p-1) 
			\int_0^1
			e^{(p-1)[\theta\mathbbm{u}(y)+(1-\theta)\mathbbm{v}(y)]}d\theta\right] w(y)\ge 0 \end{equation} for $y\in(\varepsilon,R)$. We claim that $w\le0$ on $[\varepsilon,R]$. \smallskip
		
		Suppose, to the contrary, that $ \max_{[\varepsilon,R]} w>0.$ By \eqref{eq:sub_super_boundary_ineq}, the positive maximum is attained at some point $y_0\in[\varepsilon,R)$. If $y_0\in(\varepsilon,R)$, then $w(y_0)>0$, $w'(y_0)=0$, and $w''(y_0)\le0$. Thus, evaluating \eqref{eq:w_comparison_ineq} at $y_0$ gives 
		\[ 0 \le \frac{a(y_0)^2}{2}w''(y_0) + \left[ \eta(p-1) 
		\int_0^1
		e^{(p-1)[\theta\mathbbm{u}(y_0)+(1-\theta)\mathbbm{v}(y_0)]}d\theta\right] w(y_0) <0, \] 
		which is impossible. Next, we consider the case where $y_0=\varepsilon$. Since \(w\) is nonconstant, Hopf's lemma gives $w'(\varepsilon)<0.$ This is a contradiction. Hence $
		\max_{[\varepsilon,R]}w\leq0.$

		
		Finally, if $h_1$ and $h_2$ are two classical solutions of \eqref{eq:eps_BVP}, applying the comparison result first with $u=h_1$, $v=h_2$, and then with $u=h_2$, $v=h_1$, yields the uniqueness. 
	\end{proof}

	\begin{lemma}[\bf Passage to the limit]\label{prop:pass_R_to_infty}
		Suppose that Assumptions \ref{ass:kappa} and \ref{ass poly growth R_+} hold. Then there exist a sequence
		$\{R_n\}_{n\in \mathbb{N}}$ tending to infinity and $
		h_\star\in C^{2,\alpha} ((0,\infty))$
		such that $h_\star$ solves the reduced HJB equation \eqref{eq: reduced HJB} and the unique solution $ h_{1/R_n,R_n}$ of \eqref{eq:eps_BVP} satisfies
		\[h_{1/R_n,R_n}\to h_\star
		\quad\text{locally in }C^2((0,\infty)).
		\]
		Moreover it satisfies
		\[
		\delta_0 e^{-y-y^{-1}}
		\le h_\star(y)
		\le
		\overline h
		:=
		\max\left\{
		\frac{1}{\underline\kappa},
		\left(\frac{\eta}{\underline\kappa}\right)^{1/(1-p)}
		\right\},
		\qquad \text{for     } y>0.
		\]
	\end{lemma}
	
	\begin{proof}  Let $\mathcal K'\subset\subset \mathcal K \subset\subset (0,\infty)$ be compact sets. Let $R>1/\varepsilon_0>1$ be such that $\mathcal K'\subset\subset \mathcal K \subset\subset (1/R,R)$. Since $\phi(y)=\delta_0\exp(-y-y^{-1})$ is strictly positive on $\mathcal K$, \Cref{prop:existence on bdd domain when E=R_+} implies that there exists $m_{\mathcal K}>0$ such that
		\[
		0<m_{\mathcal K}\le h_{1/R,R}(y)\le\overline h,\qquad \text{for     } y\in \mathcal K.
		\]
		Therefore the nonlinear term \(h_{1/R,R}^p\) is uniformly (in $R$) bounded on \(\mathcal K\). We proceed similarly as in \Cref{prop:pass_R_to_infty E=R} to prove the lemma.
		\nihil{ 
			Moreover, on \(\mathcal K\), the coefficient \(a^2\) is positive and all coefficients are
			locally Hölder continuous. Bootstrapping arguments and interior Schauder estimates then give 
			\[
			\|h_{1/R,R}\|_{C^{2,\alpha}(\mathcal K')}
			\le C_{\mathcal K',\mathcal K},
			\]
			where \(C_{\mathcal K',\mathcal K}\) is independent of \(R\). A diagonal compactness argument gives a sequence \(R_n\to\infty\) and a function
			\(h_\star\in C^{2,\alpha} ((0,\infty))\) such that
			\[
			h_{R_n}\to h_\star
			\quad\text{locally in }C^2((0,\infty)).
			\]
			Passing to the limit in the equation and the bounds, we prove the lemma.}
	\end{proof}

	\section{Verification theorem}\label{sec. Verification theorem}
	
	In this section, we verify that the optimal strategies and value function can be recovered from the solution of the reduced HJB equation \eqref{eq: reduced HJB}, under some mild assumptions. Then we discuss how these assumptions can be checked using the properties of the solutions $h$ obtained in Lemmas \ref{prop:pass_R_to_infty E=R} and \ref{prop:pass_R_to_infty}.
	
	\subsection{The regime \texorpdfstring{$\vartheta \in (0,1)$}{theta in (0,1)}}
	
	\nihil{
		\begin{lemma}[\bf Uniqueness in the bounded positive class]
			\label{lem:whole_line_unique_bounded}
			Suppose Assumptions~\ref{assum. eta is a density} and~\ref{assu. Qualitative Verification}
			hold. Then the reduced HJB equation \eqref{eq: reduced HJB} has at most one classical solution $h\in C^2(E)$ such that $h>0$ on $E$ and it is uniform bounded on $E$.
		\end{lemma}
		
		\begin{proof}
			Let $h_1,h_2\in C^2(E)$ both satisfy \eqref{eq: reduced HJB} with $0<h_i<M$ on $E$.
			By symmetry it suffices to show $h_1\leq h_2$; the same argument with the two
			solutions interchanged then yields $h_2\leq h_1$, and hence $h_1\equiv h_2$.
			
			\smallskip
			\noindent\textit{Step 1 (Divergence form).}
			From the explicit formula~\eqref{eq. 1-d eta}, a direct computation gives
			$(a^2\eta)'=2\widetilde{b}\,\eta$. Multiplying \eqref{eq: reduced HJB} through by
			$\eta$ and using this identity, each solution satisfies
			\begin{equation}\label{eq:div_form}
				\tfrac{1}{2}\bigl(a^2\eta\,h_i'\bigr)'
				+\eta\bigl(-\kappa\,h_i+\eta_0\,h_i^p\bigr)=0,
				\qquad i=1,2.
			\end{equation}
			
			\smallskip
			\noindent\textit{Step 2 (Equation for the ratio).}
			Set $r:=h_1/h_2$. Multiplying \eqref{eq:div_form} for $h_1$ by $h_2$ and
			\eqref{eq:div_form} for $h_2$ by $h_1$, then subtracting, the $\kappa$-terms
			cancel. The identity
			\[
			h_2\bigl(a^2\eta\,h_1'\bigr)'-h_1\bigl(a^2\eta\,h_2'\bigr)'
			=\bigl(a^2\eta\,h_2^2\,r'\bigr)'
			\]
			then delivers
			\begin{equation}\label{eq:r_eqn}
				(D\,r')'+N=0,
			\end{equation}
			where
			\[
			D(y):=\tfrac{1}{2}a(y)^2\,\eta(y)\,h_2(y)^2>0,
			\qquad
			N(y):=\eta_0\,\eta(y)\,h_1(y)\,h_2(y)
			\bigl(h_1(y)^{p-1}-h_2(y)^{p-1}\bigr).
			\]
			Since $\eta_0(p-1)<0$, the function $h\mapsto\eta_0\,h^{p-1}$ is strictly
			decreasing on $(0,\infty)$, so
			\begin{equation}\label{eq:sign_N}
				N<0\qquad\text{on }\{h_1>h_2\}=\{r>1\}.
			\end{equation}
			
			\smallskip
			\noindent\textit{Step 3 (Cutoff functions via recurrence).}
			Write $s$ for the scale function from Assumption~\ref{assu. Qualitative Verification}
			(there denoted $p$; we use $s$ here to avoid a clash with the exponent above).
			Recurrence gives $s(E)=\mathbb{R}$. For each $n\geq1$, choose
			$\theta_n\in C_c^1(\mathbb{R})$ with $0\leq\theta_n\leq1$,
			$\theta_n\equiv1$ on $[-n,n]$, $\operatorname{supp}\theta_n\subseteq[-2n,2n]$,
			and $|\theta_n'|\leq C/n$, and set $\zeta_n:=\theta_n\circ s$. Then $\zeta_n$
			is compactly supported in $E$, with $\zeta_n\to1$ locally on $E$.
			
			A direct calculation from the definitions of $\eta$ and $s$ gives
			$a^2\eta\,s'=C_0$ for some constant $C_0>0$. Consequently, the change of
			variables $z=s(y)$ yields
			\begin{equation}\label{eq:zeta_to_zero}
				\int_E a^2\eta\,|\zeta_n'|^2\,dy
				\;=\;
				C_0\int_{\mathbb{R}}|\theta_n'(z)|^2\,dz
				\;\longrightarrow\; 0.
			\end{equation}
			
			\smallskip
			\noindent\textit{Step 4 (Energy estimate).}
			Set $\varphi:=(r-1)^+$. Testing \eqref{eq:r_eqn} against $\zeta_n^2\varphi$
			and integrating by parts (the compact support of $\zeta_n$ eliminates boundary
			terms),
			\begin{equation}\label{eq:ibp}
				\int_E\zeta_n^2\,D\,|\varphi'|^2\,dy
				\;=\;
				-2\int_E\zeta_n\,D\,\varphi\,\varphi'\,\zeta_n'\,dy
				+\underbrace{\int_E\zeta_n^2\varphi\,N\,dy}_{\leq\,0},
			\end{equation}
			where the last term is non-positive by~\eqref{eq:sign_N}. Dropping it,
			\[
			\int_E\zeta_n^2\,D\,|\varphi'|^2\,dy
			\;\leq\;
			2\Bigl|\int_E\zeta_n\,D\,\varphi\,\varphi'\,\zeta_n'\,dy\Bigr|.
			\]
			Young's inequality then gives
			\[
			\int_E\zeta_n^2\,D\,|\varphi'|^2\,dy
			\;\leq\;
			\frac{1}{2}\int_E\zeta_n^2\,D\,|\varphi'|^2\,dy
			+2\int_E D\,\varphi^2\,|\zeta_n'|^2\,dy,
			\]
			and absorbing the first term on the right,
			\[
			\int_E\zeta_n^2\,D\,|\varphi'|^2\,dy
			\;\leq\;
			4\int_E D\,\varphi^2\,|\zeta_n'|^2\,dy.
			\]
			Since
			\[
			D\varphi^2
			=\tfrac{1}{2}a^2\eta\,h_2^2\Bigl(\tfrac{h_1}{h_2}-1\Bigr)_+^2
			=\tfrac{1}{2}a^2\eta\,(h_1-h_2)_+^2
			\leq\tfrac{1}{2}M^2\,a^2\eta,
			\]
			applying~\eqref{eq:zeta_to_zero} gives
			\[
			\int_E\zeta_n^2\,D\,|\varphi'|^2\,dy
			\;\leq\;
			2M^2\int_E a^2\eta\,|\zeta_n'|^2\,dy
			\;\longrightarrow\; 0.
			\]
			Letting $n\to\infty$ and using $\zeta_n\to1$ locally, we conclude that
			$D|\varphi'|^2=0$ on every compact subset of $E$. Since $D>0$, this forces
			$\varphi'\equiv0$, so $\varphi$ is constant on $E$.
			
			\smallskip
			\noindent\textit{Step 5 (Ruling out $\varphi\equiv c_0>0$).}
			If $\varphi\equiv0$, then $r\leq1$, i.e., $h_1\leq h_2$, and we are done.
			Suppose for contradiction that $\varphi\equiv c_0>0$, so that
			$h_1=\lambda h_2$ on $E$ for some $\lambda:=1+c_0>1$. Substituting into the
			equation for $h_1$ and invoking the equation satisfied by $h_2$,
			\[
			0
			=\frac{a^2}{2}(\lambda h_2)''+\widetilde{b}(\lambda h_2)'
			-\kappa(\lambda h_2)+\eta_0(\lambda h_2)^p
			=\lambda\underbrace{\Bigl[\frac{a^2}{2}h_2''+\widetilde{b}h_2'
				-\kappa h_2\Bigr]}_{=-\,\eta_0 h_2^p}
			+\eta_0\lambda^p h_2^p
			=\eta_0\,h_2^p\,(\lambda^p-\lambda).
			\]
			Since $h_2>0$, this requires $\eta_0\,\lambda(\lambda^{p-1}-1)=0$. But $\lambda>1$
			and $\eta_0(p-1)<0$ together imply $\eta_0(\lambda^{p-1}-1)\neq0$, a
			contradiction. Hence $\varphi\equiv0$, giving $h_1\leq h_2$. As noted at the
			outset, the symmetric argument yields $h_2\leq h_1$, and therefore
			$h_1\equiv h_2$ on $E$.
		\end{proof}
	}

	
	\begin{theorem}[\bf Verification]\label{thm:verification}
		Let $h\in C^2(E;\mathbb{R}_+)$ be the classical solution of \eqref{eq: reduced HJB}. Assume:
		\begin{enumerate}[\upshape(i)]
			\item $ \vartheta \in (0,1)$;
			\item there is a unique solution $\overline{\mathbb{P}}$ to the martingale
			problem on $\mathbb{R}^n\times E$ for the generator
			\[
			\frac{1}{2}\sum_{i,j=1}^{n+1}
			\overline{A}_{ij}(\cdot)\p_{ij}+
			\sum_{i=1}^{n+1}\overline{b}_i(\cdot)\p_i, \h{5pt} \text{with} \h{5pt}
			\overline{A}:=
			\begin{pmatrix}\Sigma \!\!& \Upsilon\\ \Upsilon^\top\!\! & a^2\end{pmatrix}\h{5pt} \text{and} \h{5pt}\overline{b}:=
			\begin{pmatrix}
				\frac{1}{\gamma}\!\left(\mu+\chi\frac{h'}{h}\Upsilon\right)
				\\[10pt]
				\widetilde b
				+
				a^2\frac{h'}{h}
			\end{pmatrix};
			\]
			
			\item $\displaystyle\int_0^\infty h(Y_t)^{-\chi\frac{\psi-1}{1-\gamma}}\,dt\;=\;\infty$,\quad
			$\overline{\mathbb{P}}$-a.s.
			
		\end{enumerate}
		Then $\pi^*:E\to\mathbb{R}^n$ and $l^*:E\to\mathbb{R}$ defined by
		\begin{equation}\label{eq:feedback_controls}
			\pi^*(\cdot)
			\;:=\;
			\frac{1}{\gamma}\,\Sigma(\cdot)^{-1}\!\left[\mu(\cdot)\;+\;\chi\,\frac{h'(\cdot)}{h(\cdot)}\,\Upsilon(\cdot)\right],
			\qquad
			l^*(\cdot)
			\;:=\;
			\beta^\psi\,h(\cdot)^{-\chi\frac{\psi-1}{1-\gamma}}
		\end{equation}
		are the optimal feedback controls of the problem $\displaystyle\sup_{(\pi,l)\in\mathcal A_{EZ}(x,y)} J_0^{\,\xi,\pi,l}$ whose value is
		$\dfrac{x^{1-\gamma}}{1-\gamma}\,h(y)^\chi$. 
	\end{theorem}
	
	\begin{proof}
		Define $\mathcal{V}(t,x,y)
		:=
		e^{-\delta\vartheta t}\,\dfrac{x^{1-\gamma}}{1-\gamma}\,h(y)^\chi$, $X^*:=X^{\xi,\pi^*,l^*}$, $C^*:=l^*(Y)X^*$, and
		\begin{align*}
			J^*_t
			\;:=\;
			\mathcal{V}(t,X^*_t,Y_t)
			\;=\;
			e^{-\delta\vartheta t}\,\frac{(X^*_t)^{1-\gamma}}{1-\gamma}\,h(Y_t)^\chi\quad \text{for each $t\geq0$.}
		\end{align*}
		
		\noindent\textbf{Step 1: Admissibility $(\pi^*,l^*)=(\pi^*(Y),l^*(Y)) \in \mathcal{A}(x,y)$.} As $h>0$ on $E$, and $Y$ is non-explosive and continuous by \Cref{ass.market well-posed}, we see that
		$$	\int_0^T \pig|\big[\pi^*(Y_t)\big]^\top \mu(Y_t)\pig|
		+\big[\pi^*(Y_t)\big]^\top \Sigma(Y_t)\pi^*(Y_t)
		+l^*(Y_t)\,dt < \infty, \text{ for every } T<\infty,\quad \mathbb P\text{-a.s.}$$ 
		As the wealth equation \eqref{eq. wealth eq} is linear, the wealth process \(X^{\h{.5pt}\xi,\pi^*,l^*}\) exists by the Doléans--Dade exponential formula and is positive for all times, $\mathbb{P}$-a.s. Hence $(\pi^*(Y),l^*(Y)) \in \mathcal{A}(x,y)$ by the definition in \eqref{def. A(x,y)}.
		
		\smallskip
		\noindent\textbf{Step 2: $J^*$ is the Epstein--Zin utility process for $(\pi^*,l^*)$.}
		The first-order condition for $l^*$ in \eqref{eq:l_star_chi} reads
		$\beta(l^*)^{-1/\psi}h^{-\frac{\chi}{1-\gamma}(1-1/\psi)}=1$.
		Inserting this into the aggregator $g_{EZ}$ in \eqref{eq:discounted_EZ_aggregator} and using $q=(\vartheta-1)/\vartheta$, a direct computation gives
		\begin{equation}\label{eq:aggregator_linear}
			g_{EZ}(t,C^*_t,J^*_t)\;=\;\vartheta\,l^*(Y_t)\,J^*_t.
		\end{equation} 
		Applying It\^{o}'s formula to $J^*_t = \mathcal{V}(t, X^*_t, Y_t)$ gives 
		\begin{align}\label{4633}
			dJ^*_t
			\;=\;
			\Bigl[\partial_t\mathcal{V}
			+ \mathcal{L}^{\pi^*,\,l^*}\mathcal{V}\Bigr](t,X^*_t,Y_t)\,dt
			\;+\; J^*_t\,dN_t,
		\end{align}
		where the second-order generator and $N$ are defined by
		\begin{align}
			\mathcal{L}^{\pi,l}V(t,x,y)
			&:=
			x\bigl[r(y)+\pi^\top\mu(y)-l\bigr]\partial_xV
			+b(y)\partial_yV \nonumber\\
			&\quad
			+\frac{1}{2}x^2\pi^\top\Sigma(y)\pi\,\partial_{xx}V
			+x\pi^\top\Upsilon(y)\,\partial_{xy}V
			+\frac{1}{2}a(y)^2\partial_{yy}V,
			\nonumber\\
			N_t
			&:=
			(1-\gamma)\int_0^t \big[\pi^*(Y_s)\big]^\top \sigma(Y_s)\,dZ_s
			\;+\;
			\chi\int_0^t\frac{h'(Y_s)}{h(Y_s)}\,a(Y_s)\,dW_s	\label{2997}
		\end{align}
		respectively. Since $(\pi^*,l^*)$ achieves the supremum in
		\eqref{eq:time_dependent_HJB_EZ} when $V=\mathcal{V}$ there, the HJB equation holds with equality at
		this control: $\partial_t\mathcal{V}
		\;+\;
		\mathcal{L}^{\pi^*,\,l^*}\mathcal{V}
		\;+\;
		g_{EZ}(t,C^*_t,J^*_t)
		\;=\; 0.$ Substituting this into \eqref{4633} and then applying
		\eqref{eq:aggregator_linear} yields the SDE
		\begin{equation}\label{eq:J_SDE}
			dJ^*_t
			\;=\;
			-\vartheta\,l^*(Y_t)\,J^*_t\,dt
			\;+\;
			J^*_t\,dN_t.
		\end{equation}
		Denote the stochastic exponential of $\{N_t\}_{t\geq 0}$ by $\mathcal E(N)_t
		:=
		\exp\left(N_t-\frac12\langle N\rangle_t\right).$ We first justify that \(\mathcal E(N)\) is a true martingale on every finite horizon and identify the corresponding change of measure. 
		Let $\{E_m\}_{m\in\mathbb N}$ be an increasing sequence of
		relatively compact open subsets of $E$ such that
		$\overline E_m\subset E_{m+1}$ and $\bigcup_mE_m=E$, and define
		\[
		\tau_m:=\inf\{t\geq0:Y_t\notin E_m\}.
		\]
		Fix \(T<\infty\). These are canonical exit times from a compact exhaustion of $E$ such that
		\(\{\mathcal E(N)_{t\wedge\tau_m}\}_{t\geq 0}\) is a \(\mathbb P\)-martingale on \([0,T]\), for each $m\in\mathbb{N}$. Define a probability measure
		\(\overline{\mathbb P}^{\,m}\) on \((\Omega,\mathcal F_T)\) by
		\[
		\overline{\mathbb P}^{\,m}(A)
		:=
		\mathbb E
		\left[
		\mathcal E(N)_{T\wedge\tau_m}\mathbf 1_A
		\right],
		\qquad \text{for any } A\in\mathcal F_T.
		\]
		Since \(\{\mathcal E(N)_{t\wedge\tau_m}\}_{t\geq 0}\) is a martingale on \([0,T]\), we have $
		\mathbb E [\mathcal E(N)_{T\wedge\tau_m}]=1,$ and hence \(\overline{\mathbb P}^{\,m}\) is indeed a probability measure.\smallskip

		We now compute the drift of the stopped coordinate process under
		\(\overline{\mathbb P}^{\,m}\). Under the original measure \(\mathbb P\), the coordinate process has the decomposition \[ dR_t=\mu(Y_t)\,dt+dM^R_t, \quad M^R_t:=\int_0^t\sigma(Y_s)\,dZ_s \h{10pt} \text{and} \h{10pt}
		dY_t=b(Y_t)\,dt+dM^Y_t, \quad M^Y_t:=\int_0^ta(Y_s)\,dW_s.
		\] 
		By Girsanov's theorem, the process $M^R_{t\wedge\tau_m} - \langle M^R,N\rangle_{t\wedge\tau_m}$ is a $\overline{\mathbb P}^{\,m}$-local martingale on $[0,T]$. Hence we may rewrite the stopped dynamics under \(\overline{\mathbb P}^{\,m}\) as 
		\[ dR_{t\wedge\tau_m} 
		=\mathbf 1_{\{t\le \tau_m\}}\mu(Y_t)\,dt + dM^R_{t\wedge\tau_m}
		= \mathbf 1_{\{t\le \tau_m\}}\mu(Y_t)\,dt + d\Bigl( M^R_{t\wedge\tau_m} - \langle M^R,N\rangle_{t\wedge\tau_m} \Bigr) + d\langle M^R,N\rangle_{t\wedge\tau_m}. \] 
		The second term on the right hand side is now an \(\overline{\mathbb P}^{\,m}\)-local martingale. 
		Since $d\langle M^R\rangle_t=\Sigma(Y_t)\,dt$ and $d\langle M^R,M^Y\rangle_t=\Upsilon(Y_t)\,dt$, we obtain \[ d\langle M^R,N\rangle_t = \left[ (1-\gamma)\Sigma(Y_t)\pi^*(Y_t) + \chi\frac{h'(Y_t)}{h(Y_t)}\Upsilon(Y_t) \right]dt. \] Consequently, under \(\overline{\mathbb P}^{\,m}\), and up to time \(\tau_m\), we use $\pi^* = \frac1\gamma \Sigma^{-1} \left( \mu+\chi\frac{h'}{h}\Upsilon \right)$ to see that the drift of the first \(n\) coordinates is 
		\[ \mu + (1-\gamma)\Sigma\pi^* + \chi\frac{h'}{h}\Upsilon = \frac1\gamma \left( \mu+\chi\frac{h'}{h}\Upsilon \right).\]  Similarly, up to time \(\tau_m\), the drift of the last coordinate becomes
		\[
		b
		+
		(1-\gamma)\Upsilon^\top\pi^*
		+
		\chi\frac{h'}{h}a^2
		=
		\widetilde b
		+
		\chi\frac{h'}{h}
		\left(
		a^2
		+
		\frac{1-\gamma}{\gamma}
		\Upsilon^\top\Sigma^{-1}\Upsilon
		\right)
		= 	\widetilde b
		+
		a^2\frac{h'}{h}
		\] 
		by the definition of $\widetilde b$ in \eqref{eq:p_eta_def} and the choice of \(\chi\) in \eqref{eq:chi_constant}. These two equalities give \(\overline b\).\smallskip

		The new local martingale parts are obtained from
		the old ones by subtracting finite-variation processes: $M^R_{t\wedge\tau_m}
		-
		\langle M^R,N\rangle_{t\wedge\tau_m}$ and $M^Y_{t\wedge\tau_m}
		-\langle M^Y,N\rangle_{t\wedge\tau_m}.$ Since quadratic variation and covariation are unaffected by adding or
		subtracting finite-variation processes, the quadratic covariation matrix of the coordinate process remains $\overline A$, as defined in (ii) of this theorem. Thus, under
		\(\overline{\mathbb P}^{\,m}\), the stopped coordinate process solves the
		stopped martingale problem associated with the generator in (ii) of this theorem up to time \(\tau_m\).\smallskip
		
		By the uniqueness of the martingale problem in (ii) of this theorem and the standard localisation theorem for martingale
		problems, the law of this stopped process agrees with the law of the corresponding stopped process under the unique solution \(\overline{\mathbb P}\). Since the martingale
		problem under \(\overline{\mathbb P}\) is conservative on every finite horizon, $ \overline{\mathbb P}(\tau_m\le T)\longrightarrow 0$ as $m \to \infty$. Since $
		\{\tau_m\le T\}
		=
		\{Y_{T\wedge\tau_m}\notin E_m\},$ the event $\{\tau_m\le T\}$ is measurable with respect to the
		stopped coordinate process	\(\{R_{t\wedge\tau_m},Y_{t\wedge\tau_m}\}_{0\le t\le T}\).  
		By uniqueness of the stopped martingale problem, the law of
		\(\{R_{t\wedge\tau_m},Y_{t\wedge\tau_m}\}_{0\le t\le T}\) under \(\overline{\mathbb P}^{\,m}\)
		coincides with its law under \(\overline{\mathbb P}\). Therefore $\overline{\mathbb P}^{\,m}(\tau_m\le T)
		=
		\overline{\mathbb P}(\tau_m\le T).$ Hence  
		\[
		\overline{\mathbb P}^{\,m}(\tau_m\le T)=\overline{\mathbb P}(\tau_m\le T)\longrightarrow 0\qquad
		\text{as $m\to\infty$.}
		\]
		Using the definition of \(\overline{\mathbb P}^{\,m}\), we have
		\[
		1
		=
		\mathbb E[\mathcal E(N)_{T\wedge\tau_m}]
		=
		\mathbb E
		\left[
		\mathcal E(N)_T\mathbf 1_{\{\tau_m>T\}}
		\right]
		+
		\overline{\mathbb P}^{\,m}(\tau_m\le T).
		\]
		Letting \(m\to\infty\), we obtain $\mathbb E[\mathcal E(N)_T]=1.$ Since
		\(T<\infty\) was arbitrary, \(\mathcal E(N)\) is a true martingale on every finite horizon. Moreover, for every \(A\in\mathcal F_T\),
		\[
		\begin{aligned}
			\mathbb E[\mathcal E(N)_T\mathbf 1_A] 
			=
			\lim_{m\to\infty}
			\mathbb E[\mathcal E(N)_{T\wedge\tau_m}
			\mathbf 1_{A\cap\{\tau_m>T\}}] 
			=
			\lim_{m\to\infty}
			\overline{\mathbb P}^{\,m}(A\cap\{\tau_m>T\})
			&=
			\lim_{m\to\infty}
			\overline{\mathbb P}(A\cap\{\tau_m>T\})  \\
			&=
			\overline{\mathbb P}(A).
		\end{aligned}
		\] 
		
		Finally, we solve \eqref{eq:J_SDE}. The Doléans--Dade exponential formula gives, for \(0\le t\le T<\infty\),
		\begin{align}
			J^*_T
			=
			J^*_t 
			\exp\left[-\int_t^T \vartheta l^*(Y_s)\,ds
			+N_T-N_t
			-\frac12\bigl(\langle N\rangle_T-\langle N\rangle_t\bigr)
			\right]
			=
			J^*_t
			\exp\left(
			-\int_t^T \vartheta l^*(Y_s)\,ds
			\right)
			\frac{\mathcal E(N)_T}{\mathcal E(N)_t}.
			\label{3190}
		\end{align}
		Changing measure to $\overline{\mathbb{P}}$:
		\begin{align}
			\mathbb{E}_t[J^*_T]
			&\;=\;
			J^*_t\,
			\overline{\mathbb{E}}_t\!\left[
			\exp\!\left(-\int_t^T\vartheta\,l^*(Y_s)\,ds\right)\right],
			\label{eq:Ehat_terminal}
			\\[4pt]
			\mathbb{E}_t\!\left[\int_t^T g_{EZ}(s,C^*_s,J^*_s)\,ds\right]
			&\;=\;
			J^*_t\,
			\overline{\mathbb{E}}_t\!\left[
			\int_t^T\vartheta\,l^*(Y_s)\,
			\exp\left(-\int_t^s\vartheta\,l^*(Y_u)\,du\right)\,ds\right],
			\label{eq:Ehat_integral}
		\end{align}
		where \eqref{eq:Ehat_integral} uses \eqref{eq:aggregator_linear}.
		Recognising the integrand in \eqref{eq:Ehat_integral} as
		$-\frac{d}{ds}\,\exp(-\int_t^s\vartheta\,l^*(Y_u)\,du)$ and integrating, then adding
		\eqref{eq:Ehat_terminal}, gives the finite-horizon consistency
		\begin{equation}\label{eq:consistency}
			J^*_t
			\;=\;
			\mathbb{E}_t\!\left[
			\int_t^T g_{EZ}(s,C^*_s,J^*_s)\,ds\;+\;J^*_T
			\right],
			\qquad 0\le t\le T<\infty.
		\end{equation}
		By the transversality condition~(iii) of this theorem and the fact that $Y$ is non-explosive, $\exp\pig(-\int_t^T\vartheta\,l^*(Y_s)\,ds\pig)\to0$ as
		$T\to\infty$ $\overline{\mathbb{P}}$-a.s., so $\mathbb{E}_t[J^*_T]\to0$ by
		\eqref{eq:Ehat_terminal} and the conditional dominated
		convergence. Passing $T\to\infty$
		in \eqref{eq:consistency},
		\[
		J^*_t
		\;=\;
		\mathbb{E}_t\!\left[\int_t^\infty g_{EZ}(s,C^*_s,J^*_s)\,ds\right].
		\]
		The absolute integrability of $g_{EZ}(s,C^*_s,J^*_s)$ follows from \eqref{eq:aggregator_linear} and \eqref{3190}. So $J^*$ is the utility process generated by $(\pi^*,l^*)$. We next establish the uniqueness.

		\smallskip
		\noindent\textbf{Step 3: Uniqueness of utility process associated with $(\pi^*,l^*)$.} For $i=1,2$, we let \(J^{(i)}\) be adapted càdlàg processes
		taking values in \((1-\gamma)\mathbb R_+\) and satisfying 
		\begin{align}
			J^{(i)}_t
			=
			\mathbb E_t
			\left[
			\int_t^\infty g_{EZ}\Big(s,C_s^*,J^{(i)}_s\Big)\,ds
			\right],
			\quad \text{for } t\ge0\h{5pt} \text{and} \h{5pt}
			\mathbb E\left[
			\int_0^\infty
			\left|g_{EZ}\Big(s,C_s^*,J^{(i)}_s\Big)\right|\,ds
			\right]<\infty.
			\label{3250}
		\end{align}
		We verify that the comparison theorem
		\cite[Theorem 5.8]{herdegen2023infinite_II} applies. Since \(C^*>0\) and \(\vartheta \in (0,1)\), it suffices to check for the uniform integrability of \(J^{(i)}\) and the sub/supersolution property. Set $
		f^{(i)}
		:=
		\int_0^\infty
		\left|
		g_{EZ}(u,C_u^*,J_u^{(i)})
		\right|\,du$ and $F_t^{(i)}:=\mathbb E_t[f^{(i)}].$ By \eqref{3250}, $f^{(i)}\in L^1$, and hence $F^{(i)}$
		is a uniformly integrable martingale. Moreover, for every finite
		stopping time $\tau$,
		\[
		|J_\tau^{(i)}|
		\le
		\mathbb E_\tau\left[
		\int_\tau^\infty
		\left|
		g_{EZ}(u,C_u^*,J_u^{(i)})
		\right|\,du
		\right]
		\le F_\tau^{(i)}.
		\]
		Therefore $J^{(i)}$ is uniformly integrable.\smallskip 
		
		Each \(J^{(i)}\) is both a subsolution and a supersolution in the sense of
		\cite[Definition 5.3]{herdegen2023infinite_II}. Indeed, by the
		integrability in \eqref{3250} and the above estimate,
		\[
		J^{(i)}_t
		+
		\int_0^t
		g_{EZ}\bigl(s,C_s^*,J_s^{(i)}\bigr)\,ds
		=
		\mathbb E_t
		\left[
		\int_0^\infty
		g_{EZ}\bigl(s,C_s^*,J_s^{(i)}\bigr)\,ds
		\right]
		\]
		is a uniformly integrable martingale. Therefore, for any bounded stopping times \(\tau_1\le \tau_2\), we use the optional stopping theorem
		to obtain,
		\[
		J^{(i)}_{\tau_1}
		=
		\mathbb E_{\tau_1}
		\left[
		J^{(i)}_{\tau_2}
		+
		\int_{\tau_1}^{\tau_2}
		g_{EZ}\bigl(s,C_s^*,J_s^{(i)}\bigr)\,ds
		\right]
		=
		\mathbb E_{\tau_1}
		\left[
		J^{(i)}_{\tau_2^+}
		+
		\int_{\tau_1}^{\tau_2}
		g_{EZ}\bigl(s,C_s^*,J_s^{(i)}\bigr)\,ds
		\right]
		\]
		since \(J^{(i)}\) is c\`adl\`ag. Furthermore, dominated convergence theorem and \eqref{3250} give
		\[
		\lim_{t\to\infty}\mathbb E\left[\big|J^{(i)}_{t+}\big|\right]=
		\lim_{t\to\infty}\mathbb E\left[\big|J^{(i)}_t\big|\right]
		\le
		\lim_{t\to\infty}\mathbb E
		\left[
		\int_t^\infty
		\left|
		g_{EZ}\bigl(s,C_s^*,J_s^{(i)}\bigr)
		\right|\,ds
		\right]
		= 0, 
		\] 
		because \(J^{(i)}\) is c\`adl\`ag. Applying \cite[Theorem 5.8]{herdegen2023infinite_II}, we conclude that $
		J^{(1)}_\tau=J^{(2)}_\tau$, $\mathbb P$-a.s. for every finite stopping time \(\tau\). Since both processes are c\`adl\`ag, it follows that \(J^{(1)}\) and \(J^{(2)}\) are indistinguishable.

		\smallskip
		\noindent\textbf{Step 4: Optimality over $\mathcal{A}_{EZ}(x,y)$.} Fix \((\pi,l)\in\mathcal A_{EZ}(x,y)\). Let \(X^{*,1}\) denote the wealth process generated by \((\pi^*,l^*)\) with unit initial wealth. For
		\(\varepsilon>0\), define
		\begin{align}
			X_t^\varepsilon
			:=
			X_t^{\xi,\pi,l}
			+
			\varepsilon X_t^{*,1}>0,
			\qquad
			\Phi_t^\varepsilon
			:=
			\mathcal V(t,X_t^\varepsilon,Y_t).
			\label{3850}
		\end{align}
		Moreover, \(X^\varepsilon\)
		is the wealth process associated with the portfolio and consumption--wealth ratio
		\begin{align}
			\pi_t^\varepsilon
			:=
			\frac{
				X_t^{\xi,\pi,l}\pi_t
				+
				\varepsilon X_t^{*,1}\pi^*(Y_t)
			}
			{X_t^\varepsilon},
			\qquad
			l_t^\varepsilon
			:=
			\frac{
				l_tX_t^{\xi,\pi,l}
				+
				\varepsilon l^*(Y_t)X_t^{*,1}
			}
			{X_t^\varepsilon}>0.
			\label{3817}
		\end{align}
		Applying It\^o's formula to $
		\Phi_t^\varepsilon
		=
		\mathcal V(t,X_t^\varepsilon,Y_t)$ and using the HJB equation gives
		\begin{align}
			d\Phi_t^\varepsilon
			\le
			-
			g_{EZ}\bigl(t,l_t^\varepsilon X_t^\varepsilon,\Phi_t^\varepsilon\bigr)\,dt
			+
			dM_t^\varepsilon,
			\label{3885}
		\end{align}
		where \(M^\varepsilon\) is a local martingale. Let \(\tau_1\le\tau_2\) be bounded
		stopping times. Localising the martingale \(M^\varepsilon\), applying optional
		sampling, and then removing the localisation by Fatou's lemma and conditional monotone convergence theorem gives
		\begin{align}
			\Phi_{\tau_1}^\varepsilon
			\ge
			\mathbb E_{\tau_1}^y
			\left[
			\Phi_{\tau_2}^\varepsilon
			+
			\int_{\tau_1}^{\tau_2}
			g_{EZ}\bigl(s,l_s^\varepsilon X_s^\varepsilon,\Phi_s^\varepsilon\bigr)\,ds
			\right].
			\label{3898}
		\end{align}
		Here the passage to the limit is justified as follows. The function
		\(x\mapsto\mathcal V(t,x,y)\) is increasing on \((0,\infty)\), and
		\(X_t^\varepsilon\ge \varepsilon X_t^{*,1}\). Hence $
		\Phi_t^\varepsilon
		\ge
		\mathcal V(t,\varepsilon X_t^{*,1},Y_t).$
		By Steps 2-3, the process on the right-hand side is uniformly integrable. Thus the localised terminal terms are bounded from below by a uniformly integrable
		family, which permits Fatou's lemma. In particular,
		\begin{align}
			\liminf_{T\to\infty}
			\mathbb E\left[\Phi_T^\varepsilon\right]
			\geq 	\liminf_{T\to\infty}
			\mathbb E\left[\mathcal V(T,\varepsilon X_T^{*,1},Y_T)\right]
			\ge 0
			\label{3917}
		\end{align}
		by Steps 2-3. Therefore \(\Phi^\varepsilon\) is a supersolution associated with the consumption \(l^\varepsilon X^\varepsilon\).\smallskip
		
		Since the aggregator $g_{EZ}$ is increasing in consumption and $
		l_t^\varepsilon X_t^\varepsilon
		=
		l_tX_t^{\xi,\pi,l}
		+
		\varepsilon l^*(Y_t)X_t^{*,1}
		\ge
		l_tX_t^{\xi,\pi,l},$ inequality \eqref{3898} implies that \(\Phi^\varepsilon\) is also a supersolution associated with the consumption \(lX^{\xi,\pi,l}\). On the other hand, since \((\pi,l)\in\mathcal A_{EZ}(x,y)\), \(J^{\xi,\pi,l}\) is both a subsolution and a
		supersolution associated with \(lX^{\xi,\pi,l}\). Moreover, the absolute
		integrability condition in the definition of \(\mathcal A_{EZ}(x,y)\) and
		\[
		\left|J_t^{\xi,\pi,l}\right|
		\le
		\mathbb E_t
		\left[
		\int_t^\infty
		\left|
		g_{EZ}\bigl(s,l_sX_s^{\xi,\pi,l},J_s^{\xi,\pi,l}\bigr)
		\right|\,ds
		\right]
		\]
		imply that \(J^{\xi,\pi,l}\) is uniformly integrable by similar reasons as in Step 3.\smallskip
		
		The comparison theorem \cite[Theorem 5.8]{herdegen2023infinite_II} applied to \(J^{\xi,\pi,l}\) as a subsolution and
		\(\Phi^\varepsilon\) as a supersolution gives $J_\tau^{\xi,\pi,l}
		\le
		\Phi_\tau^\varepsilon$ for every finite stopping time \(\tau\). In particular, taking \(\tau=0\), we obtain
		\[
		J_0^{\xi,\pi,l}
		\le
		\Phi_0^\varepsilon
		=
		\mathcal V(0,x+\varepsilon,y)
		=
		\frac{(x+\varepsilon)^{1-\gamma}}{1-\gamma}h(y)^\chi.
		\]
		Letting \(\varepsilon\downarrow0\), since \((\pi,l)\in\mathcal A_{EZ}(x,y)\) is arbitrary, we have $\sup_{(\pi,l)\in\mathcal A_{EZ}(x,y)}
		J_0^{\xi,\pi,l}
		\le
		\frac{x^{1-\gamma}}{1-\gamma}h(y)^\chi.$ The reverse inequality follows from Step \(2\), because \((\pi^*,l^*)\) belongs to \(\mathcal A_{EZ}(x,y)\) and its utility process is $
		J_t^*
		=
		\mathcal V(t,X_t^*,Y_t).$ This proves the optimality of \((\pi^*,l^*)\) over \(\mathcal A_{EZ}(x,y)\).\end{proof}
	
	\nihil{
		\begin{theorem}[Verification]\label{thm:verification}
			Let $h$ be the classical solution of the reduced HJB equation \eqref{eq: reduced HJB} and define
			\begin{equation}\label{eq:V_def}
				\mathcal{V}^*(x,y)
				\;:=\;
				\frac{x^{1-\gamma}}{1-\gamma}\,h(y)^{\chi}, 
			\end{equation}
			and
			\begin{align}\label{eq:feedback_controls}
				l^{*}(y)
				&:=
				\beta^{\psi}\,
				h(y)^{-\frac{\chi(\psi-1)}{1-\gamma}},
				&
				\pi^{*}(y)
				&:=
				\frac{1}{\gamma}\,\Sigma(y)^{-1}
				\!\left[
				\mu(y)
				+
				\chi\,\frac{h'(y)}{h(y)}\,\Upsilon(y)
				\right].
			\end{align}
			For the candidate controlled state process, define the continuous local
			martingale
			\begin{equation}\label{eq:N_def}
				N_t
				:=
				(1-\gamma)\int_0^t
				\pi^{*}(Y_s)^{\top}\sigma(Y_s)\,dZ_s
				+
				\chi\int_0^t
				\frac{h'(Y_s)}{h(Y_s)}\,a(Y_s)\,dW_s.
			\end{equation}
			Suppose the following conditions hold:
			\begin{enumerate}[(i).]
				\item\label{ass:vartheta}
				$0 < \vartheta := \dfrac{1-\gamma}{1-\psi^{-1}} < 1$;
				\item\label{ass:admissibility}
				$\pig(\pi^{*}(Y),l^{*}(Y)\pig) \in \mathcal{A}_{EZ}(x,y)$;
				\item\label{ass:true_mg}
				$\mathcal{E}(N)_t := \exp\bigl(N_t - \tfrac{1}{2}\langle N\rangle_t\bigr)$ is a true martingale on every finite horizon; let
				$\widehat{\mathbb{P}}$ be the probability measure defined on finite horizons by
				\[
				\frac{d\widehat{\mathbb{P}}}{d \mathbb{P}}\bigg|_{\mathcal{F}_{T}}
				=
				\mathcal{E}(N)_{T}.
				\]
				\item\label{ass:transversality}
				The transversality condition
				\begin{equation}\label{eq:transversality}
					\int_0^{\infty} l^{*}(Y_s)\,ds = \infty
					\qquad\widehat{\mathbb{P}}\text{-a.s.}
				\end{equation}
				holds.
			\end{enumerate}
			The stochastic integral in It\^{o}'s
			formula applied to $\Phi^{\pi,l}_t := \mathcal{V}(t,X^{\h{.5pt}\xi,\pi,l}_t,Y_t)$ is a
			true martingale after localisation, and
			\begin{equation}\label{eq:terminal_cond}
				\liminf_{T\to\infty}
				\mathcal{E}_t\!\left[\Phi^{\pi,l}_T\right]
				\geq 0.
			\end{equation}
			Then $(\pi^{*},l^{*})$ is optimal over $\mathcal{A}_{EZ}(x,y)$ and
			\begin{equation}\label{eq:value}
				\sup_{(\pi,l)\in\mathcal{A}_{EZ}(x,y)}
				J_0^{\pi,l}
				\;=\;
				V(x,y)
				\;=\;
				\frac{x^{1-\gamma}}{1-\gamma}\,h(y)^{\chi}.
			\end{equation}
		\end{theorem}
		
		\begin{proof}
			We write $X^{*}_t := X^{x,\pi^{*},l^{*}}_t$, $C^{*}_t := l^{*}(Y_t)X^{*}_t$,
			and
			\[
			J^{*}_t
			:=
			\mathcal{V}(t,X^{*}_t,Y_t)
			=
			e^{-\delta\vartheta t}\,
			\frac{(X^{*}_t)^{1-\gamma}}{1-\gamma}\,h(Y_t)^{\chi}.
			\]
			
			\smallskip\noindent
			\textbf{Step 1: Linearity of the aggregator along the candidate path.}
			The first-order condition for $l^{*}$ reads
			\begin{equation}\label{eq:foc_l}
				\beta\,(l^{*})^{-1/\psi}\,
				h^{-\frac{\chi}{1-\gamma}\!\left(1-\frac{1}{\psi}\right)}
				=
				1.
			\end{equation}
			Using \eqref{eq:foc_l} together with the definition of the aggregator $g_{EZ}$ and the identity $q = (\vartheta-1)/\vartheta$, a direct
			computation gives
			\begin{equation}\label{eq:aggregator_linear}
				g_{EZ}(t,C^{*}_t,J^{*}_t)
				=
				\beta\,e^{-\delta t}\,
				\frac{(l^{*}(Y_t)X^{*}_t)^{1-1/\psi}}{1-1/\psi}\,
				\bigl[(1-\gamma)J^{*}_t\bigr]^{q}
				=
				\vartheta\,l^{*}(Y_t)\,J^{*}_t.
			\end{equation}
			
			\smallskip\noindent
			\textbf{Step 2: $J^{*}$ is the Epstein--Zin utility process for
				$(\pi^{*},l^{*})$.}
			
			Applying It\^{o}'s formula to $\mathcal{V}(t,X^{*}_t,Y_t)$ and invoking the
			Hamilton--Jacobi--Bellman equality at the maximising controls $(\pi^{*},l^{*})$
			yields, in light of \eqref{eq:aggregator_linear},
			\begin{equation}\label{eq:J_SDE}
				dJ^{*}_t
				=
				-\vartheta\,l^{*}(Y_t)\,J^{*}_t\,dt
				+
				J^{*}_t\,dN_t,
			\end{equation}
			where $N$ is defined in \eqref{eq:N_def}. The stochastic exponential
			$\mathcal{E}(N)_t := \exp\bigl(N_t - \tfrac{1}{2}\langle N\rangle_t\bigr)$
			satisfies
			\[
			\begin{aligned}
				\langle N\rangle_t
				=
				\int_0^t
				\Bigl[&
				(1-\gamma)^2\,\pi^{*}(Y_s)^{\top}\Sigma(Y_s)\pi^{*}(Y_s)
				\\
				&+\,
				\chi^2\!\left(\frac{h'(Y_s)}{h(Y_s)}\right)^{\!2}
				a(Y_s)^2
				+
				2(1-\gamma)\chi\,\frac{h'(Y_s)}{h(Y_s)}\,
				a(Y_s)\,\pi^{*}(Y_s)^{\top}\sigma(Y_s)\rho(Y_s)
				\Bigr]ds.
			\end{aligned}
			\]
			Solving the linear SDE \eqref{eq:J_SDE} gives, for $0 \le t \le T$,
			\begin{equation}\label{eq:J_solution}
				J^{*}_T
				=
				J^{*}_t\,
				\exp\left(
				-\int_t^T \vartheta\,l^{*}(Y_s)\,ds
				\right)
				\frac{\mathcal{E}(N)_T}{\mathcal{E}(N)_t}.
			\end{equation}
			Since $\mathcal{E}(N)$ is a true martingale by assumption~\ref{ass:true_mg}, we may
			change measure to $\widehat{\mathbb{P}}$ and obtain
			\begin{align}
				\mathcal{E}_t[J^{*}_T]
				&=
				J^{*}_t\,
				\widehat{\mathbb{E}}_t\!\left[
				\exp\left(-\int_t^T \vartheta\,l^{*}(Y_s)\,ds\right)
				\right],
				\label{eq:Ehat_terminal}
				\\
				\mathcal{E}_t\!\left[\int_t^T g_{EZ}(s,C^{*}_s,J^{*}_s)\,ds\right]
				&=
				J^{*}_t\,
				\widehat{\mathbb{E}}_t\!\left[
				\int_t^T \vartheta\,l^{*}(Y_s)\,
				\exp\left(-\int_t^s \vartheta\,l^{*}(Y_u)\,du\right)ds
				\right],
				\label{eq:Ehat_integral}
			\end{align}
			\sloppy where the second identity uses \eqref{eq:aggregator_linear}. Recognising the
			integrand in \eqref{eq:Ehat_integral} as the derivative
			$-\tfrac{d}{ds}\exp\bigl(-\int_t^s\vartheta l^{*}(Y_u)\,du\bigr)$ and
			integrating, then adding \eqref{eq:Ehat_terminal} and \eqref{eq:Ehat_integral},
			yields the finite-horizon consistency relation
			\begin{equation}\label{eq:consistency}
				J^{*}_t
				=
				\mathcal{E}_t\!\left[
				\int_t^T g_{EZ}(s,C^{*}_s,J^{*}_s)\,ds
				+
				J^{*}_T
				\right],
				\qquad 0 \le t \le T < \infty.
			\end{equation}
			By the transversality condition~\eqref{eq:transversality},
			$\exp\bigl(-\int_t^T\vartheta\,l^{*}(Y_s)\,ds\bigr)\to 0$ as
			$T\to\infty$, $\widehat{\mathbb{P}}$-a.s.
			Hence $\mathcal{E}_t[J^{*}_T]\to 0$ as $T\to\infty$ by \eqref{eq:Ehat_terminal} and
			the monotone convergence theorem. Passing $T\to\infty$ in
			\eqref{eq:consistency} gives
			\[
			J^{*}_t
			=
			\mathcal{E}_t\!\left[
			\int_t^{\infty} g_{EZ}(s,C^{*}_s,J^{*}_s)\,ds
			\right],
			\]
			so $J^{*}$ is indeed the Epstein--Zin utility process generated by
			$(\pi^{*},l^{*})$.
			
			\smallskip\noindent
			\textbf{Step 3: Comparison with an arbitrary admissible control.}
			
			Let $(\pi,l)\in\mathcal{A}_{EZ}(x,y)$ and set $X_t := X^{\h{.5pt}\xi,\pi,l}_t$,
			$C_t := l_t X_t$, and $\Phi_t := \mathcal{V}(t,X_t,Y_t)$. Applying It\^{o}'s
			formula to $\Phi$ and using the fact that $\mathcal{V}$ satisfies the HJB inequality
			at the sub-optimal controls $(\pi,l)$, we obtain
			\[
			d\Phi_t
			\;\leq\;
			-g_{EZ}(t,C_t,\Phi_t)\,dt + dM_t,
			\]
			where $M$ is a local martingale. The localisation and terminal
			conditions \eqref{eq:terminal_cond} in the definition of $\mathcal{A}_{EZ}(x,y)$
			then give
			\begin{equation}\label{eq:Phi_supersolution}
				\Phi_t
				\;\geq\;
				\mathcal{E}_t\!\left[
				\int_t^T g_{EZ}(s,C_s,\Phi_s)\,ds + \Phi_T
				\right],
				\qquad 0 \le t \le T < \infty,
			\end{equation}
			so $\Phi$ is a supersolution of the recursive utility equation for the
			consumption stream $C$. The Epstein--Zin utility process $J^{\h{.5pt}\xi,\pi,l}$, by
			definition, satisfies
			\begin{equation}\label{eq:J_solution_arb}
				J^{\h{.5pt}\xi,\pi,l}_t
				=
				\mathcal{E}_t\!\left[
				\int_t^T g_{EZ}(s,C_s,J^{\h{.5pt}\xi,\pi,l}_s)\,ds + J^{\h{.5pt}\xi,\pi,l}_T
				\right],
				\qquad 0 \le t \le T < \infty,
			\end{equation}
			together with the infinite-horizon transversality condition.
			
			We now establish $J^{\h{.5pt}\xi,\pi,l}_t \le \Phi_t$ for all $t \ge 0$.
			We claim that $J^{\h{.5pt}\xi,\pi,l}_{\tau} \le \Phi_{\tau}$ for every finite
			stopping time $\tau$; it suffices to prove this for bounded $\tau$, since the
			general case follows by applying the result to $\tau \wedge n$ and
			letting $n\to\infty$.
			
			A key ingredient is the monotonicity of the aggregator in its third argument.
			Since $0<\vartheta<1$ by assumption~\ref{ass:vartheta}, we have
			$q = (\vartheta-1)/\vartheta < 0$, and a direct differentiation gives
			\[
			\partial_v\,g_{EZ}(t,c,v)
			=
			\beta\,e^{-\delta t}\,c^{1-1/\psi}\,q\vartheta\,
			\bigl[(1-\gamma)v\bigr]^{q-1}
			< 0
			\qquad\text{on }\{(1-\gamma)v > 0\}.
			\]
			Hence $g_{EZ}$ is strictly decreasing in $v$ wherever $(1-\gamma)v>0$.
			
			We first handle the case $\gamma < 1$, so that $J^{\h{.5pt}\xi,\pi,l} \ge 0$ and
			$\Phi \ge 0$. Suppose for contradiction that for some bounded stopping time
			$\tau$,
			\[
			\mathbb{P}\!\left(J^{\h{.5pt}\xi,\pi,l}_{\tau} > \Phi_{\tau}\right) > 0.
			\]
			Then there exists $\varepsilon > 0$ such that the event
			$A := \{J^{\h{.5pt}\xi,\pi,l}_{\tau} - \Phi_{\tau} \ge \varepsilon\} \in \mathcal{F}_{\tau}$
			has positive probability. Set $D_t := J^{\h{.5pt}\xi,\pi,l}_t - \Phi_t$ and define
			\[
			\sigma := \inf\{s \ge \tau : D_s \le 0\}.
			\]
			Since $J^{\h{.5pt}\xi,\pi,l}$ and $\Phi$ are c\`{a}dl\`{a}g, $D_{\sigma} \le 0$ on
			$\{\sigma < \infty\}$, and $D_s > 0$ on the stochastic interval $[\tau,\sigma)$.
			
			Fix $n > \|\tau\|_{\infty}$ and set $\sigma_n := \sigma \wedge n$.
			Subtracting \eqref{eq:Phi_supersolution} from \eqref{eq:J_solution_arb}
			on the interval $[\tau,\sigma_n]$, multiplying by $\mathbf{1}_A$, and using
			the strict decrease of $g_{EZ}$ in $v$ on $[\tau,\sigma)$ (where $D_s > 0$),
			we obtain
			\begin{equation}\label{eq:comparison_key}
				\mathcal{E}\!\left[\mathbf{1}_A D_{\tau}\right]
				\;\leq\;
				\mathcal{E}\!\left[\mathbf{1}_A D_{\sigma_n}\right]
				=
				\mathcal{E}\!\left[\mathbf{1}_A\mathbf{1}_{\{\sigma\le n\}}D_{\sigma}\right]
				+
				\mathcal{E}\!\left[\mathbf{1}_A\mathbf{1}_{\{\sigma > n\}}D_n\right].
			\end{equation}
			The first term on the right is non-positive since $D_{\sigma} \le 0$.
			On $\{\sigma > n\}$ we have $D_n > 0$, so $D_n = J^{\h{.5pt}\xi,\pi,l}_n - \Phi_n \le
			J^{\h{.5pt}\xi,\pi,l}_n$ (using $\Phi_n \ge 0$). Hence
			\[
			\mathcal{E}\!\left[\mathbf{1}_A D_{\tau}\right]
			\;\leq\;
			\mathcal{E}\!\left[J^{\h{.5pt}\xi,\pi,l}_n\right].
			\]
			The infinite-horizon transversality condition for $J^{\h{.5pt}\xi,\pi,l}$ gives
			$\limsup_{n\to\infty}\mathcal{E}[J^{\h{.5pt}\xi,\pi,l}_n] \le 0$, so the right-hand side
			vanishes as $n\to\infty$. But $D_{\tau} \ge \varepsilon$ on $A$,
			giving $\mathcal{E}[\mathbf{1}_A D_{\tau}] \ge \varepsilon\,\mathbb{P}(A) > 0$, a contradiction.
			
			For the case $\gamma > 1$, in which $J^{\h{.5pt}\xi,\pi,l} \le 0$ and $\Phi \le 0$, the
			same argument applies to the transformed processes $\widetilde{J} := -\Phi$
			and $\widetilde{\Phi} := -J^{\h{.5pt}\xi,\pi,l}$ with the transformed aggregator
			$\widetilde{g}_{EZ}(s,c,u) := -g_{EZ}(s,c,-u)$. The supersolution property
			of $\Phi$ becomes the subsolution property of $\widetilde{J}$, the solution
			property of $J^{\h{.5pt}\xi,\pi,l}$ becomes the supersolution property of
			$\widetilde{\Phi}$, the transformed aggregator remains strictly decreasing in
			$u$, and the transformed processes take values in $\mathbb{R}_{+}$, so the positive
			domain argument yields $\widetilde{J}_{\tau} \le \widetilde{\Phi}_{\tau}$,
			equivalently $J^{\h{.5pt}\xi,\pi,l}_{\tau} \le \Phi_{\tau}$.
			
			Hence, in both cases, $J^{\h{.5pt}\xi,\pi,l}_{\tau} \le \Phi_{\tau}$ $\mathbb{P}$-a.s.\ for
			every finite stopping time $\tau$. Taking $\tau = 0$ gives
			\[
			J_0^{\pi,l}
			\;\leq\;
			\mathcal{V}(0,x,y)
			=
			V(x,y)
			=
			\frac{x^{1-\gamma}}{1-\gamma}\,h(y)^{\chi}.
			\]
			Since the candidate control $(\pi^{*},l^{*})$ attains this bound by Step~2,
			it is optimal and
			\[
			\sup_{(\pi,l)\in\mathcal{A}_{EZ}(x,y)} J_0^{\pi,l}
			=
			\frac{x^{1-\gamma}}{1-\gamma}\,h(y)^{\chi},
			\]
			which completes the proof.
		\end{proof}
	}
	
	Let $h$ be the solution obtained in Lemma \ref{prop:pass_R_to_infty} when $E=\mathbb{R}_+$ or in Lemma \ref{prop:pass_R_to_infty E=R} when $E=\mathbb{R}$. We next prove that this $h$ satisfies the assumptions in \Cref{thm:verification}. To this end, we impose recurrence assumptions on the reference diffusion with drift $\widetilde{b}$, in that it has an invariant probability density and does not leave its domain. We introduce the reference diffusion
	\begin{equation} \label{eq. shifted state variable}
		d\widetilde{Y}_t=\widetilde{b}(\widetilde{Y}_t)dt+a(\widetilde{Y}_t)dW_t
	\end{equation} 

	\noindent Recall that, for an arbitrary $c \in E$, the scale function $\widetilde{s}:E \to \mathbb{R}$ of \eqref{eq. shifted state variable} is defined by
	\begin{equation} \label{eq. scale fucntion of shifted state variable}
		\widetilde{s}(y):=\int_c^y \exp\left(-2\int_c^u \frac{\widetilde{b}(s)}{a(s)^2}ds\right)du,
	\end{equation}
	\begin{assumption}[\bf Recurrence] \label{assu. Qualitative Verification}
		The scale function $\widetilde{s}$ defined in \eqref{eq. scale fucntion of shifted state variable} for \eqref{eq. shifted state variable} satisfies
		\begin{equation*}
			\lim_{y \to (\partial E)_+}\widetilde{s}(y)=+\infty,
			\qquad
			\lim_{y \to (\partial E)_-}\widetilde{s}(y)=-\infty;
		\end{equation*}
		where $(\partial E)_+,(\partial E)_{-}$ are the right and left endpoints of $E$, respectively. More precisely, $(\partial E)_+=\infty$. When $E=\mathbb{R}_+$, $(\partial E)_{-}=0$; when $E=\mathbb{R}$, $(\partial E)_{-}=-\infty$ .
	\end{assumption} 
	\begin{theorem}\label{thm verify h is optimal} Assume that $\vartheta\in(0,1)$, Assumptions~\ref{ass:kappa} and \ref{assu. Qualitative Verification} hold, and $h$ is a bounded positive solution of \eqref{eq: reduced HJB}. Then the controls defined in \eqref{eq:feedback_controls} corresponding to this $h$ are optimal and the value function is given by $\displaystyle\sup_{(\pi,l)\in\mathcal{A}_{EZ}(x,y)}
		J_0^{\h{.5pt}\xi,\pi,l} 
		\;=\;
		\frac{x^{1-\gamma}}{1-\gamma}\,h(y)^{\chi}.$
		In particular, $h$ can be chosen from:
		\begin{enumerate}[(i).]
			\item the limiting solution obtained in Lemma \ref{prop:pass_R_to_infty E=R} if $E=\mathbb{R}$ and \Cref{ass:polynomial_growth} holds;
			\item the limiting solution obtained in Lemma \ref{prop:pass_R_to_infty} if $E=\mathbb{R}_+$ and \Cref{ass poly growth R_+} holds. 
		\end{enumerate}  
	\end{theorem}
	\begin{proof}We only need to verify Assumptions (ii) and (iii) in \Cref{thm:verification}.
		
		\smallskip\noindent
		\textbf{Step 1: Verifying Assumption (ii).} Recall the projection $\Pi_2(x)=x^{n+1}$ for $x\in \mathbb{R}^{n+1}$. 
		The scale function $\overline{s} $ of the state $\overline Y$ satisfying $	d\,\overline Y_t=\Pi_2\pig(\overline b(\overline Y_t)\pig)dt
		+a(\overline Y_t)dW_t$ is
		\[
		\overline s (y)
		:=
		\int_c^y
		\exp\left(
		-2\int_c^u
		\left[
		\frac{\widetilde b(s)}{a(s)^2}
		+
		\frac{h'(s)}{h(s)}
		\right]ds
		\right)du
		= 
		\int_c^y
		\exp\left(
		-2\int_c^u
		\frac{\widetilde b(s)}{a(s)^2}ds
		\right)
		\frac{h(c)^2}{h(u)^2}\,du.
		\] 
		Thus, we have
		\[
		\overline s (y)
		\ge
		\frac{h(c)^2}{(\sup h)^2}\widetilde{s}(y),\quad\text{if \(y>c\);}\quad
		\overline s (y)
		\le
		\frac{h(c)^2}{(\sup h)^2}\widetilde{s}(y),\quad\text{if \(y<c\).}
		\] 
		By the recurrence assumption on \(\widetilde{s}\) and the fact that $h$ is uniformly bounded on $E$, it follows that
		\[
		\lim_{y\to(\partial E)_+}\overline s (y)=\infty,
		\qquad
		\lim_{y\to(\partial E)_-}\overline s (y)=-\infty.
		\]
		Hence the state $\overline Y$ does not exit \(E\) by \cite[Problem 5.27 and Theorem 5.29, Section C]{karatzas2014brownian}. Thus, the SDE $	d\, \overline Y_t=\Pi_2\pig(\overline b(\overline Y_t)\pig)dt
		+a(\overline Y_t)dW_t$ has a unique weak solution on $E$, by the regularity in \Cref{ass.market well-posed}. One can construct the solution of the first $n$-dimensional martingale problem by adding Brownian noise with the prescribed covariance and cross-covariance. Since $a\neq 0$ in $E$, the Brownian motion driving $\overline Y$ can be recovered from $\overline Y$. It implies that the martingale problem is wellposed.

		\smallskip\noindent
		\textbf{Step 2: Verifying Assumption (iii).} This is immediate because $h$ is bounded from above uniformly in $E$ and $\vartheta>0$.
	\end{proof} 
	
	\begin{proposition}[\bf Classical and Generalised 
		Epstein--Zin Value]
		\label{prop:classical_generalised_value_equality}
		Under the assumptions of the first part of \Cref{thm verify h is optimal}, the generalised utility process defined in 
		\Cref{def:generalised_EZ_solution} satisfies
		\begin{equation*}
			\sup_{(\pi,l)\in\mathcal{A}(x,y)}
			\widetilde J_0^{\,\xi,\pi,l}
			=
			\sup_{(\pi,l)\in\mathcal{A}_{EZ}(x,y)}
			J_0^{\,\xi,\pi,l}
			=\frac{x^{1-\gamma}}{1-\gamma}\,h(y)^{\chi}.
		\end{equation*}
	\end{proposition}
	
	\begin{proof} By \Cref{lem:generalised_EZ_existence_comparison}, the generalised utility process exists for every admissible control in $\mathcal{A}(x,y)$. By \cite[Theorem~6.7]{herdegen2023infinite_II}, the generalised and 
		classical Epstein--Zin utility processes are indistinguishable whenever 
		the latter is well-defined. Since $\mathcal{A}_{EZ}(x,y)\subset
		\mathcal{A}(x,y)$, we have
		\[
		\sup_{(\pi,l)\in\mathcal{A}(x,y)}
		\widetilde J^{\,\xi,\pi,l}_0
		\;\ge\;
		\sup_{(\pi,l)\in\mathcal{A}_{EZ}(x,y)}
		J^{\h{.5pt}\xi,\pi,l}_0
		\;=\;
		\frac{x^{1-\gamma}}{1-\gamma}\,h(y)^\chi,
		\]
		where the right-hand equality is due to \Cref{thm verify h is optimal}. 
		It remains to establish the reverse inequality.\smallskip
		
		Fix an arbitrary $(\pi,l)\in\mathcal{A}(x,y)$ and $\varepsilon>0$. We recall $X^{*,1}$, $X^\varepsilon$, $\pi^\varepsilon$ and $l^\varepsilon$ defined in Step 4 of \Cref{thm:verification}. Set
		\[  
		C_t:=l_tX^{\h{.5pt}\xi,\pi,l}_t,\quad
		C^{*,1}_t:=l^*(Y_t)X^{*,1}_t,\quad
		C^\varepsilon_t:=C_t+\varepsilon C^{*,1}_t.
		\]
		Since $X^{*,1}_t>0$, the perturbed wealth satisfies 
		$X^\varepsilon_t>0$ for all $t\ge0$, while $X^{\h{.5pt}\xi,\pi,l}_t/X^\varepsilon_t$ and $\varepsilon X^{*,1}_t/X^\varepsilon_t$ are well-defined and both lie in $[0,1]$. From \eqref{3817}, we obtain the bounds
		\[
		l^\varepsilon_t\le l_t+l^*(Y_t),\qquad\qquad
		\pig|(\pi^\varepsilon_t)^\top\mu(Y_t)\pig|
		\le\pig|\pi_t^\top\mu(Y_t)\pig|+\pig|(\pi^*(Y_t))^\top\mu(Y_t)\pig|,
		\]
		and
		\[
		(\pi^\varepsilon_t)^\top\Sigma(Y_t)\pi^\varepsilon_t
		\le
		2\,\pi_t^\top\Sigma(Y_t)\pi_t
		+2\,(\pi^*(Y_t))^\top\Sigma(Y_t)\pi^*(Y_t)
		\]
		confirm that $(\pi^\varepsilon,l^\varepsilon)\in\mathcal{A}(x+\varepsilon,y)$ by \Cref{thm:verification} and \Cref{thm verify h is optimal}.\smallskip
		
		\noindent\textbf{Step~1: A candidate supersolution.} Recall $\Phi^\varepsilon$ defined in \eqref{3850} and that $X^\varepsilon_t>0$. It\^{o}'s formula and the HJB equation give \eqref{3885}. The continuous local martingale $M^\varepsilon$ is given by
		\[
		M^\varepsilon_t
		:=\int_0^t
		X^\varepsilon_s\,\partial_x\mathcal{V}(s,X^\varepsilon_s,Y_s)\,
		(\pi^\varepsilon_s)^\top\sigma(Y_s)\,dZ_s
		+\int_0^t
		a(Y_s)\,\partial_y\mathcal{V}(s,X^\varepsilon_s,Y_s)\,dW_s.
		\] 
		Fix two bounded stopping times $\tau\le\zeta$. Define the localising sequence
		\[\theta_m:=\zeta\wedge\rho_m, \quad\text{where }
		\rho_m
		:=
		\inf\bigl\{t\ge\tau:
		\langle M^\varepsilon\rangle_t-\langle M^\varepsilon\rangle_\tau\ge m
		\bigr\}, 
		\]
		so that $\tau\le\theta_m\le\zeta$. The stopped increment 
		$N^m_t:=M^\varepsilon_{t\wedge\theta_m}-M^\varepsilon_{t\wedge\tau}$ 
		is a square-integrable martingale, since 
		$\langle N^m\rangle_t\le m$ for $t\geq 0$, and in particular 
		$\mathbb E_\tau(M^\varepsilon_{\theta_m}-M^\varepsilon_\tau)=0$. Integrating \eqref{3885} over $[\tau,\theta_m]$ and taking conditional 
		expectations then yields
		\begin{equation}
			\Phi^\varepsilon_\tau
			\ge
			\mathbb E_\tau\left[
			\Phi^\varepsilon_{\theta_m}
			+\int_\tau^{\theta_m}
			g_{EZ}\bigl(s,C^\varepsilon_s,\Phi^\varepsilon_s\bigr)\,ds
			\right].
			\label{4106}
		\end{equation}
		We now pass $m\to\infty$. Since $\partial_x\mathcal{V}>0$,
		\begin{equation}
			\Phi^\varepsilon_t
			=\mathcal{V}(t,X^{\h{.5pt}\xi,\pi,l}_t+\varepsilon X^{*,1}_t,Y_t)
			\ge\mathcal{V}(t,\varepsilon X^{*,1}_t,Y_t).
			\label{4113}
		\end{equation}
		By \Cref{thm verify h is optimal} applied with initial wealth 
		$\varepsilon$, the right-hand process is the utility process of the consumption $\varepsilon C^{*,1}$; it 
		provides the integrable lower bound required for the terminal terms 
		in \eqref{4106}. Since the aggregator is 
		one-signed, the integral terms converge by the conditional monotone 
		convergence. Letting $m\to\infty$ in \eqref{4106},
		\begin{equation}
			\Phi^\varepsilon_\tau
			\ge
			\mathbb E_\tau\left[
			\Phi^\varepsilon_\zeta
			+\int_\tau^\zeta
			\widetilde{g}_{EZ}\bigl(s,C^\varepsilon_s,\Phi^\varepsilon_s\bigr)\,ds
			\right].
			\label{4137}
		\end{equation}
		The substitution of $\widetilde{g}_{EZ}$ for $g_{EZ}$ is legitimate: 
		since $C^\varepsilon_s>0$ and $(1-\gamma)\Phi^\varepsilon_s>0$ for all 
		$s\ge0$, the pair $(C^\varepsilon_s,\Phi^\varepsilon_s)$ lies in the 
		interior of the domain of the aggregator, where the two 
		aggregators coincide.
		
		\smallskip
		\noindent\textbf{Step~2: Case $\gamma<1$.}
		We first handle the case $\gamma<1$. Since $1-\tfrac{1}{\psi}>0$, the generalised utility process 
		$\widetilde{J}^{\,\xi,\pi,l}$ takes values in $[0,\infty]$, and 
		$\mathcal{V}\ge0$. From \eqref{4113} we have $\Phi^\varepsilon_t\ge0$, 
		so $\liminf_{T\to\infty}\mathbb E[\Phi^\varepsilon_T]\ge0$. 
		Together with \eqref{4137}, this identifies $\Phi^\varepsilon$ as a 
		supersolution associated with $C^\varepsilon$. As 
		$C\le C^\varepsilon$ a.s. and $\widetilde{g}_{EZ}$ is increasing 
		in consumption, $\Phi^\varepsilon$ is also a supersolution 
		associated with $C$. \cite[Theorem~6.5]{herdegen2023infinite_II} yields that $\widetilde{J}^{\,\xi,\pi,l}$ is the minimal supersolution. Hence
		\[
		\widetilde{J}^{\,\xi,\pi,l}_0
		\le\Phi^\varepsilon_0
		=\frac{(x+\varepsilon)^{1-\gamma}}{1-\gamma}\,h(y)^\chi.
		\]
		
		\noindent\textbf{Step~3: Case \(\gamma>1\).} In this case, the generalised utility processes take values in
		\([-\infty,0]\), and the map $
		v\longmapsto \widetilde g_{EZ}(t,c,v)$ is nonincreasing. By the monotonicity of the generalised utility process in consumption, and since
		\(C\le C^\varepsilon\), we employ \cite[Proposition~6.8]{herdegen2023infinite_II} to obtain
		\begin{equation}
			\widetilde J^{\,\xi,\pi,l}_0
			\le
			\widetilde J^{\,C^\varepsilon}_0 ,
			\label{eq:monotonicity_C_Cepsilon}
		\end{equation}
		where \(\widetilde J^{\,C^\varepsilon}\) denotes the generalised utility process associated with \(C^\varepsilon\). We claim that
		\begin{equation*}
			\widetilde J^{\,C^\varepsilon}_0
			\le
			\Phi^\varepsilon_0 .
		\end{equation*}
		The value function $
		\mathcal V(t,x,y)
		=
		e^{-\delta\vartheta t}
		\frac{x^{1-\gamma}}{1-\gamma}
		h(y)^\chi \leq 0$ is increasing in \(x\). Therefore, using
		\(X^\varepsilon_t=X^{\h{.5pt}\xi,\pi,l}_t+\varepsilon X^{\h{.5pt}*,1}_t\ge \varepsilon X^{\h{.5pt}*,1}_t\), we get $0
		\le
		-\Phi^\varepsilon_t
		\le
		-\mathcal V(t,\varepsilon X^{\h{.5pt}*,1}_t,Y_t).$ By \Cref{thm verify h is optimal}, the process $\mathcal V(t,\varepsilon  X^{*,1}_t,Y_t)$ is the utility process associated with \(\varepsilon C^{*,1}\). Hence, by \eqref{3917}, it satisfies the
		transversality condition and hence  
		\begin{equation} 
			0
			\le-	\lim_{t\to\infty}
			\mathbb E[\Phi^\varepsilon_t]
			\le
			-\lim_{t\to\infty} \mathbb E\big[\mathcal V(t,\varepsilon X^{\h{.5pt}*,1}_t,Y_t)\big]
			=0.
			\label{eq:Phi_transversality_negative_case}
		\end{equation}
		
		Next, from \eqref{4137} and \eqref{eq:Phi_transversality_negative_case}, \(\Phi^\varepsilon\) is a supersolution associated with \((\widetilde g_{EZ},C^\varepsilon)\). On the other hand, by the maximal-subsolution characterisation in \cite[Theorem~6.5]{herdegen2023infinite_II}, the process
		\(\widetilde J^{\,C^\varepsilon}\) is a subsolution associated with
		\((\widetilde g_{EZ},C^\varepsilon)\). Suppose, for contradiction, that $
		\widetilde J^{\,C^\varepsilon}_0-\Phi^\varepsilon_0
		>0.$ Define the stopping time
		\[
		\iota
		:=
		\inf\left\{
		t\ge0:
		\widetilde J^{\,C^\varepsilon}_t-\Phi^\varepsilon_t\le0
		\right\}.
		\]
		On \([0,\iota)\), $
		\widetilde J^{\,C^\varepsilon}_t>\Phi^\varepsilon_t.$ Since \(v\mapsto\widetilde g_{EZ}(t,c,v)\) is nonincreasing, it follows that, for a.e. \(t<\iota\),
		\begin{equation}
			\widetilde g_{EZ}
			\bigl(t,C^\varepsilon_t,\widetilde J^{\,C^\varepsilon}_t\bigr)
			\le
			\widetilde g_{EZ}
			\bigl(t,C^\varepsilon_t,\Phi^\varepsilon_t\bigr).
			\label{eq:g_monotone_on_pre_iota}
		\end{equation}
		Strictly speaking, since \(\widetilde J^{\,C^\varepsilon}\) may be extended-valued,
		the following comparison is understood at the level of the finite-valued
		utility processes in the monotone approximation defining
		\(\widetilde J^{\,C^\varepsilon}\), followed by passage to the limit using
		\cite[Theorem~6.5]{herdegen2023infinite_II}. This avoids subtracting
		extended-valued quantities directly. Fix \(n\in\mathbb N\). Applying the subsolution inequality to \(\widetilde J^{\,C^\varepsilon}\) on
		\([0,\iota\wedge n]\), applying the supersolution inequality
		\eqref{4137} to \(\Phi^\varepsilon\) on the same interval,
		and using \eqref{eq:g_monotone_on_pre_iota}, we obtain
		\begin{align}
			\widetilde J^{\,C^\varepsilon}_0-\Phi^\varepsilon_0
			\le
			\mathbb E
			\left[
			\widetilde J^{\,C^\varepsilon}_{\iota\wedge n}
			-
			\Phi^\varepsilon_{\iota\wedge n}
			\right].
			\label{eq:comparison_until_iota_n}
		\end{align}
		Now decompose according to the events \(\{\iota\le n\}\) and \(\{\iota>n\}\). On
		\(\{\iota\le n\}\), by the definition of \(\iota\), $\widetilde J^{\,C^\varepsilon}_\iota
		-
		\Phi^\varepsilon_\iota
		\le0.$ On \(\{\iota>n\}\), since the generalised utility process takes values in
		\([-\infty,0]\), we have $\widetilde J^{\,C^\varepsilon}_n
		-
		\Phi^\varepsilon_n
		\le
		-\Phi^\varepsilon_n.$ Using these two inequalities in \eqref{eq:comparison_until_iota_n}, we get
		\begin{align*}
			\widetilde J^{\,C^\varepsilon}_0-\Phi^\varepsilon_0
			\le
			\mathbb E
			\left[
			\left(
			\widetilde J^{\,C^\varepsilon}_\iota
			-
			\Phi^\varepsilon_\iota
			\right)
			\mathbf 1_{\{\iota\le n\}}
			\right]
			+
			\mathbb E
			\left[
			\left(
			\widetilde J^{\,C^\varepsilon}_n
			-
			\Phi^\varepsilon_n
			\right)
			\mathbf 1_{\{\iota>n\}}
			\right]
			\le
			\mathbb E
			\left[
			-\Phi^\varepsilon_n\mathbf 1_{\{\iota>n\}}
			\right]
			\le
			\mathbb E[-\Phi^\varepsilon_n].
			\label{eq:d0_bound_by_minus_Phi}
		\end{align*}
		Letting \(n\to\infty\) and using
		\eqref{eq:Phi_transversality_negative_case}, we obtain a contradiction to \(\widetilde J^{\,C^\varepsilon}_0-\Phi^\varepsilon_0>0\). Combining this with \eqref{eq:monotonicity_C_Cepsilon}, we conclude that $\widetilde J^{\,\xi,\pi,l}_0
		\le
		\Phi^\varepsilon_0
		=
		\dfrac{(x+\varepsilon)^{1-\gamma}}{1-\gamma}h(y)^\chi.$ Taking $\varepsilon \to 0^+$, we complete the proof.\end{proof}

	\subsection{The regime \texorpdfstring{$\vartheta>1$}{theta>1}: proper utility}
	\label{sec:verification_theta_greater_one}
	
	Throughout this subsection, we assume that $\vartheta>1$. We shall use the positive coordinates of
	\cite[Section~5.1]{herdegen2025proper}.  For a consumption \(C \in \mathscr{P}([0,\infty))\) and the associated utility process \(J^C\) satisfying \eqref{eq:EZ_utility eq}, set
	\begin{equation*}
		\mathcal V^C:=(1-\gamma)J^C \geq 0,
		\qquad
		U_t^C:=\beta\vartheta e^{-\delta t}C_t^{1-1/\psi}.
	\end{equation*}
	When $C=lX^{\xi,\pi,l}$, we also write $\mathcal V^{\,\xi,\pi,l}=\mathcal V^{\,lX^{\xi,\pi,l}}$ and $J^{\,\xi,\pi,l}=J^{\,lX^{\xi,\pi,l}}$. We say $\mathcal V$ is the utility process associated with some process $U$ if solves the utility equation
	\begin{equation}
		\mathcal V_t=\mathbb E_t\left[\int_t^\infty U_s  (\mathcal V_s)^q\,ds\right],
		\quad  \text{for $t\ge0$ with} \quad \mathbb E\left[\int_0^\infty U_s  (\mathcal V_s)^q\,ds\right]<\infty.
		\label{eq:positive_recursion_theta_greater_one}
	\end{equation}  
	From \eqref{eq:EZ_utility eq}, $J^C$ is the utility process solving \eqref{eq:positive_recursion_theta_greater_one} associated with $U^C=\beta\vartheta e^{-\delta t}C_t^{1-1/\psi}$. Let $X^{*,1}$ be the wealth generated by the feedback controls in \eqref{eq:feedback_controls} from unit initial wealth with the solution $h$ of the HJB equation \eqref{eq: reduced HJB}. For initial wealth $x>0$, we set
	\begin{equation}
		\mathcal V^{*,1}_t:=e^{-\delta\vartheta t}
		(X^{*,1}_t)^{1-\gamma}h(Y_t)^\chi,
		\,\,\,
		U^{*,1}_t:=U^{ C^{*,1}}_t,
		\,\,\,
		C^{*,x}_t:=xl^*(Y_t)X^{*,1}_t,
		\,\,\,
		J_t^{*,x}:=e^{-\delta\vartheta t}
		\frac{(xX^{*,1}_t)^{1-\gamma}}{1-\gamma}h(Y_t)^\chi.
		\label{eq:candidate_positive_coordinates}
	\end{equation}
	We first show that the feedback controls lie in $\mathcal A^\dagger(x,y)$.
	
	\begin{lemma}[\bf Candidate unique proper utility]
		\label{lem:candidate_proper_theta_greater_one} Suppose that $\vartheta>1$, \Cref{assu. Qualitative Verification} holds and $h$ is a bounded positive solution of the HJB equation \eqref{eq: reduced HJB}. Then the corresponding feedback control $(\pi^*(Y),l^*(Y))$ in \eqref{eq:feedback_controls} belongs to
		$\mathcal A^\dagger(x,y)$, and $J^{*,x}$ in
		\eqref{eq:candidate_positive_coordinates} is the unique proper utility
		process solving \eqref{eq:positive_recursion_theta_greater_one} associated with $C^{*,x}$. In addition, with
		\begin{equation}
			\vartheta l^*(Y_t)\geq\underline\lambda:=\vartheta\beta^\psi (\sup h)^{p-1}>0,
			\label{eq:candidate_discount_rate}
		\end{equation}
		one has
		\begin{equation}
			\mathbb E_t\big(|J_T^{*,x}|\big)
			\le |J_t^{*,x}|e^{-\underline\lambda(T-t)}\qquad \text{for $0\le t\le T$.}
			\label{eq:candidate_exponential_tail}
		\end{equation}
	\end{lemma}
	
	\begin{proof}  As in Step~1 of the proof of \Cref{thm:verification}, we see that $(\pi^*(Y),l^*(Y))\in\mathcal A(x,y).$ Indeed, all the relevant coefficients are locally bounded along the path of \(Y\) on every finite time interval. Moreover, \(X^{*,1}\) and \(Y\) are continuous and \(l^*\) is continuous, so \(C^{*,x}\)
		is continuous, and hence right-continuous. It remains to prove that
		\(U^{C^{*,x}}\) admits a unique proper utility process solving \eqref{eq:positive_recursion_theta_greater_one}.\smallskip

		\noindent\textbf{Step~1: $J^{*,1}$ is a proper utility process.} The HJB calculation in \eqref{eq:aggregator_linear}-\eqref{eq:J_SDE} uses only the reduced HJB equation \eqref{eq: reduced HJB} and remains valid for
		$\vartheta>1$.  In the positive coordinates, it gives
		\begin{equation}
			d \mathcal V^{*,1}_t=-\vartheta l^*(Y_t) \mathcal V^{*,1}_t\,dt+ \mathcal V^{*,1}_t\,dN_t,
			\qquad
			U^{*,1}_t (\mathcal V^{*,1}_t)^q=\vartheta l^*(Y_t) \mathcal V^{*,1}_t,
			\label{eq:candidate_dynamics_theta_greater_one}
		\end{equation}
		where $N$ is the continuous local martingale in \eqref{2997}. We first verify that $\mathcal E(N)$ is a true martingale on every finite
		horizon.  Under the measure locally induced by $\mathcal E(N)$, the factor
		drift is
		\[
		\widetilde b(y)+a(y)^2\frac{h'(y)}{h(y)}.
		\]
		Normalise the scale function $\widetilde s$ in \eqref{eq. scale fucntion of shifted state variable} by
		$\widetilde s(c)=0$, and normalise the scale function $s_h$ of the tilted
		factor in the same way.  For $y>c$ and $y<c$, respectively,
		\[
		s_h(y)\ge\frac{h(c)^2}{\sup h^2}\widetilde s(y),
		\qquad
		s_h(y)\le\frac{h(c)^2}{\sup h^2}\widetilde s(y).
		\]
		\Cref{assu. Qualitative Verification} therefore implies
		\[
		\lim_{y\uparrow(\partial E)_+}s_h(y)=+\infty,
		\qquad
		\lim_{y\downarrow(\partial E)_-}s_h(y)=-\infty.
		\]
		Thus the tilted factor diffusion is conservative. \Cref{ass.market well-posed} and the local $C^2$-regularity and positivity of $h$ give local weak uniqueness; the proof of \Cref{thm verify h is optimal} then gives well-posedness of the full tilted martingale problem.  The compact-exhaustion
		argument in Step~2 of the proof of \Cref{thm:verification} consequently yields
		\begin{equation*}
			\mathbb E[\mathcal E(N)_T]=1,
			\qquad T<\infty.
		\end{equation*} 
		We denote by
		$\widehat{\mathbb P}$ the law of the conservative tilted martingale problem;
		then
		\begin{equation}
			\left.\frac{d\widehat{\mathbb P}}{d\mathbb P}
			\right|_{\mathcal F_T}=\mathcal E(N)_T,
			\qquad T<\infty.
			\label{eq:candidate_tilted_measure}
		\end{equation}
		We write $\widehat{\mathbb E}$ for expectation under
		$\widehat{\mathbb P}$.\smallskip
		
		By \eqref{eq:p_eta_def} and \eqref{eq:candidate_discount_rate}, and the fact that $p<1$, it follows that
		\begin{equation}
			\vartheta l^*(Y_t)=\vartheta\beta^\psi h(Y_t)^{p-1}
			\ge\vartheta\beta^\psi (\sup h)^{p-1}=\underline\lambda>0.
			\label{eq:candidate_rate_lower_bound}
		\end{equation}
		Solving the first equation in
		\eqref{eq:candidate_dynamics_theta_greater_one} and using Bayes' formula
		gives
		\begin{equation}
			\mathbb E_t[ \mathcal V^{*,1}_T]
			= \mathcal V^{*,1}_t\widehat{\mathbb E}_t\left[
			\exp\left(-\int_t^T\vartheta l^*(Y_s)\,ds\right)\right]
			\le \mathcal V^{*,1}_te^{-\underline\lambda(T-t)}.
			\label{eq:candidate_positive_exponential_tail}
		\end{equation}
		This proves \eqref{eq:candidate_exponential_tail} after scaling by
		$x^{1-\gamma}/|1-\gamma|$.\smallskip
		
		The second identity in
		\eqref{eq:candidate_dynamics_theta_greater_one}, conditional Tonelli's
		theorem, and Bayes' formula imply
		\begin{align}
			\mathbb E_t\left[\int_t^T
			U^{*,1}_s (\mathcal V^{*,1}_s)^q\,ds\right]
			&= \mathcal V^{*,1}_t\widehat{\mathbb E}_t\left[
			\int_t^T\vartheta l^*(Y_s)
			\exp\left(-\int_t^s\vartheta l^*(Y_u)\,du\right)ds\right]\notag\\
			&= \mathcal V^{*,1}_t\left\{1-
			\widehat{\mathbb E}_t\left[
			\exp\left(-\int_t^T\vartheta l^*(Y_s)\,ds\right)\right]\right\}.
			\label{eq:candidate_finite_horizon_aggregator}
		\end{align}
		Letting $T\to\infty$ in
		\eqref{eq:candidate_finite_horizon_aggregator} proves
		\eqref{eq:positive_recursion_theta_greater_one} for $ \mathcal V^{*,1}$, and taking
		$t=0$ proves the required absolute integrability.  Hence
		$J^{*,1}= \mathcal V^{*,1}/(1-\gamma)$ is a utility process solving \eqref{eq:EZ_utility eq} associated with $C^{*,1}$.  It is proper because
		$ \mathcal V^{*,1}$ and $ C^{*,1}$ are strictly positive.  Homogeneity gives the
		same conclusion for every $x>0$.\smallskip
		
		\noindent\textbf{Step~2: Uniqueness.} It remains to prove uniqueness among proper utility processes.  It is enough to treat $x=1$.  Let $\widetilde{\mathcal V}$ be the positive coordinate of another
		proper utility process solving \eqref{eq:positive_recursion_theta_greater_one} associated with $U^{C^{*,1}}$, and put
		\[
		Q_t:=\frac{\widetilde{\mathcal V}_t}{ \mathcal V^{*,1}_t}.
		\]
		Because the filtration is continuous,
		$\widetilde{\mathcal V}+\int_0^\cdot U^{*,1}_s\widetilde{\mathcal V}_s^qds$ is a continuous
		uniformly integrable martingale.  It\^o's formula and Girsanov's theorem
		therefore give, under $\widehat{\mathbb P}$,
		\begin{equation*}
			dQ_t=\vartheta l^*(Y_t)(Q_t-Q_t^q)\,dt+dL_t,
		\end{equation*}
		where $L$ is a continuous local $\widehat{\mathbb P}$-martingale. Define the uniformly integrable \(\mathbb P\)-martingale
		\[
		\widetilde M_t
		:=
		\widetilde{\mathcal V}_t
		+
		\int_0^t
		U^{*,1}_s
		(\widetilde{\mathcal V}_s)^q\,ds
		=
		\mathbb E_t\left[
		\int_0^\infty
		U^{*,1}_s
		(\widetilde{\mathcal V}_s)^q\,ds
		\right].
		\]
		To see
		the change of measure explicitly, if $\widetilde M$ denotes the preceding
		uniformly integrable martingale and
		\[
		K_t:=\int_0^t\frac{1}{ \mathcal V^{*,1}_s}\,d\widetilde M_s
		-\int_0^tQ_s\,dN_s,
		\]
		then the quotient formula under $\mathbb P$ gives
		\[
		dQ_t=\vartheta l^*(Y_t)(Q_t-Q_t^q)\,dt+dK_t-d\langle K,N\rangle_t.
		\]
		Thus $L:=K-\langle K,N\rangle$ is a local martingale under
		$\widehat{\mathbb P}$.\smallskip

		\noindent\textbf{Step~2A: $Q\ge1$.} We first show that \(Q\ge1\). Properness of
		\(\widetilde{\mathcal V}\) and the strict positivity of \(C^{*,1}\)
		imply that \(\widetilde{\mathcal V}_t>0\) and hence \(Q_t>0\), \(\mathbb P\)-almost surely,
		for every \(t\ge0\). By the equivalence of \(\mathbb P\) and
		\(\widehat{\mathbb P}\) on \(\mathcal F_t\), the same conclusion holds
		\(\widehat{\mathbb P}\)-almost surely.\smallskip

		Suppose that
		\(\widehat{\mathbb P}(Q_t<1)>0\) for some \(t\ge0\).
		Since \(1-q>0\), we write the event \(\{Q_t<1\}\) as
		\[
		\{Q_t<1\}
		=
		\bigcup_{n=2}^{\infty}
		\left\{
		1-Q_t^{1-q}\ge\frac1n
		\right\},
		\qquad \widehat{\mathbb P}\text{-a.s.}
		\]
		Consequently, there exists \(\alpha\in(0,1)\) such that
		\[
		B:=\{1-Q_t^{1-q}\ge\alpha\}\in\mathcal F_t,
		\qquad
		\widehat{\mathbb P}(B)>0.
		\]
		On the conditional probability space $\bigl(
		\Omega,\mathcal F,\{\mathcal F_s\}_{s\ge t},
		\widehat{\mathbb P}(\,\cdot\mid B)
		\bigr)$, 
		we define
		\[
		\sigma:=\inf\{s\ge t:Q_s\notin(0,1)\},
		\qquad
		H_s:=1-Q_{s\wedge\sigma}^{1-q},
		\qquad s\ge t.
		\]	
		Then $\alpha\le H_t<1$ and $0\le H_s\le1$ for $s\ge t$. Since \(B\in\mathcal F_t\), the local-martingale property of the
		increments after \(t\) is preserved under conditioning on \(B\).
		Applying It\^o's formula, after localisation away from the endpoints
		of \((0,1)\), gives, for \(t\le s<\sigma\),
		\[
		dH_s
		=
		(1-q)\vartheta l^*(Y_s)H_s\,ds
		+
		\frac{q(1-q)}{2}Q_s^{-q-1}\,d\langle L\rangle_s
		+
		d\widetilde L_s,
		\]
		where \(\widetilde L_t=0\) and $
		d\widetilde L_s
		=
		(q-1)Q_s^{-q}\,dL_s$ locally on the stochastic interval \([t,\sigma)\). Consequently, by the same localisation and passage to the limit,
		\begin{equation}
			G_s
			:=
			H_s-H_t
			-
			\int_t^{s\wedge\sigma}
			(1-q)\vartheta l^*(Y_r)H_r\,dr,
			\qquad s\ge t,
			\label{eq:ratio_below_one_submartingale}
		\end{equation}
		is a local submartingale under the conditional law. Indeed, before
		\(\sigma\), the finite-variation term of $G$ is $\frac{q(1-q)}{2}
		\int_t^\cdot Q_s^{-q-1}\,d\langle L\rangle_s,$ which is nondecreasing because \(q\in(0,1)\).\smallskip
		
		Since \(0\le H_s\le1\) and the integral in
		\eqref{eq:ratio_below_one_submartingale} is nonnegative, $G_s\le1-H_t\le1-\alpha.$ Thus, being a local submartingale bounded from above by a constant, \(G\) is a true submartingale under the conditional law. By continuity, if \(Q\) exits \((0,1)\) through \(0\), then \(H\)
		reaches \(1\); if \(Q\) exits through \(1\), then \(H\) is stopped at
		\(0\). Therefore, until \(H\) reaches \(1\),
		\begin{equation}
			H_s-H_t
			\ge
			(1-q)\underline\lambda
			\int_t^s H_r\,dr
			+
			G_s-G_t,
			\qquad
			t\le s\le\inf\{u\ge t:H_u=1\}.
			\label{eq:ratio_below_one_hitting_inequality}
		\end{equation}
		
		We now adapt the proof of
		\cite[Lemma~6.5]{herdegen2025proper}. Fix
		\(\alpha'\in(0,\alpha)\) and define
		\[
		\tau:=\inf\{s\ge t:H_s\notin(\alpha',1)\}.
		\]
		The proof of \cite[Lemma~6.5]{herdegen2025proper} uses the initial value only through a
		strictly positive lower bound. Applying optional sampling to
		\eqref{eq:ratio_below_one_hitting_inequality} with $\tau \wedge (t+n)$ under the conditional law
		\(\widehat{\mathbb P}(\,\cdot\mid B)\), and then sending $n \to \infty$, gives $\widehat{\mathbb E}[\tau-t \,|\, B]
		\le
		\frac{1-\alpha}
		{(1-q)\underline\lambda\,\alpha'}
		<\infty.$ Thus, \(\tau<\infty\) almost surely under the conditional law. Moreover,
		by continuity, \(H_\tau\in\{\alpha',1\}\), and the same optional-sampling
		argument yields
		\[
		\widehat{\mathbb P}(H_\tau=1\mid B)
		\ge
		\frac{\alpha-\alpha'}{1-\alpha'}
		>0.
		\]
		Consequently, there exists a deterministic time \(T>t\) such that
		\[
		\widehat{\mathbb P}
		\bigl(\tau\le T,\ H_\tau=1\mid B\bigr)>0.
		\]
		By the definition of \(H\), the event \(\{H_\tau=1\}\) implies
		\(Q_\tau=0\), and hence
		\(\widetilde{\mathcal V}_\tau=0\). Since
		\(\widehat{\mathbb P}(B)>0\), it
		also has positive \(\mathbb P\)-probability as \(\mathbb P\) and \(\widehat{\mathbb P}\) are equivalent on
		\(\mathcal F_T\) by \eqref{eq:candidate_tilted_measure}.\smallskip

		Let $\tau_0:=\inf\{s\ge t:\widetilde{\mathcal V}_s=0\}$ and choose
		a deterministic $T>t$ such that $\mathbb P(\tau_0\le T)>0$.  Optional
		sampling at the bounded stopping time $\tau_0\wedge T$ gives, on
		$\{\tau_0\le T\}$,
		\[
		0=\widetilde{\mathcal V}_{\tau_0\wedge T}
		=\mathbb E_{\tau_0\wedge T}\left[\widetilde{\mathcal V}_T+
		\int_{\tau_0\wedge T}^T U^{*,1}_s
		\widetilde{\mathcal V}_s^q\,ds\right].
		\]
		Both terms inside the conditional expectation are nonnegative, so
		$\widetilde{\mathcal V}_T=0$ on an event of positive probability.  This contradicts
		properness, and hence $Q_t\ge1$ for every deterministic $t$. \smallskip
		
		\noindent\textbf{Step~2B: $Q\leq 1$.} We next rule out strict inequality.  Fix $t\ge0$ and define, for $u\ge0$,
		\begin{equation*}
			T_u^t:=\inf\left\{s\ge t:
			\int_t^s\vartheta l^*(Y_r)\,dr\ge u\right\}.
		\end{equation*}
		By \eqref{eq:candidate_rate_lower_bound},
		$T_u^t\le t+u/\underline\lambda$.
		Moreover, $T_u^t\uparrow\infty$ as $u \to \infty$: continuity of $Y$ and local continuity of
		$l^*$ make the integral of $\vartheta l^*(Y)$ finite on every finite interval.
		Optional sampling of the utility equation of $\widetilde{\mathcal V}$ in \eqref{eq:positive_recursion_theta_greater_one} at $T_u^t$, followed by Bayes'
		formula and the pathwise change of variables
		$v=\int_t^s\vartheta l^*(Y_r)dr$, gives
		\begin{equation}
			Q_t=\widehat{\mathbb E}_t\left[
			e^{-u}Q_{T_u^t}+\int_0^u e^{-v}Q_{T_v^t}^q\,dv\right].
			\label{eq:proper_ratio_clock_identity}
		\end{equation}
		Under a regular conditional law given $\mathcal F_t$, set
		$m(u):=\widehat{\mathbb E}_t[Q_{T_u^t}]$. Conditional Tonelli's theorem and
		\eqref{eq:proper_ratio_clock_identity} give
		\begin{equation*}
			e^{-u}m(u)=Q_t-\int_0^u e^{-v}
			\widehat{\mathbb E}_t[Q_{T_v^t}^q]\,dv
			\le Q_t<\infty.
		\end{equation*}
		Thus $m$ is locally absolutely continuous, and for almost every $u$, conditional Jensen's inequality implies that $m'(u)=m(u)-\widehat{\mathbb E}_t[Q_{T_u^t}^q]
		\ge m(u)-m(u)^q$. Hence
		\begin{equation}
			m(u)\ge\left[1+
			\big(Q_t^{1-q}-1\big)e^{(1-q)u}\right]^{1/(1-q)}.
			\label{eq:proper_ratio_clock_ode_bound}
		\end{equation}
		On the other hand, Bayes' formula gives
		\begin{equation}
			\mathbb E_t[\widetilde{\mathcal V}_{T_u^t}]
			= \mathcal V^{*,1}_te^{-u}m(u).
			\label{eq:proper_ratio_terminal_clock}
		\end{equation}
		
		If $\mathbb P(Q_t>1)>0$, there exist $\varepsilon>0$ and an event
		$B':=\{Q_t\ge1+\varepsilon\}\in\mathcal F_t$ of positive probability. By
		\eqref{eq:proper_ratio_clock_ode_bound},
		\[
		e^{-u}m(u)
		\ge\left[e^{-(1-q)u}+Q_t^{1-q}-1\right]^{1/(1-q)}
		\ge\left[(1+\varepsilon)^{1-q}-1\right]^{1/(1-q)}
		\quad\text{on }B'.
		\]
		Together with \eqref{eq:proper_ratio_terminal_clock}, this yields
		\[
		\inf_{u\ge0}\mathbb E\left[
		\mathbf{1}_{B'}\widetilde{\mathcal V}_{T_u^t}\right]
		\ge
		\left[(1+\varepsilon)^{1-q}-1\right]^{1/(1-q)}
		\mathbb E[\mathbf{1}_{B'} \mathcal V^{*,1}_t]>0.
		\]
		This is impossible because the utility equation of $\widetilde{\mathcal{V}}$ at the bounded stopping
		time $T_u^t$ gives
		\[
		\mathbb E\left[\mathbf{1}_{B'}\widetilde{\mathcal V}_{T_u^t}\right]
		=\mathbb E\left[\mathbf{1}_{B'}\int_{T_u^t}^\infty
		U^{*,1}_s\widetilde{\mathcal V}_s^q\,ds\right]\longrightarrow0\qquad
		\text{as $u \to \infty$}
		\]
		by dominated convergence, since the total aggregator is integrable and
		$T_u^t\uparrow\infty$ as $u \to \infty$. Thus \(Q_t\le1\), almost surely, for every deterministic \(t\).
		Applying this conclusion at rational times and using the continuity of
		\(Q\), we obtain $
		Q_t\le1$ for $t\ge0$ almost surely.
	\end{proof}
	
	\subsubsection{Comparison principle}
	
	The following construction is the incomplete market counterpart of the
	perturbation used in \cite[Theorem~8.1]{herdegen2023infinite_II}
	and \cite[Lemma~E.6]{herdegen2025proper}.  We include the argument because
	the wealth perturbation, localisation, and terminal estimate all have to be
	checked under the present unbounded factor coefficients. Fix $(\pi,l)\in\mathcal A^\dagger(x,y)$, and we continue to use the notations in the proof of \Cref{thm:verification}: $X_t^\varepsilon:=X^{\xi,\pi,l}_t+\varepsilon X^{*,1}_t$, $C_t^\varepsilon:=C^{\xi,\pi,l}_t+\varepsilon C^{*,1}_t$, and $\Phi_t^\varepsilon:=e^{-\delta\vartheta t} \frac{(X_t^\varepsilon)^{1-\gamma}}{1-\gamma}h(Y_t)^\chi$ for $\varepsilon>0$. 
	
	\begin{lemma}[\bf Perturbed HJB supersolution]
		\label{lem:perturbed_supersolution_theta_greater_one}Suppose that $\vartheta>1$, \Cref{assu. Qualitative Verification} holds and $h$ is a bounded positive solution of the HJB equation \eqref{eq: reduced HJB}. Fix $(\pi,l)\in\mathcal A^\dagger(x,y)$. For every $\varepsilon>0$ and all bounded stopping times
		$0\le\tau\le\zeta$, we have
		\begin{equation}
			\Phi_\tau^\varepsilon\ge
			\mathbb E_\tau\left[\Phi_\zeta^\varepsilon+
			\int_\tau^\zeta
			g_{EZ}(s,C_s^\varepsilon,\Phi_s^\varepsilon)\,ds\right]\qquad\text{and}\qquad
			\liminf_{t\to\infty}\mathbb E[\Phi_t^\varepsilon]\ge0.
			\label{eq:perturbed_supersolution_inequality}
		\end{equation} 
		Thus $\Phi^\varepsilon$ is a supersolution for $C^\varepsilon$ and also for $C^{\xi,\pi,l}$.  Moreover, it holds that $\varepsilon^{\gamma-1}\displaystyle\sup_{t\geq0}\mathbb E\bigl[|\Phi_t^\varepsilon|\bigr]<\infty$.
	\end{lemma}

	\begin{proof} The It\^o and localisation arguments are the same as those in Step~4 of the proof of \Cref{thm:verification}. Indeed, these arguments use only the HJB
		equation, and the local integrability of the dynamics of \(X^\varepsilon\), and does not use the restriction \(0<\vartheta<1\). It therefore gives the localised version of
		\eqref{eq:perturbed_supersolution_inequality}.\smallskip
		
		It remains to justify the removal of the localisation and the
		transversality condition in the present regime. Since
		\(X_t^\varepsilon\ge\varepsilon X^{*,1}_t\), we have
		\begin{equation}
			\Phi_t^\varepsilon
			\ge
			e^{-\delta\vartheta t}
			\frac{(\varepsilon X^{*,1}_t)^{1-\gamma}}{1-\gamma}
			h(Y_t)^\chi
			=
			\varepsilon^{1-\gamma}J_t^{*,1}.
			\label{eq:perturbed_candidate_lower_bound}
		\end{equation}
		We first verify the uniform integrability needed to remove the
		localisation. For every bounded stopping time \(\tau\), the candidate
		utility equation in \eqref{eq:EZ_utility eq} gives
		\[
		|J_\tau^{*,1}|
		\le
		\mathbb E_\tau\left[
		\int_0^\infty
		\left|
		g_{EZ}(s,C_s^{*,1},J_s^{*,1})
		\right|\,ds
		\right].
		\]
		The random variable inside the conditional expectation is integrable by
		Lemma~\ref{lem:candidate_proper_theta_greater_one}. Therefore, the family
		of random variables on the right-hand side, indexed by bounded stopping
		times \(\tau\), is uniformly integrable.\smallskip
		
		If \(\gamma<1\), then the localised terminal values
		\(\Phi^\varepsilon\) are nonnegative. If \(\gamma>1\), then
		\eqref{eq:perturbed_candidate_lower_bound} and
		\(\Phi^\varepsilon\le0\) give $0
		\le
		-\Phi_\tau^\varepsilon
		\le
		\varepsilon^{1-\gamma}|J_\tau^{*,1}|.$ Thus, in both cases, conditional Fatou's lemma applies to the localised
		terminal values. Since the aggregator is one-signed, the integral terms
		converge by conditional monotone convergence. We may therefore remove
		the localisation and obtain
		\eqref{eq:perturbed_supersolution_inequality}.\smallskip
		
		It remains to verify the transversality condition in \eqref{eq:perturbed_supersolution_inequality}. If \(\gamma<1\), then
		\(\Phi_t^\varepsilon\ge0\) and the result is obvious. If \(\gamma>1\), then 	\eqref{eq:perturbed_candidate_lower_bound} and Lemma~\ref{lem:candidate_proper_theta_greater_one} yield $\lim_{t\to\infty}
		\mathbb E[\Phi_t^\varepsilon]
		=0$. \smallskip

		Finally, \(g_{EZ}\) is increasing in consumption. Since
		\(C^\varepsilon\ge C^{\xi,\pi,l}\), the supersolution inequality for
		\(C^\varepsilon\) also implies the corresponding inequality for \(C^{\xi,\pi,l}\). Inequality \eqref{eq:perturbed_candidate_lower_bound} implies that
		\[
		\sup_{t\ge0}\mathbb E\bigl[|\Phi_t^\varepsilon|\bigr]
		\le
		\varepsilon^{1-\gamma}
		\mathbb E\left[
		\int_0^\infty
		\left|
		g_{EZ}(s,C_s^{*,1},J_s^{*,1})
		\right|\,ds
		\right]
		<\infty.
		\]  
	\end{proof}
	
	Define
	\begin{equation}
		\Lambda_t:=\int_0^t\vartheta l^*(Y_s)\,ds,
		\qquad
		T_u:=\inf\{t\ge0:\Lambda_t\ge u\},\qquad u\ge0.
		\label{eq:candidate_discount_clock}
	\end{equation}
	By \eqref{eq:candidate_rate_lower_bound}, $T_u\le u/\underline\lambda$; moreover,
	$T_u\uparrow\infty$ as $u \to \infty$ by Step 2B of the proof of \Cref{lem:candidate_proper_theta_greater_one}.

	\begin{lemma}[\bf Comparison principle]
		\label{lem:time_change_comparison} Suppose that $\vartheta>1$, \Cref{assu. Qualitative Verification} holds and $h$ is a bounded positive solution of the HJB equation \eqref{eq: reduced HJB}. Let \(U\in\mathscr{P}([0,\infty))\), \(\mathcal V\) be the utility process solving \eqref{eq:positive_recursion_theta_greater_one} associated with $U$, and \(S\) be a strictly positive supersolution of \eqref{eq:positive_recursion_theta_greater_one} associated with $U$, in the sense of \cite[Definition B.2]{herdegen2025proper}. Suppose that there exist constants
		\(K,c>0\) such that
		\begin{equation*}
			qU_tS_t^{q-1}
			\le
			K \vartheta l^*(Y_t) e^{-c\Lambda_t},
			\qquad t\ge0.
		\end{equation*}
		Then $
		\mathcal V_t\le S_t$, for $t\ge0$, $\mathbb P$-a.s. 
	\end{lemma}
	
	\begin{proof}	Let \(\sigma\le\tau\) be bounded stopping times with respect to
		\(\{\mathcal F_{T_u}\}_{u\ge0}\). Then \(T_\sigma\le T_\tau\)
		are bounded stopping times with respect to
		\(\{\mathcal F_t\}_{t\ge0}\). The change of variables \(u=\Lambda_s\), for which
		\(du=\vartheta l^*(Y_s)\,ds\) and $s=T_{\Lambda_s}=T_u$, gives 
		\[
		\mathcal V_{T_\sigma}
		=
		\mathbb E\left[
		\mathcal V_{T_\tau}
		+
		\int_\sigma^\tau
		\frac{U_{T_u}}{\vartheta l^*(Y_{T_u})} \mathcal V_{T_u}^q\,du
		\,\middle|\,
		\mathcal F_{T_\sigma}
		\right]\quad\text{and}\quad
		S_{T_\sigma}
		\ge
		\mathbb E\left[
		S_{T_\tau}
		+
		\int_\sigma^\tau
		\frac{U_{T_u}}{\vartheta l^*(Y_{T_u})} S_{T_u}^q\,du
		\,\middle|\,
		\mathcal F_{T_\sigma}
		\right].
		\]
		Since \(T_u\to\infty\) as $u \to \infty$, the corresponding transversality conditions also hold.\smallskip

		We next verify the integrability condition required for comparison.
		For every \(u\ge0\),
		\[
		0\le\mathcal V_{T_u}
		=
		\mathbb E\left[
		\int_{T_u}^\infty
		U_s\mathcal V_s^q\,ds
		\,\middle|\,
		\mathcal F_{T_u}
		\right]
		\le
		\mathbb E\left[
		\int_0^\infty
		U_s\mathcal V_s^q\,ds
		\,\middle|\,
		\mathcal F_{T_u}
		\right].
		\]
		The random variable $\int_0^\infty U_s\mathcal V_s^q\,ds$ is integrable because \(\mathcal V\) is a utility process.
		Consequently, \(\{\mathcal V_{T_u}\}_{u\ge0}\) is uniformly
		integrable. Moreover, the same change of variables gives
		\[
		\mathbb E\left[
		\int_0^\infty
		\frac{U_{T_u}}{\vartheta l^*(Y_{T_u})}\mathcal V_{T_u}^q\,du
		\right]
		=
		\mathbb E\left[
		\int_0^\infty
		U_s\mathcal V_s^q\,ds
		\right]
		<\infty.
		\]
		Furthermore, since $0
		\le
		(\mathcal V_{T_u}-S_{T_u})^+
		\le
		\mathcal V_{T_u}$, we have $\sup_{u\ge0}
		\mathbb E\left[
		(\mathcal V_{T_u}-S_{T_u})^+
		\right]
		<\infty.$ Thus, \cite[Condition~B.4]{herdegen2025proper} is satisfied for $\{\mathcal V_{T_u}\}_{u \geq 0}$ and $\{S_{T_u}\}_{u \geq 0}$.\smallskip
		
		The aggregator $
		\overline g(u,w):= \dfrac{U_{T_u}}{\vartheta l^*(Y_{T_u})} w^q$ is nondecreasing and concave in \(w\), since \(q\in(0,1)\). Its
		derivative with respect to $w$ at \(S_{T_u}\) satisfies
		\[
		\partial_w\overline g(u,S_{T_u})
		=
		q \frac{U_{T_u}}{\vartheta l^*(Y_{T_u})} S_{T_u}^{q-1}
		\le
		Ke^{-c\Lambda_{T_u}}
		=
		Ke^{-cu}.
		\]
		All the assumptions of
		\cite[Proposition~B.6]{herdegen2025proper} are therefore satisfied,
		and that proposition gives $\mathcal V_{T_\iota}\le S_{T_\iota}$, for any finite stopping time $\iota$ with respect to the
		time-changed filtration
		\(\mathcal F_{T_u}\). Taking \(\iota=\Lambda_t\) and using
		\(T_{\Lambda_t}=t\), we obtain the lemma.\end{proof}
	
	We next establish existence, properness, and monotonicity for processes
	dominated by \(U^{*,1}\).
	
	\begin{lemma}[\bf Existence of maximal proper solution]
		\label{lem:maximal_proper_solution_candidate_domination} Suppose that $\vartheta>1$, \Cref{assu. Qualitative Verification} holds and $h$ is a bounded positive solution of the HJB equation \eqref{eq: reduced HJB}. Let \(U \in \mathscr{P}([0,\infty))\) be such that 
		\begin{equation}
			\text{(1). it is right-continuous},\qquad
			\text{(2). there is \(K<\infty\) such that $U_t\le K U_t^{*,1}$ for $t\ge0$, $\mathbb P$-a.s.} 
			\label{eq:candidate_domination}
		\end{equation}
		Then the utility equation \eqref{eq:positive_recursion_theta_greater_one} associated with $U$ admits a maximal
		solution \(\overline{\mathcal V}\): if \(\mathcal V\) is any other solution, then $ \mathcal V_t\le\overline{\mathcal V}_t$ for $t\ge0$, $\mathbb P$-a.s. Moreover, \(\overline{\mathcal V}\) is proper. Finally, if \(0\leq U^1\le U^2\) and both processes satisfy
		\eqref{eq:candidate_domination}, then their maximal solutions satisfy $
		\overline{\mathcal V}^{\,1}_t
		\le
		\overline{\mathcal V}^{\,2}_t$ for $t\ge0$, $\mathbb P$-a.s.
	\end{lemma}
	
	\begin{proof}
		Set $k:=\max\{1,K^\vartheta\}.$ Fix $0<\nu<1-q$ and $\zeta>0,$ and define
		\[
		S_t^\zeta
		:=
		k\bigl(1+\zeta e^{\nu\Lambda_t}\bigr)
		\mathcal V_t^{*,1}.
		\]
		\noindent\textbf{Step 1. $S^\zeta$ is a supersolution:}
		We first prove that $S^\zeta$ is a supersolution for the utility equation
		\eqref{eq:positive_recursion_theta_greater_one} associated with $U$ in the sense of \cite[Definition B.2]{herdegen2025proper}. Since \(q\in(0,1)\), we have $(1+r)^q\le1+qr$ for $r\geq 0$. Because \(q<1-\nu\), it follows that
		\begin{align}
			\bigl(1+\zeta e^{\nu\Lambda_t}\bigr)^q
			\le
			1+q\zeta e^{\nu\Lambda_t}
			\le
			1+(1-\nu)\zeta e^{\nu\Lambda_t}.
			\label{53243}
		\end{align}
		From \eqref{eq:J_SDE} and \(d\Lambda_t=\vartheta l^*(Y_t)\,dt\), the product rule gives
		\[
		dS_t^\zeta
		=
		-k\vartheta l^*(Y_t)\mathcal V_t^{*,1}
		\bigl[
		1+(1-\nu)\zeta e^{\nu\Lambda_t}
		\bigr]\,dt
		+
		S_t^\zeta\,dN_t.
		\]
		On the other hand, since \(Kk^{q-1}\le1\), we use \eqref{53243}, \eqref{eq:candidate_domination}, and $U_t^{*,1}(\mathcal V_t^{*,1})^q=\vartheta l^*(Y_t)\mathcal V_t^{*,1}$ to obtain
		\begin{align}
			U_t(S_t^\zeta)^q
			\le
			Kk^q U_t^{*,1}(\mathcal V_t^{*,1})^q
			\bigl(1+\zeta e^{\nu\Lambda_t}\bigr)^q
			\le
			k\vartheta l^*(Y_t)\mathcal V_t^{*,1}
			\bigl[
			1+(1-\nu)\zeta e^{\nu\Lambda_t}
			\bigr].\label{5347}
		\end{align}
		Consequently,
		\begin{equation}
			dS_t^\zeta
			\le
			-U_t(S_t^\zeta)^q\,dt
			+
			S_t^\zeta\,dN_t.
			\label{eq:perturbed_candidate_differential_inequality}
		\end{equation}
		
		As the candidate process $\mathcal{V}^{*,1}$ satisfies \eqref{eq:J_SDE}, we have $S_t^\zeta
		=
		k\mathcal V_0^{*,1}\mathcal E(N)_t
		\left[
		e^{-\Lambda_t}
		+
		\zeta e^{-(1-\nu)\Lambda_t}
		\right].$ Therefore, by Bayes' formula and
		\(\Lambda_t\ge\underline\lambda t\),
		\[
		\mathbb E[S_t^\zeta]
		\le
		k\mathcal V_0^{*,1}
		\left[
		e^{-\underline\lambda t}
		+
		\zeta e^{-(1-\nu)\underline\lambda t}
		\right]
		\longrightarrow0 \qquad \text{as $t \to \infty$}.
		\]
		Furthermore, we use the definition of $S^\zeta$ and \eqref{eq:candidate_domination} to obtain
		\begin{align*}
			\mathbb E\left[
			\int_0^\infty U_t(S_t^\zeta)^q\,dt
			\right]
			\le
			Kk^q\mathcal V_0^{*,1}
			\int_0^\infty
			e^{-u}(1+\zeta e^{\nu u})^q\,du
			<\infty,
		\end{align*}
		where the substitution \(u=\Lambda_t\) is used under
		\(\widehat{\mathbb P}\). The last integral is finite because
		\(\nu q<1\). Integrating
		\eqref{eq:perturbed_candidate_differential_inequality} between
		bounded stopping times, after localising its local-martingale term,
		and then applying conditional Fatou's lemma proves that
		\(S^\zeta\) is a supersolution for the integral equation
		\eqref{eq:positive_recursion_theta_greater_one}.\smallskip

		\noindent\textbf{Step 2. Maximal solution:} 
		We now construct the maximal solution. Set $\mathcal V^0:=k\mathcal V^{*,1}$ and define successively
		\begin{equation}
			\mathcal V_t^{m+1}
			:=
			\mathbb E_t\left[
			\int_t^\infty
			U_s(\mathcal V_s^m)^q\,ds
			\right],
			\qquad m\ge0.
			\label{eq:monotone_iteration_maximal_solution}
		\end{equation}
		Since \(1-q=1/\vartheta\), we have
		\(Kk^q\le k\). Using \eqref{eq:candidate_domination} and \Cref{lem:candidate_proper_theta_greater_one}, we obtain
		\begin{align}
			\mathbb E_t\left[
			\int_t^\infty
			U_s(k\mathcal V_s^{*,1})^q\,ds
			\right]
			\le
			Kk^q
			\mathbb E_t\left[
			\int_t^\infty
			U_s^{*,1}(\mathcal V_s^{*,1})^q\,ds
			\right]  
			\le
			k\mathcal V_t^{*,1}.
			\label{eq:multiple_candidate_supersolution}
		\end{align} 
		By \eqref{eq:multiple_candidate_supersolution},
		\(\mathcal V^1\le\mathcal V^0\). Since \(w\mapsto w^q\) is
		increasing, induction gives
		\[
		0\le\mathcal V^{m+1}\le\mathcal V^m
		\le k\mathcal V^{*,1}.
		\]
		Hence, $
		U_s(\mathcal V_s^m)^q
		\le
		U_s(k\mathcal V_s^{*,1})^q,$ and the process on the right is integrable by
		\eqref{eq:multiple_candidate_supersolution}. Conditional dominated
		convergence therefore shows that
		\[
		\overline{\mathcal V}_t
		:=
		\lim_{m\to\infty}\mathcal V_t^m
		\]
		satisfies \eqref{eq:positive_recursion_theta_greater_one}. The conditional
		expectation on the right-hand side of that equation provides a
		càdlàg version of \(\overline{\mathcal V}\).\smallskip
		
		Let \(\mathcal V\) be any other solution of \eqref{eq:positive_recursion_theta_greater_one}. Moreover, similarly to \eqref{5347}, we have $	qU_t(S_t^\zeta)^{q-1}
		\le
		qKk^{q-1}\vartheta l^*(Y_t)
		\bigl(1+\zeta e^{\nu\Lambda_t}\bigr)^{q-1}
		\le
		q\zeta^{q-1}\vartheta l^*(Y_t)
		e^{-(1-q)\nu\Lambda_t}.$ 	Lemma~\ref{lem:time_change_comparison}, applied with
		\(c=(1-q)\nu\), therefore gives $
		\mathcal V_t\le S_t^\zeta$ for every solution \(\mathcal V\) of
		\eqref{eq:positive_recursion_theta_greater_one}. Letting
		\(\zeta\downarrow0\) yields
		\begin{equation}
			\mathcal V_t
			\le
			k\mathcal V_t^{*,1},
			\qquad t\ge0.
			\label{eq:all_solutions_candidate_bound}
		\end{equation}  
		Then
		\eqref{eq:all_solutions_candidate_bound} gives
		\(\mathcal V\le\mathcal V^0\). Using
		\eqref{eq:positive_recursion_theta_greater_one} and
		\eqref{eq:monotone_iteration_maximal_solution} inductively yields
		\[
		\mathcal V\le\mathcal V^m.
		\]
		Letting \(m\to\infty\), we obtain
		\(\mathcal V\le\overline{\mathcal V}\). Hence
		\(\overline{\mathcal V}\) is maximal.\smallskip
		
		\noindent\textbf{Step 3. Monotonicity:} We next prove monotonicity. Let \(0\leq U^1\le U^2\), where
		\(U^i\le K_iU^{*,1}\), and choose
		\[
		\Gamma\ge\max\{1,K_1^\vartheta,K_2^\vartheta\}.
		\]
		Starting the iteration
		\eqref{eq:monotone_iteration_maximal_solution} for \(U^2\) from
		\(\Gamma\mathcal V^{*,1}\), let us denote the resulting sequence by
		\(\{\mathcal V^{2,m}\}_{m\ge0}\). Since
		\(\overline{\mathcal V}^{\,1}\le \Gamma\mathcal V^{*,1}\), we have
		\[
		\overline{\mathcal V}^{\,1}_t
		=
		\mathbb E_t\left[
		\int_t^\infty
		U_s^1(\overline{\mathcal V}_s^{\,1})^q\,ds
		\right]
		\le
		\mathbb E_t\left[
		\int_t^\infty
		U_s^2(\overline{\mathcal V}_s^{\,1})^q\,ds
		\right]
		\le
		\mathcal V^{2,1}_t.
		\]
		Induction gives
		\(\overline{\mathcal V}^{\,1}\le\mathcal V^{2,m}\) for every \(m\). The same dominated-convergence argument as in Step~2 shows that
		\(\displaystyle\lim_{m\to\infty}\mathcal V^{2,m}\) is a solution associated with \(U^2\). Since
		\(\overline{\mathcal V}^{\,2}\) is maximal, we have $\displaystyle\lim_{m\to\infty}\mathcal V^{2,m}\le\overline{\mathcal V}^{\,2}.$
		Conversely, since
		\[
		\overline{\mathcal V}^{\,2}
		\le
		\max\{1,K_2^\vartheta\}\mathcal V^{*,1}
		\le
		\Gamma\mathcal V^{*,1},
		\]
		the utility equation \eqref{eq:positive_recursion_theta_greater_one} and induction give
		\(\overline{\mathcal V}^{\,2}\le\mathcal V^{2,m}\) for every \(m\).
		Hence,
		\(\overline{\mathcal V}^{\,2}\le \displaystyle\lim_{m\to\infty}\mathcal V^{2,m}\), and therefore
		\(\displaystyle\lim_{m\to\infty}\mathcal V^{2,m}=\overline{\mathcal V}^{\,2}\). Letting \(m\to\infty\) in
		\(\overline{\mathcal V}^{\,1}\le\mathcal V^{2,m}\) now yields $\overline{\mathcal V}^{\,1}
		\le
		\overline{\mathcal V}^{\,2}$. \smallskip
		
		\noindent\textbf{Step 4. Properness:} It remains to prove properness. Fix \(t\ge0\). By
		\cite[Lemma~6.3]{herdegen2025proper},
		\[
		\mathbb P\left(
		\left\{
		\mathbb E_t\left[
		\int_t^\infty U_s^\vartheta\,ds
		\right]>0
		\right\}
		\setminus
		\bigcup_{T,\varepsilon}B_{T,\varepsilon}
		\right)
		=0,
		\]
		where $B_{T,\varepsilon}
		:=
		\left\{
		\mathbb E_t
		\left[
		\mathbf 1_{\{U_T\ge\varepsilon\}}
		\right]>0
		\right\},$ and the union is taken over $T\in\mathbb Q\cap[t,\infty)$ and $\varepsilon\in\mathbb Q\cap(0,\infty).$ Fix such \(T\) and \(\varepsilon\), and define
		\[ 
		\varrho
		:=
		\inf\{s\ge T:U_s\le\varepsilon/2\}.
		\]
		The right-continuity of \(U\) implies that $\varrho>T$ on $\{U_T\ge\varepsilon\}$. Define
		\[
		\widetilde U_s
		:=
		\frac{\varepsilon}{2}e^{-s}
		\mathbb E\left[
		\mathbf 1_{\{U_T\ge\varepsilon\}}
		\mid\mathcal F_s
		\right]
		\mathbf 1_{\{T\le s<\varrho\}}.
		\]
		On \(\{T\le s<\varrho\}\), we have \(U_s>\varepsilon/2\), while $0\le
		\widetilde U_s
		\le
		\varepsilon/2.$ Thus \(\widetilde U\le U\). The construction in the proof of
		\cite[Theorem~4.7]{herdegen2025proper} yields a solution
		\(\widetilde{\mathcal V}\) associated with \(\widetilde U\) such
		that
		\begin{align}
			\widetilde{\mathcal V}_t>0
			\quad\text{on }B_{T,\varepsilon}.
			\label{5543}
		\end{align}
		Since $\widetilde U\le U\le K U^{*,1}$, the preceding construction gives maximal solutions
		\(\overline{\mathcal V}^{\,\widetilde U}\) and
		\(\overline{\mathcal V}\) associated with $\widetilde U$ and $U$, respectively. By maximality and the monotonicity
		already proved, $
		\widetilde{\mathcal V}
		\le
		\overline{\mathcal V}^{\,\widetilde U}
		\le
		\overline{\mathcal V}.$ Therefore, \eqref{5543} implies $\overline{\mathcal V}_t>0$ on $B_{T,\varepsilon}$. Taking the countable union over \(T\) and \(\varepsilon\) gives
		\[
		\mathbb E_t\left[
		\int_t^\infty U_s^\vartheta\,ds
		\right]>0
		\quad\Longrightarrow\quad
		\overline{\mathcal V}_t>0,
		\qquad\mathbb P\text{-a.s.}
		\]
	\end{proof}
	
	\subsubsection{Verification}

	For \(\gamma>1\), define
	\begin{align*}
		&\mathcal A^\dagger_{\mathrm{ext}}(x,y)
		:=
		\left\{
		(\pi,l)\in\mathcal A^\dagger(x,y):\begin{aligned}
			& \mathbb P\pig(
			\mathcal V_t
			\le
			(1-\gamma)J_t^{\xi,\pi,l}
			\text{ for all }t\ge0
			\pig)=1 \text{ for every subsolution}\\[-1mm]
			& \text{\(\mathcal V\geq 0\) associated with \(U^{\xi,\pi,l}\) with $\sup_{t\ge0}\mathbb E[\mathcal V_t]<\infty$}
		\end{aligned}
		\right\},
		\\
		&\mathcal A^\dagger_{\mathrm{env}}(x,y)
		:=
		\Bigl\{
		(\pi,l)\in\mathcal A^\dagger(x,y):
		\text{\(\exists\, K>0\) such that }
		\mathbb P\pig(
		U_t^{\xi,\pi,l}\!
		\le
		K U_t^{*,1}
		\text{ for all }t\ge0
		\pig)=1
		\Bigr\}.
	\end{align*}
	The subsolutions in the definition of
	$\mathcal A^\dagger_{\mathrm{ext}}(x,y)$ are understood in
	the sense of \cite[Appendix~B]{herdegen2025proper}.  Finally, set
	\begin{equation}
		\mathcal A^\dagger_{\mathrm{ver}}(x,y):=
		\begin{cases}
			\mathcal A^\dagger(x,y),&\gamma<1,\\
			\mathcal A^\dagger_{\mathrm{ext}}(x,y)\cup
			\mathcal A^\dagger_{\mathrm{env}}(x,y),&\gamma>1.
		\end{cases}
		\label{eq:verification_class_theta_greater_one}
	\end{equation}
	
	\begin{theorem}[\bf Verification for $\vartheta>1$]
		\label{thm:verification_theta_greater_one}
		Suppose that $\vartheta>1$, \Cref{assu. Qualitative Verification} holds, and $h$ is a bounded positive solution of the HJB equation \eqref{eq: reduced HJB}. For every $x>0$ and
		$y\in E$, we have $
		\displaystyle\sup_{(\pi,l)\in\mathcal A^\dagger_{\mathrm{ver}}(x,y)}
		J_0^{\xi,\pi,l}
		=\frac{x^{1-\gamma}}{1-\gamma}h(y)^\chi$ and the supremum is attained by the feedback controls in \eqref{eq:feedback_controls}. 
	\end{theorem}
	
	\begin{remark}
		If, in addition, \Cref{ass:kappa} holds, then $h$ can be chosen from (i) or (ii) in \Cref{thm verify h is optimal}. Note that \Cref{prop:pass_R_to_infty E=R,prop:pass_R_to_infty} hold for any $\vartheta>0$.
	\end{remark}
	
	\begin{proof}[Proof of \Cref{thm:verification_theta_greater_one}] By \Cref{lem:candidate_proper_theta_greater_one} and Step 1 of the proof of \Cref{thm:verification}, we see that $	\bigl(\pi^*(Y),l^*(Y)\bigr)\in\mathcal A^\dagger(x,y)$. If \(\gamma>1\), homogeneity gives $
		C^{*,x}=xC^{*,1},$ and $U^{C^{*,x}}
		=
		x^{1-1/\psi}U^{*,1}.$ Therefore, the candidate belongs to
		\(\mathcal A^\dagger_{\mathrm{env}}(x,y)\), and hence also to
		\(\mathcal A^\dagger_{\mathrm{ver}}(x,y)\). Set
		\begin{equation*}
			X^*:=X^{\xi,\pi^*,l^*},\qquad	\Psi^\varepsilon:=(1-\gamma)\Phi^\varepsilon=e^{-\delta\vartheta t}\,(X^\varepsilon)^{1-\gamma}\,h(Y)^\chi,
			\qquad
			\mathcal V^{\xi,\pi,l}:=(1-\gamma)J^{\xi,\pi,l}.
		\end{equation*}
		Now, we take $(\pi,l) \in \mathcal A^\dagger_{\mathrm{ver}}(x,y)$.\smallskip
		
		\noindent\textbf{Case 1. \(\gamma>1\):} Suppose first that \(\gamma>1\). By
		\Cref{lem:perturbed_supersolution_theta_greater_one}, for every
		\(\varepsilon>0\) and all bounded stopping times
		\(0\le\tau\le\zeta\),
		\begin{align}
			\Phi_\tau^\varepsilon
			\ge
			\mathbb E_\tau\left[
			\Phi_\zeta^\varepsilon
			+
			\int_\tau^\zeta
			g_{EZ}(s,C_s^\varepsilon,\Phi_s^\varepsilon)\,ds
			\right]
			\ge
			\mathbb E_\tau\left[
			\Phi_\zeta^\varepsilon
			+
			\int_\tau^\zeta
			g_{EZ}(s,C^{\xi,\pi,l}_s,\Phi_s^\varepsilon)\,ds
			\right]
			\label{5656}
		\end{align}
		since \(C^\varepsilon\ge C^{\xi,\pi,l}\) and \(g_{EZ}\) is increasing in its consumption argument. Moreover, \Cref{lem:perturbed_supersolution_theta_greater_one}
		implies that \(\Psi^\varepsilon\) is bounded in \(L^1\) and that
		\(\lim_{t\to\infty}\mathbb E[\Psi_t^\varepsilon]=0\).
		Consequently, multiplying \eqref{5656} by \(1-\gamma<0\) shows that \(\Psi^\varepsilon\) is a nonnegative
		\(L^1\)-bounded subsolution associated with
		\(U^{\xi,\pi,l}\), in the sense of
		\cite[Appendix~B]{herdegen2025proper}. If
		\((\pi,l)\in\mathcal A^\dagger_{\mathrm{ext}}(x,y)\), we may therefore
		apply its defining property, which yields
		\begin{equation}	\label{eq:gamma_greater_extremal_comparison}
			\Psi_t^\varepsilon
			\le
			\mathcal V_t^{\xi,\pi,l},
			\qquad t\ge0,
		\end{equation}

		Now let $(\pi,l)\in\mathcal A^\dagger_{\mathrm{env}}(x,y)$ and choose $K$ such
		that $U^{\xi,\pi,l}\le K U^{*,1}$. \Cref{lem:maximal_proper_solution_candidate_domination} supplies a maximal proper solution for
		$U^{\xi,\pi,l}$; unique properness identifies it with $\mathcal V^{\xi,\pi,l}$.  Since $1-\gamma<0$ and
		$X^\varepsilon\ge\varepsilon X^{*,1}$,
		\begin{equation}
			0\le\Psi_t^\varepsilon
			\le\varepsilon^{1-\gamma} \mathcal V^{*,1}_t
			=\varepsilon^{1-\gamma} (1-\gamma)J^{*,1}.
			\label{eq:Psi_upper_bound_gamma_greater_one}
		\end{equation}
		Choosing
		$\Gamma\ge\max\{1,K^\vartheta,\varepsilon^{1-\gamma}\}$, we using the defining property $ U_s^{\xi,\pi,l} 
		\le
		KU_s^{*,1}$ and $K\Gamma^q
		\le
		\Gamma$ to obtain
		\begin{align}
			U_s^{\xi,\pi,l}
			\bigl(\Gamma\mathcal V_s^{*,1}\bigr)^q
			=
			\Gamma^qU_s^{\xi,\pi,l}
			\bigl(\mathcal V_s^{*,1}\bigr)^q
			\le
			K\Gamma^qU_s^{*,1}
			\bigl(\mathcal V_s^{*,1}\bigr)^q
			\le
			\Gamma U_s^{*,1}
			\bigl(\mathcal V_s^{*,1}\bigr)^q.
			\label{eq:scaled_candidate_aggregator_bound}
		\end{align}
		For all bounded stopping times \(0\le\tau\le\zeta\), we use
		\eqref{eq:scaled_candidate_aggregator_bound} and utility equation \eqref{eq:positive_recursion_theta_greater_one} to obtain
		\[
		\Gamma\mathcal V_\tau^{*,1}
		=
		\mathbb E_\tau\left[
		\Gamma\mathcal V_\zeta^{*,1}
		+
		\int_\tau^\zeta
		\Gamma U_s^{*,1}
		\bigl(\mathcal V_s^{*,1}\bigr)^q\,ds
		\right]
		\ge
		\mathbb E_\tau\left[
		\Gamma\mathcal V_\zeta^{*,1}
		+
		\int_\tau^\zeta
		U_s^{\xi,\pi,l}
		\bigl(\Gamma\mathcal V_s^{*,1}\bigr)^q\,ds
		\right].
		\]
		Hence
		$\Gamma \mathcal V^{*,1}$ is a supersolution for $U^{\xi,\pi,l}$.  The downward
		iteration from $\Gamma \mathcal V^{*,1}$ has a solution as its limit, by the
		argument in Step 2 of the proof of Lemma~\ref{lem:maximal_proper_solution_candidate_domination}. This limit is the maximal solution and hence it is $\mathcal V^{\xi,\pi,l}$ by the uniqueness of the utility process since $(\pi,l) \in \mathcal A^\dagger_{\mathrm{ver}}(x,y)$. By \eqref{eq:Psi_upper_bound_gamma_greater_one}, we have
		\[
		U_t^{\xi,\pi,l}(\Psi_t^\varepsilon)^q
		\le K\varepsilon^{q(1-\gamma)}
		U^{*,1}_t (\mathcal V^{*,1}_t)^q
		\]whose integral has finite expectation by
		\eqref{eq:candidate_finite_horizon_aggregator}.  In addition,
		\eqref{eq:Psi_upper_bound_gamma_greater_one} and
		\eqref{eq:candidate_positive_exponential_tail} imply
		$\mathbb E[\Psi_T^\varepsilon]\to0$. Letting $T\to\infty$ in the subsolution inequality in \eqref{5656}, using conditional monotone
		convergence for the aggregator and $L^1$-convergence of the terminal term therefore gives
		\[
		\Psi_t^\varepsilon
		\le\mathbb E_t\left[\int_t^\infty
		U_s^{\xi,\pi,l}(\Psi_s^\varepsilon)^q\,ds\right].
		\]
		Together with \eqref{eq:Psi_upper_bound_gamma_greater_one} and similar arguments as in Step 2 of the proof of Lemma~\ref{lem:maximal_proper_solution_candidate_domination}, this shows inductively that
		$\Psi^\varepsilon$ lies below every iterate and hence below $\mathcal V^{\xi,\pi,l}$. Thus
		\eqref{eq:gamma_greater_extremal_comparison} holds in both
		classes $\mathcal A^\dagger_{\mathrm{ext}}(x,y)$ and $\mathcal A^\dagger_{\mathrm{env}}(x,y)$.  Dividing by $1-\gamma<0$ yields
		\[
		J_0^{\xi,\pi,l}
		\le\Phi_0^\varepsilon
		=\frac{(x+\varepsilon)^{1-\gamma}}{1-\gamma}h(y)^\chi.
		\]
		Letting $\varepsilon\downarrow0$ proves the required upper bound when
		$\gamma>1$.\smallskip 
		
		\noindent\textbf{Case 2. \(\gamma<1\):}	Suppose now that $\gamma<1$.  Then $1-1/\psi>0$.  Fix
		$\nu\in(0,1-q)$ and define the strictly positive right-continuous reference
		stream
		\begin{equation*}
			C_t^0:=\exp\left(-\frac{\nu \Lambda_t}{1-1/\psi}\right) C^{*,1}_t.
		\end{equation*}
		It satisfies $
		U_t^{C^0}=e^{-\nu \Lambda_t} U^{*,1}_t.$ We consider $C^n:=C^{\xi,\pi,l}\wedge (nC^0) \uparrow C^{\xi,\pi,l}$ as $n \to \infty$ and
		\begin{equation}
			U_t^{C^n}
			\le n^{1-1/\psi}e^{-\nu \Lambda_t} U^{*,1}_t
			\le n^{1-1/\psi} U^{*,1}_t.
			\label{eq:truncated_candidate_envelope}
		\end{equation}
		Lemma~\ref{lem:maximal_proper_solution_candidate_domination} gives a maximal proper solution
		$\mathcal V^n$ for $C^n$. Since $1-\gamma>0$, it holds that $\Psi_t^\varepsilon
		=e^{-\delta\vartheta t}(X_t+\varepsilon X^{*,1}_t)^{1-\gamma}
		h(Y_t)^\chi
		\ge\varepsilon^{1-\gamma} \mathcal V^{*,1}_t.$ Hence, together with $q-1=-1/\vartheta$,
		$(1-\gamma)/\vartheta=1-1/\psi$ and \eqref{eq:truncated_candidate_envelope}, we obtain
		$$
		qU_t^{C^n}(\Psi_t^\varepsilon)^{q-1}
		\le q\left(\frac n\varepsilon\right)^{1-1/\psi}
		\vartheta l^*(Y_t) e^{-\nu \Lambda_t}.$$
		Hence, we apply Lemma~\ref{lem:time_change_comparison} to the utility process
		$\mathcal V^n$ and the positive supersolution $\Psi^\varepsilon$ to get
		\begin{equation}
			\mathcal V^n\le\Psi^\varepsilon.
			\label{eq:truncated_utility_upper_bound}
		\end{equation}
		The maximal solutions $\mathcal V^n$ increase with $n$ by Lemma~\ref{lem:maximal_proper_solution_candidate_domination}.
		By \eqref{eq:truncated_utility_upper_bound}, the limit
		$\mathcal V^\infty:=\lim_n\mathcal V^n$ is finite.  Conditional monotone convergence gives
		\begin{equation*}
			\mathcal V_t^\infty=\mathbb E_t\left[\int_t^\infty
			U_s^{\xi,\pi,l}(\mathcal V_s^\infty)^q\,ds\right].
		\end{equation*}
		Here the c\`adl\`ag version given by the right-hand side is understood.
		Because $C^0$ is strictly positive, the events of nontrivial future
		consumption for $C^{\xi,\pi,l}\wedge C^0$ and for $C^{\xi,\pi,l}$ coincide. Thus, since
		$\mathcal V^\infty\ge \mathcal V^1$ and $\mathcal V^1$ is proper, $\mathcal V^\infty$ is a proper solution for
		$U^{\xi,\pi,l}$. Finally, because
		\((\pi,l)\in\mathcal A^\dagger(x,y)\),
		\(U^{\xi,\pi,l}\) admits a unique proper utility process solving \eqref{eq:proper_utility_theta_greater_one}. Therefore, it yields $\mathcal V^\infty=\mathcal V^{\xi,\pi,l}$.  Letting
		$n\to\infty$ in \eqref{eq:truncated_utility_upper_bound} gives
		$\mathcal V^{\xi,\pi,l}\le\Psi^\varepsilon$, and hence $J_0^{\xi,\pi,l}
		\le\Phi_0^\varepsilon
		=\frac{(x+\varepsilon)^{1-\gamma}}{1-\gamma}h(y)^\chi.$ Sending $\varepsilon\downarrow0$ completes the proof.  
	\end{proof}

	\section{Examples and numerical results}\label{sec. Examples and numerical results}
	In this section, we illustrate three models, namely, the mean reverting risk premium model, Heston-type stochastic-volatility model, and CIR stochastic interest rates model. Explicit solutions are constructed for the first and second models, in some parameter regimes. All these three models are restricted to a single risky asset ($n=1$). The Brownian motions $\{Z_t\}_{t\geq0}$ and $\{W_t\}_{t\geq0}$ have constant instantaneous correlation $\rho\in[-1,1]$. In the numerical simulation, Assumptions \ref{ass.market well-posed}, \ref{ass. Parameters constraint}, \ref{ass:kappa}, \ref{ass:polynomial_growth} (resp. \ref{ass poly growth R_+}), and \ref{assu. Qualitative Verification} are satisfied by the example in \Cref{sec. Mean reverting risk premium} (resp. \Cref{sec. Heston-type stochastic-volatility,sec. CIR stochastic interest rates}). We also restrict to $\vartheta \in (0,1)$. \smallskip 
	
	We compute the optimal strategies for these three examples by computing the bounded positive solution $h$ of the reduced HJB equation \eqref{eq: reduced HJB}. We approximate the solution by restricting the equation to a sufficiently large bounded domain, with the boundary conditions designed in \Cref{sec. Robin-Dirichlet boundary problem,sec. Dirichlet boundary problem}. Error estimates of explicit solutions and the solutions computed by this scheme are also provided.

	\subsection{Mean reverting risk premium}\label{sec. Mean reverting risk premium}
	
	Let $E=\mathbb{R}$. We consider the Kim--Omberg model \cite{KE96}:
	\begin{equation} 
		\frac{dS^0_t}{S^0_t} = r \, dt,\qquad
		\frac{dS^1_t}{S^1_t} = \pig(r+\underbrace{\sigma Y_t}_{\mu(\cdot)} \pig) dt + \sigma\, dZ_t,\qquad
		dY_t = \underbrace{k(\theta-Y_t)}_{b(\cdot)} \, dt + a\, dW_t,
		\label{eq kim omberg}
	\end{equation} 
	where $r$, $\sigma$, $k$, $\theta$ and $a$ are positive constants. The risk premium $Y_t$ of the risky asset is a mean reverting Ornstein-Uhlenbeck process. \smallskip

	\paragraph{Explicit solution.} We first construct explicit solutions to the problem under the market \eqref{eq kim omberg}. This also justifies the accuracy of our numerical scheme. We impose the tractability restriction 
	\begin{align}
		\psi=1+\frac{1-\gamma}{\chi}=2-\gamma+\frac{(1-\gamma)^2}{\gamma}\rho^2\in(0,1) ,\qquad \rho \in (0,1),\qquad \delta >0,\qquad\text{and} \qquad\gamma>1.
		\label{4677}
	\end{align}
	Note this is only used in constructing explicit solutions and is not needed for the  subsequent numerical experiments. The first equality in \eqref{4677} is also used in \cite[Remark 4.3]{kraft2017optimal}. The reduced HJB equation \eqref{eq: reduced HJB} therefore becomes the linear inhomogeneous equation
	\begin{equation}
		\frac{a^2}{2}h''(y)+\left\{k\theta-\left[k+\Big(1-\frac1\gamma\Big)a\rho\right]y\right\}h'(y)-(K_0+K_2y^2)h(y)+\beta^\psi=0,
		\label{eq reduced ode}
	\end{equation}
	where $\kappa(y)=K_0+K_2y^2>0$, $K_0:=\frac{\gamma-1}{\chi}\left(r-\frac{\delta\psi}{\psi-1}\right)>0,$ and $K_2:=\frac{\gamma-1}{2\gamma\chi}>0$. Set 
	\begin{align}
		\begin{alignedat}{3}
			K_1 &:= k+\Bigl(1-\frac{1}{\gamma}\Bigr)a\rho>0,
			&\quad I(y) &:= \frac{k\theta}{a^2}y-\frac{K_1}{2a^2}y^2,
			&\quad m &:= \sqrt{\frac{2K_2}{a^2}+\frac{K_1^2}{a^4}}, \\[0.6em]
			y_* &:= \frac{k\theta K_1}{a^4m^2}>0,
			&\quad \nu &:= \frac{K_1}{a^2}
			-\frac{2K_0}{a^2}
			-\frac{k^2\theta^2}{a^4}
			+m^2y_*^2,
			&\quad \alpha &:= \frac{\nu}{2m}-\frac{1}{2}<0.
		\end{alignedat}
		\label{4703}
	\end{align}
	Writing $D_\alpha$ for the parabolic cylinder function \cite[Section 8.2]{HigherTranscendental} (solutions of the Weber differential equation) of index $\alpha$, define
	\begin{align}
		\phi_1(y):=e^{-I(y)}D_\alpha\big(\sqrt{2m}(y-y_*)\big),
		\qquad
		\phi_2(y):=e^{-I(y)}D_\alpha\big(-\sqrt{2m}(y-y_*)\big),
		\label{4705}
	\end{align}
	the two linearly independent solutions of the homogeneous equation associated with \eqref{eq reduced ode}, with Wronskian $\mathcal{W}(y):=\phi_1(y)\phi_2'(y)-\phi_1'(y)\phi_2(y)$. For a fixed $y_0 \in \mathbb{R}$, the general solution of \eqref{eq reduced ode} is then
	\begin{equation}
		h(y)
		=
		C_1\phi_1(y)+C_2\phi_2(y)
		+
		\frac{2\beta^\psi}{a^2}
		\left[
		\phi_1(y)\int_{y_0}^{y}\frac{\phi_2(z)}{\mathcal{W}(z)}\,dz
		-
		\phi_2(y)\int_{y_0}^{y}\frac{\phi_1(z)}{\mathcal{W}(z)}\,dz
		\right],
		\label{eq:general_variation_solution}
	\end{equation}
	for some $C_1$ and $C_2\in\mathbb{R}$. We can uniquely identify the solution when it is bounded and positive.  \smallskip
	
	\begin{proposition}[\bf Explicit bounded positive solution of \eqref{eq reduced ode}]
		\label{prop explicit +ve bdd sol} Suppose that \eqref{4677} holds. Let \(\phi_1,\phi_2\) be defined in \eqref{4705} and let $
		\mathcal{W}(y):=\phi_1(y)\phi_2'(y)-\phi_1'(y)\phi_2(y)$ be their Wronskian. The solution of \eqref{eq reduced ode} must have the form \eqref{eq:general_variation_solution}. Moreover, the bounded positive solution of \eqref{eq reduced ode} is unique and given by
		\begin{equation}
			h(y)
			=
			\frac{2\beta^\psi}{a^2}
			\left[
			\phi_1(y)\int_{-\infty}^{y}\frac{\phi_2(z)}{\mathcal{W}(z)}\,dz
			+
			\phi_2(y)\int_{y}^{\infty}\frac{\phi_1(z)}{\mathcal{W}(z)}\,dz
			\right].
			\label{eq:full_line_green_solution}
		\end{equation}
		Equivalently, the constants in \eqref{eq:general_variation_solution} are $
		C_1
		=
		\displaystyle\frac{2\beta^\psi}{a^2}
		\int_{-\infty}^{y_0}\frac{\phi_2(z)}{\mathcal{W}(z)}\,dz$ and $
		C_2
		=
		\displaystyle\frac{2\beta^\psi}{a^2}
		\int_{y_0}^{\infty}\frac{\phi_1(z)}{\mathcal{W}(z)}\,dz.$
	\end{proposition}
	
	\begin{proof}  \noindent\textbf{Part 1. General solutions:} We show that the homogeneous equation associated with \eqref{eq reduced ode} reduces to the Weber equation. Set $\phi(y)=e^{-I(y)}v(y)$ for some function $v$. Since $k\theta-K_1y=a^2I'(y)$, the first-order term of the homogeneous equation associated with \eqref{eq reduced ode} is eliminated:
		\[
		\frac{a^2}{2}v''(y)
		+
		\left[
		-\frac{a^2}{2}\bigl(I'(y)\bigr)^2
		-\frac{a^2}{2}I''(y)
		-K_0
		-K_2y^2
		\right]v(y)
		=0.
		\]
		Dividing by $a^2/2$ and then completing the square in $y$, the equation takes the form $
		v''(y)+\left[\nu-m^2(y-y_*)^2\right]v(y)=0.$ The substitution $x=\sqrt{2m}(y-y_*)$ transforms the equation of $v$ into the equation of $\widetilde{v}(x):=v(x/\sqrt{2m}+y_*)$:
		\begin{align}
			\widetilde{v}'' 
			+
			\left(
			\alpha+\frac{1}{2}-\frac{x^2}{4}
			\right)\widetilde{v}=
			\widetilde{v}'' 
			+
			\left(
			\frac{\nu}{2m}-\frac{x^2}{4}
			\right)\widetilde{v}
			=0
			\label{5897}
		\end{align}
		which is precisely the Weber equation. By \cite[Section 8.2]{HigherTranscendental}, the parabolic cylinder function \(D_\alpha(x)\) of index $\alpha$ solves the above equation. Since the coefficient of \eqref{5897} is even in \(x\), \(D_\alpha(-x)\) also solves it. Therefore, returning to the original variable \(y\), the two functions
		\begin{align}
			\phi_1(y)
			=
			e^{-I(y)}D_\alpha\left(\sqrt{2m}(y-y_*)\right),
			\qquad
			\phi_2(y)
			=
			e^{-I(y)}D_\alpha\left(-\sqrt{2m}(y-y_*)\right)
			\label{4820}
		\end{align}
		solve the homogeneous equation associated with \eqref{eq reduced ode}. Since \(\alpha<0\), \(\phi_1,\phi_2\) form a fundamental system. By the variation-of-constants formula for second-order linear ODEs, we can show that \eqref{eq:general_variation_solution} gives the general solution of the inhomogeneous equation \eqref{eq reduced ode}.	\smallskip

		\noindent\textbf{Part 2. Unique bounded positive solution:}  As \(y\to\infty\), \(\phi_2 (y)\) is exponentially growing by \cite[(1.7)-(1.8)]{temme2000numerical}, while \(\phi_1\) goes to $0$ exponentially by \cite[(1.8)]{temme2000numerical}. We have
		\begin{align}
			\phi_1(y) \to 0,
			\qquad
			\phi_2(y) \to \infty,
			\qquad
			\frac{\phi_1(y)}{\phi_2(y)}\to0,
			\qquad \text{as } y\to\infty .
			\label{4834}
		\end{align}
		Write the solution in the form $
		h(y)=A(y)\phi_1(y)+B(y)\phi_2(y),$ where
		\begin{align}
			A(y):=
			C_1+\frac{2\beta^\psi}{a^2}
			\int_{y_0}^{y}\frac{\phi_2(z)}{\mathcal{W}(z)}\,dz,
			\qquad
			B(y):=
			C_2-\frac{2\beta^\psi}{a^2}
			\int_{y_0}^{y}\frac{\phi_1(z)}{\mathcal{W}(z)}\,dz.
			\label{4849}
		\end{align}
		We first determine the coefficients $C_1$ and $C_2$ such that $h$ is bounded and then prove that this bounded solution is unique and positive.\smallskip
		
		\noindent\textbf{Part 2A. Choices of $C_1$ and $C_2$ if $h$ is bounded:} Suppose that $h$ is bounded. We justify that
		\begin{align}
			A(y)\frac{\phi_1(y)}{\phi_2(y)}\to0\qquad\text{as $y \to \infty$}.
			\label{4855}
		\end{align}
		By Abel's identity, $
		\mathcal{W}(y)=\mathcal{W}_0 e^{-2I(y)}$ for some nonzero constant \(\mathcal{W}_0\). Hence, by \eqref{4820}, we have
		\begin{align}
			\frac{\phi_2(y)}{\mathcal{W}(y)}
			=
			\frac{1}{\mathcal{W}_0}
			e^{I(y)}D_\alpha(-x),\qquad
			\frac{\phi_1(y)}{\mathcal{W}(y)}
			=
			\frac{1}{\mathcal{W}_0}
			e^{I(y)}D_\alpha(x),
			\qquad
			\text{for } x=\sqrt{2m}(y-y_*).
			\label{eq:phi_over_W_asymptotic_start}
		\end{align}
		Since \(\alpha<0\), the asymptotic expansions in
		\cite[(1.7)--(1.8)]{temme2000numerical} give
		\begin{align}
			D_\alpha(x)
			=
			e^{-x^2/4}x^\alpha
			\left(1+O(x^{-2})\right),\qquad
			D_\alpha(-x)
			=
			\frac{\sqrt{2\pi}}{\Gamma(-\alpha)}
			e^{x^2/4}x^{-\alpha-1}
			\left(1+O(x^{-2})\right)
			\label{eq:Dalpha_positive_asymptotic}
		\end{align}
		as \(x\to\infty\), where $\Gamma$ is the Gamma function. Consequently,
		\begin{align}
			\frac{\phi_1(y)}{\phi_2(y)}
			&=
			\frac{\Gamma(-\alpha)}{\sqrt{2\pi}}
			e^{-x^2/2}x^{2\alpha+1}
			\left(1+O(x^{-2})\right)
			\longrightarrow 0
			\qquad\text{as }y\to\infty.
			\label{eq:phi1_phi2_ratio_asymptotic}
		\end{align} 
		Moreover, \eqref{eq:phi_over_W_asymptotic_start} and
		\eqref{eq:Dalpha_positive_asymptotic} imply
		\begin{align}
			\frac{\phi_2(y)}{\mathcal{W}(y)}
			&=
			\frac{\sqrt{2\pi}}
			{\mathcal{W}_0\Gamma(-\alpha)}
			e^{I(y)+x^2/4}x^{-\alpha-1}
			\left(1+O(x^{-2})\right).
			\label{eq:phi2_W_precise_asymptotic}
		\end{align}
		Since $I(y)+\frac{x^2}{4}
		=
		\frac12
		\left(m-\frac{K_1}{a^2}\right)y^2
		+
		\left(\frac{k\theta}{a^2}-my_*\right)y
		+
		\frac{m}{2}y_*^2,$ and \(m>K_1/a^2\), the right-hand side of
		\eqref{eq:phi2_W_precise_asymptotic} grows exponentially in
		absolute value. In particular, $
		\left|
		\int_{y_0}^{y}
		\frac{\phi_2(z)}{\mathcal{W}(z)}\,dz
		\right|
		\longrightarrow\infty$ as $y\to\infty.$ Applying \eqref{eq:phi2_W_precise_asymptotic} yields
		\begin{align}
			\lim_{y\to\infty}
			\frac{
				\int_{y_0}^{y}
				\frac{\phi_2(z)}{\mathcal{W}(z)}\,dz
			}{
				\displaystyle
				e^{I(y)+x^2/4}x^{-\alpha-2}
			}
			=
			\lim_{y\to\infty}
			\frac{
				\frac{\phi_2(y)}{\mathcal{W}(y)}
			}{
				e^{I(y)+x^2/4}x^{-\alpha-2}
				\left[
				I'(y)
				+\frac{x \sqrt{2m}}{2}
				-(\alpha+2)\frac{\sqrt{2m}}{x}
				\right]
			}
			=
			\frac{\sqrt{2\pi}}
			{\mathcal{W}_0\Gamma(-\alpha)}
			\frac{\sqrt{2m}}
			{m-K_1/a^2}.
			\label{eq:integral_phi2_W_asymptotic}
		\end{align}
		It follows from
		\eqref{eq:integral_phi2_W_asymptotic} and
		\eqref{eq:phi1_phi2_ratio_asymptotic} that
		\begin{align*}
			\left(
			\int_{y_0}^{y}
			\frac{\phi_2(z)}{\mathcal{W}(z)}\,dz
			\right)
			\frac{\phi_1(y)}{\phi_2(y)}=
			O\left(
			e^{I(y)-x^2/4}x^{\alpha-1}
			\right).
		\end{align*}
		On the other hand, $
		I(y)-\frac{x^2}{4}
		=
		-\frac12
		\left(m+\frac{K_1}{a^2}\right)y^2
		+
		\left(\frac{k\theta}{a^2}+my_*\right)y
		-
		\frac{m}{2}y_*^2$ has a negative quadratic leading term. Together with \eqref{4849} and
		\eqref{eq:phi1_phi2_ratio_asymptotic}, this proves
		\eqref{4855}.\smallskip
		
		It remains to determine the coefficient \(C_2\). From
		\eqref{eq:phi_over_W_asymptotic_start} and
		\eqref{eq:Dalpha_positive_asymptotic},
		\begin{align*}
			\frac{\phi_1(y)}{\mathcal{W}(y)}
			=
			O\left(
			e^{I(y)-x^2/4}x^\alpha
			\right)
			=
			O\left(
			\exp\left\{
			-\frac12
			\left(m+\frac{K_1}{a^2}\right)y^2
			+O(y)
			\right\}
			y^\alpha
			\right)
			\qquad\text{as }y\to\infty.
		\end{align*}
		Thus $
		\int_{y_0}^{\infty}
		\frac{\phi_1(z)}{\mathcal{W}(z)}\,dz$
		is finite, and \eqref{4849} gives
		\begin{align}
			\lim_{y\to\infty}B(y)
			=
			C_2
			-
			\frac{2\beta^\psi}{a^2}
			\int_{y_0}^{\infty}
			\frac{\phi_1(z)}{\mathcal{W}(z)}\,dz.
			\label{eq:B_limit_at_infinity}
		\end{align}
		If \(h\) is bounded, then \(\lim_{y\to\infty}\phi_2(y)=\infty\) implies $A(y)\frac{\phi_1(y)}{\phi_2(y)}+B(y)
		=\frac{h(y)}{\phi_2(y)}
		\longrightarrow 0$ as $y\to\infty.$ Therefore, we may use \eqref{4855} and
		\eqref{eq:B_limit_at_infinity} to obtain $0
		=
		C_2
		-
		\frac{2\beta^\psi}{a^2}
		\int_{y_0}^{\infty}
		\frac{\phi_1(z)}{\mathcal{W}(z)}\,dz$, which confirms the choice of $C_2$ as stated. The computation of $C_1$ can be obtained similarly by taking $y\to -\infty$ in $
		\frac{h(y)}{\phi_1(y)}
		=
		A(y)+B(y)\frac{\phi_2(y)}{\phi_1(y)} .
		$\smallskip
		
		\noindent\textbf{Part 2B. Positivity of $h$:} Since \(\alpha<0\), the integral representation of $D_\alpha$ in \cite[Section 8.3]{HigherTranscendental} implies that \(D_\alpha(x)>0\) for every
		\(x\in\mathbb R\). Moreover, $\mathcal W(y)
		=
		\frac{2\sqrt{\pi m}}{\Gamma(-\alpha)}e^{-2I(y)}>0.$ Since the improper integrals in
		\eqref{eq:full_line_green_solution} are finite, that representation
		immediately gives $h(y)>0$ for $y\in\mathbb R$.\smallskip
		
		\noindent\textbf{Part 2C. Uniqueness of $h$:} Finally, uniqueness follows immediately. If \(h_1\) and \(h_2\) are two bounded positive solutions, then \(h_1-h_2\) solves the homogeneous equation. Hence $
		(h_1-h_2)(y)=A_1\phi_1(y)+A_2\phi_2(y)$ for some constants \(A_1,A_2\). The boundedness and \eqref{4834} force
		\(A_2=0\) when we take $y\to\infty$ and \(A_1=0\) when we take $y \to - \infty$. Thus
		\(h_1-h_2\equiv0\).
	\end{proof}
	
	\begin{corollary}\label{cor explicit sol solves the control problem} If \eqref{4677} holds, then the value function and optimal controls for $\displaystyle\sup_{(\pi,l)\in\mathcal{A}(x,y)}
		\widetilde J_0^{\,\xi,\pi,l}$ are given by \Cref{thm:verification} with the corresponding solution $h$ given in \eqref{eq:full_line_green_solution}.
	\end{corollary}
	
	\begin{proof}  	The relation in \eqref{4677} implies $
		\vartheta
		=
		\frac{1-\gamma}{1-\psi^{-1}}
		=
		\chi\psi
		=
		\chi+1-\gamma.$ Since \(0<\rho<1\), the definition of \(\chi\) gives
		\(\chi<\gamma\). Moreover, since \(\psi \in (0,1)\), Assumption \ref{ass. Parameters constraint} and $0<\vartheta<1$ hold. The dynamics of return process and state variable in \eqref{eq kim omberg} have unique strong solutions and therefore the martingale problem in Assumption \ref{ass.market well-posed} is well posed because the weak law is unique. From \eqref{4703}, Assumptions \ref{ass:kappa} and \ref{ass:polynomial_growth} are satisfied. The scale function of the reference dynamics is given by
		$$\widetilde{s}(y)=\int_c^y \exp\left(-2\int_c^u \frac{\widetilde{b}(s)}{a^2}ds\right)du,$$
		where $
		\widetilde{b}(y):= k\theta
		-\left[k+\left(1-\frac{1}{\gamma}\right)a\rho\right]y$. It
		clearly satisfies \Cref{assu. Qualitative Verification} as $\gamma>1$. Therefore, \Cref{thm verify h is optimal} applies as $h$ is a bounded positive solution from \Cref{prop explicit +ve bdd sol}. \Cref{prop:classical_generalised_value_equality} then identifies the
		classical and generalised values over
		\(\mathcal A_{EZ}(x,y)\) and \(\mathcal A(x,y)\), respectively,
		and shows that the same feedback controls are optimal for the
		generalised problem.  
	\end{proof} 
	
	With \Cref{cor explicit sol solves the control problem}, we estimate the maximum relative error of solutions obtained by computing \eqref{eq:full_line_green_solution} directly versus using the ODE approximation in \Cref{sec. Dirichlet boundary problem} on the bounded domain $[-R,R]$.  
	\begin{table}[htbp]
		\centering
		\label{tab:relative-error-R}
		\small
		\setlength{\tabcolsep}{4pt}
		\renewcommand{\arraystretch}{1.1}
		\begin{tabular}{c|ccccccc}
			\hline
			$R$
			& $2.05$ & $2.1$ & $2.2$ & $2.5$ & $3$ & $5$ & $8$
			\\
			\hline
			Rel. error$/10^{-7}$
			& $25.6696 $ & $9.4489$ & $9.4540$ & $9.4497$ & $9.4562$ & $9.4462$ & $9.4471$
			\\
			\hline
		\end{tabular}
		\caption{The error is measured by $\max_{[-2,2]}\frac{|h_R-h|}{h}$, where $h$ is the explicit solution in \eqref{eq:full_line_green_solution} and $h_R$ is the solution of the ODE in \eqref{eq:BVP_h} on $[-R,R]$. We use $r= 0.0168$, $\sigma=0.151$, $k= 0.271$, $\theta = 0.0788$, $a= 0.0655$, $\gamma=2$, $\rho=0.5$, and $\beta=1$.}
	\end{table}

	\paragraph{Numerical simulation.} In the numerical simulation, we do not restrict ourselves to \eqref{4677}. In this situation, we consider the following parameters from \cite[Table 1]{W02}. 
	
	\begin{table}[H]
		\centering
		\label{tab:wachter_parameters}
		\begin{tabular}{lccccccc}
			\hline
			\textbf{Parameter} & $a$ & $\delta$ & $k$ & $\rho$ & $r$ & $\sigma$ & $\theta$ \\
			\hline
			\textbf{Value} & $0.0655$ & $0.0624$ & $0.271$ & $-0.935$ & $0.0168$ & $0.151$ & $0.0788$ \\
			\hline
		\end{tabular}
		\caption{Annualised parameters based on monthly data from \cite[Table 1]{W02}.}
	\end{table} 
	
	\noindent The above parameters are converted from the monthly discrete-time VAR calibration of
	\cite{barberis2000investing}, based on data from June 1952 to December 1995. 
	Moreover, we set $\beta=1$ and domain size $R=8$. We plot the graphs with $(\gamma,\vartheta)=(2,0.25), (2,0.5), (2,1), (4,0.25), (4,0.5), (4,1)$. Under these values of parameters, \Cref{ass. Parameters constraint} holds; and all other required assumptions are satisfied, which can be proved similarly to \Cref{cor explicit sol solves the control problem}. 
	
	\begin{figure}[H]
		\centering
		\begin{subfigure}[t]{0.48\linewidth}
			\centering
			\includegraphics[width=\linewidth]{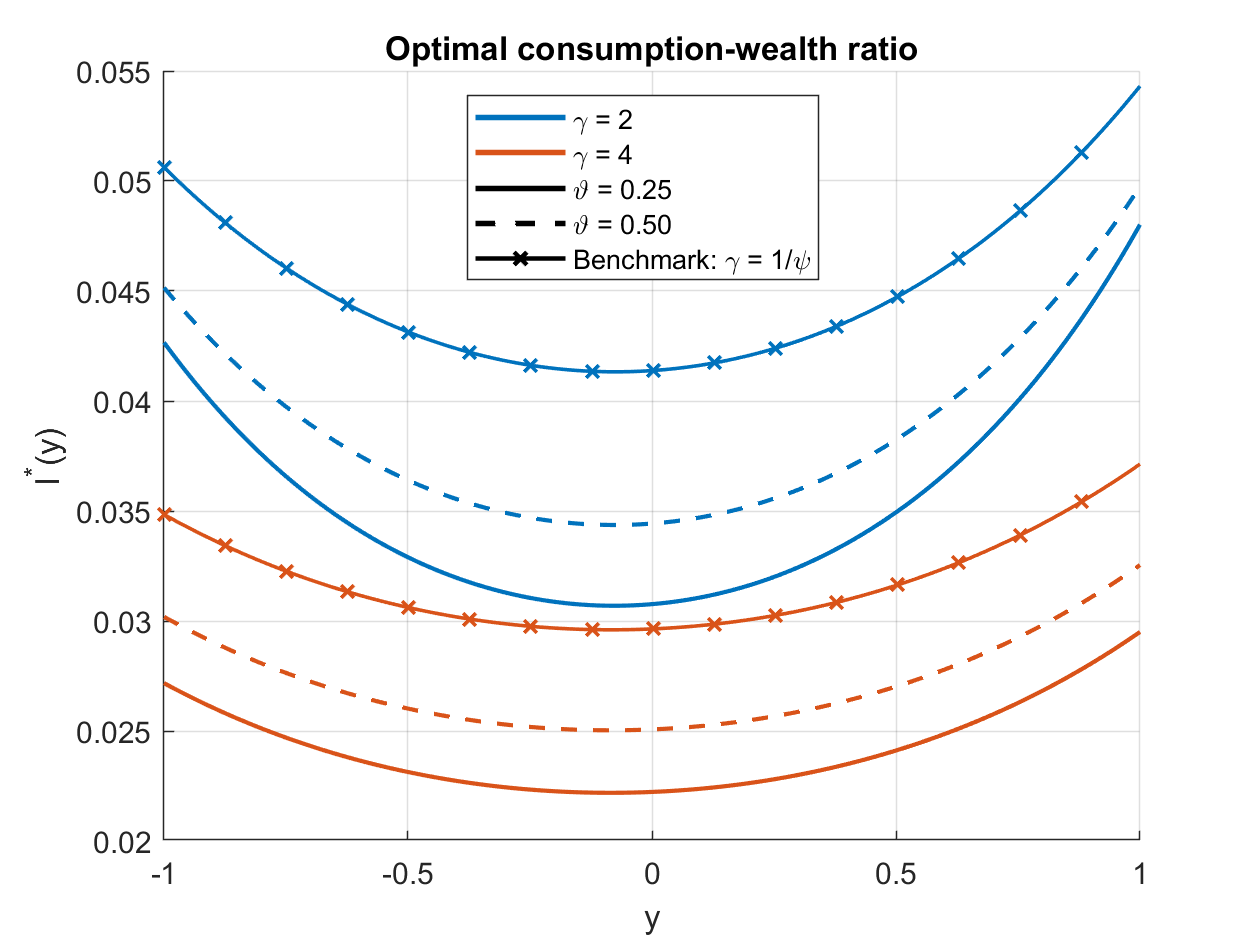} 
			\label{fig:koconsumption}
		\end{subfigure}
		\hspace{-0.01\linewidth}
		\begin{subfigure}[t]{0.48\linewidth}
			\centering
			\includegraphics[width=\linewidth]{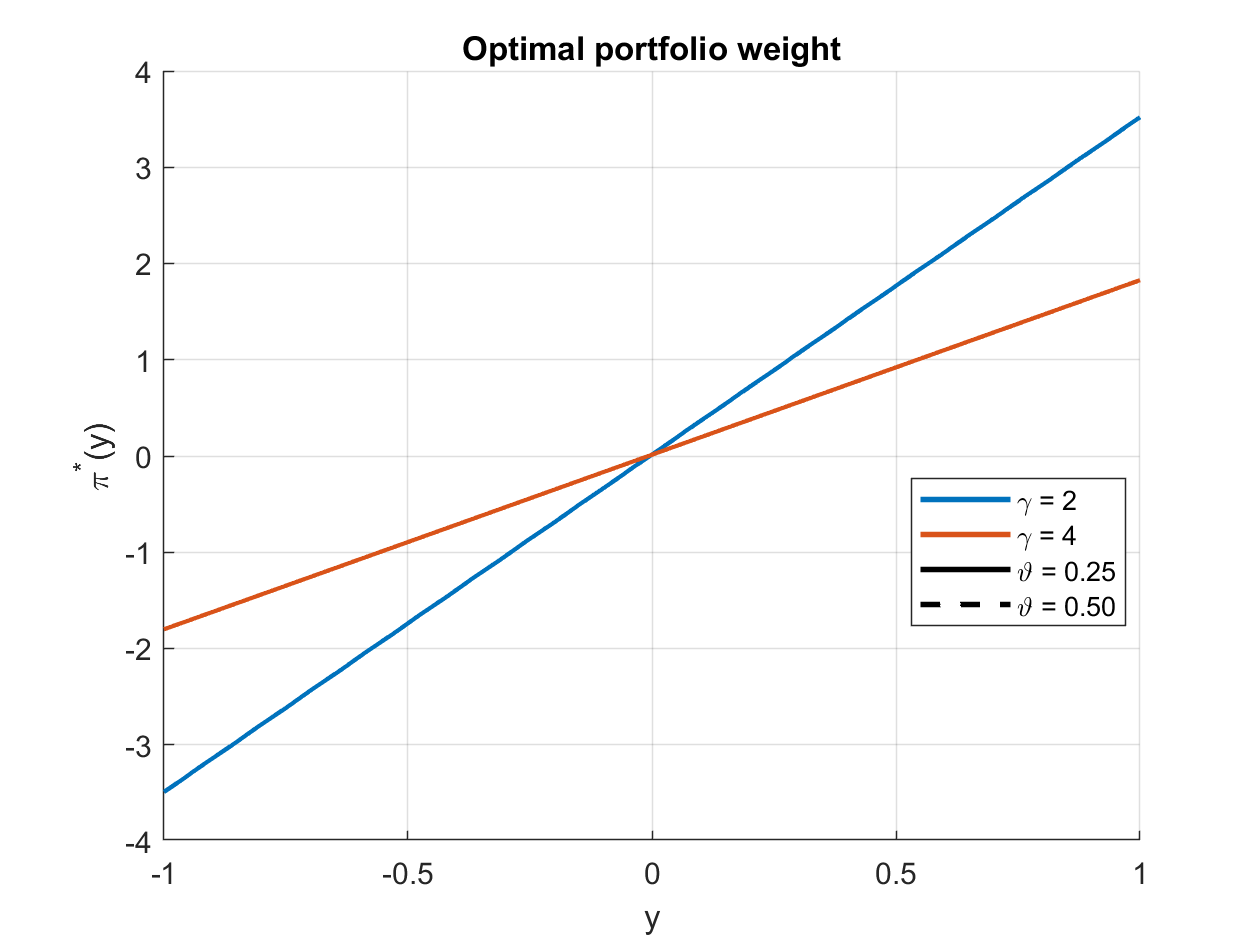} 
			\label{fig:koportfolio}
		\end{subfigure}
		\caption{Optimal consumption-wealth ratio $l^*$ (left) and portfolio weight $\pi^*$ (right) against state variable $y$. Colours
			distinguish the risk-aversion \(\gamma\), while line styles
			distinguish \(\vartheta\).}
		\label{fig:ko_consumption_portfolio}
	\end{figure}
	
	Under the specification in \eqref{eq kim omberg}, the corresponding myopic
	consumption--wealth ratio is
	\[
	l^{\mathrm{myo}}(y)
	=
	\delta\psi
	+
	(1-\psi)
	\left(
	r+\frac{y^2}{2\gamma}
	\right).
	\]
	For the parameters in \Cref{fig:ko_consumption_portfolio}, both
	a large positive and a large negative Sharpe ratio represent good investment opportunities: the investor exploits the former by taking
	a long position in the risky asset and the latter by taking a short position. This explains why the optimal consumption--wealth ratio is lowest when the Sharpe ratio is close to zero and increases towards both tails.\smallskip
	
	For a fixed value of \(\gamma\), the relation $\psi
	=
	\frac{\vartheta}{\vartheta+\gamma-1}$ shows that varying \(\vartheta\) amounts to varying the EIS. Thus, holding the
	colour fixed, the solid, dashed, and marked curves provide a genuine comparison
	with respect to the EIS. Under the present calibration, increasing \(\vartheta\), and hence
	the EIS, raises the optimal consumption--wealth ratio throughout the plotted
	region.\smallskip
	
	In
	the present model,
	\[
	\pi^*(y)
	=
	\underbrace{
		\frac{y}{\gamma\sigma}
	}_{\text{myopic demand}}
	+
	\underbrace{
		\frac{\chi a\rho}{\gamma\sigma}
		\frac{h'(y)}{h(y)}
	}_{\text{intertemporal hedging demand}}.
	\]
	The portfolio in \Cref{fig:ko_consumption_portfolio} is approximately
	linear in \(y\). Since the myopic demand is linear in \(y\), this suggests
	that it accounts for most of the variation in the optimal portfolio. The
	portfolio curves corresponding to different values of \(\vartheta\) are also
	visually almost indistinguishable. Although \(\vartheta\) affects the portfolio indirectly through the hedging term \(h'/h\), the figure indicates that this channel is small relative to the myopic demand under the present calibration. This is consistent with the numerical evidence in
	\cite{xing2017consumption,bayraktar2026infinite}, where changes in
	intertemporal substitution have a substantially stronger effect on consumption
	than on risky investment. Increasing \(\gamma\) from \(2\) to \(4\) substantially flattens the optimal
	portfolio rule. This comparison should not, however, be interpreted as a pure
	risk-aversion comparative static, because holding \(\vartheta\) fixed while
	changing \(\gamma\) also changes \(\psi\).

	\subsection{Heston-type stochastic-volatility}\label{sec. Heston-type stochastic-volatility}
	
	Let \(E=\mathbb{R}_{+}\). We consider the Heston-type stochastic-volatility model of
	\cite{10.1093/rfs/6.2.327,PAN20023}. The market is
	\[
	\frac{dS_t^0}{S_t^0}=r\,dt,\qquad
	\frac{dS_t^1}{S_t^1}
	=
	\pig(r+\underbrace{\mu_0Y_t}_{\mu(\cdot)}\pig)dt+\underbrace{\sqrt{Y_t}}_{\sigma(\cdot)}\,dZ_t,\qquad
	dY_t
	=
	\underbrace{k(\theta-Y_t)}_{b(\cdot)}\,dt+\underbrace{\zeta\sqrt{Y_t}}_{a(\cdot)}\,dW_t,
	\]
	where $r$, $\mu_0$, $k$, $\theta$, and $\zeta$ are positive constants, and $ 2k\theta\geq \zeta^2$. In this situation, the functions $	\widetilde{b}$ and $\kappa$ are defined as $\widetilde{b}(y):= k(\theta-y)
	-\left(1-\frac{1}{\gamma}\right)\mu_0\zeta\rho y$ and $\kappa(y):=\frac{1-\gamma}{\chi}
	\left(
	\frac{\delta\psi}{\psi-1}
	-r
	-\frac{\mu_0^2}{2\gamma}y
	\right)$.
	
	\begin{remark}[\bf Explicit solution and relative error]
		Assume that $\delta>0$, $\rho \in (0,1)$, $\gamma>1$, $\psi=1+\frac{1-\gamma}{\chi}$, $\frac{2k\theta}{\zeta^2}\geq 1$ and that \Cref{ass. Parameters constraint} is satisfied. Define $\phi_3(y)
		:=
		e^{q'y}M(a_0,b_0,sy)$ and $\phi_4(y)
		:=
		e^{q'y}U(a_0,b_0,sy)$, where \(M={}_1F_1\) is Kummer's confluent hypergeometric function and
		\(U\) is Tricomi's confluent hypergeometric function. Here the constants are
		\[
		\begin{alignedat}{3}
			\Delta
			&:=
			\sqrt{
				\left(k+\left(1-\frac{1}{\gamma}\right)\mu_0\zeta\rho\right)^2
				-
				\zeta^2\frac{1-\gamma}{\chi}\frac{\mu_0^2}{\gamma}
			},
			&\qquad
			q'
			&:=
			\frac{
				k+\left(1-\frac{1}{\gamma}\right)\mu_0\zeta\rho-\Delta
			}{\zeta^2},
			&\qquad
			s
			&:=
			\frac{2\Delta}{\zeta^2},
			\\[0.6em]
			b_0
			&:=
			\frac{2k\theta}{\zeta^2},
			&\qquad
			a_0
			&:=
			\frac{
				\frac{1-\gamma}{\chi}
				\left(
				\frac{\delta\psi}{\psi-1}-r
				\right)
				-k\theta q'
			}{\Delta}.
		\end{alignedat}
		\]
		Let $\mathcal{W}(y)
		:=
		\phi_3'(y)\phi_4(y)-\phi_3(y)\phi_4'(y).$ Then the unique bounded positive solution of \eqref{eq: reduced HJB} is
		\begin{equation}
			h(y)
			=
			\dfrac{2\beta^\psi}{\zeta^2}
			\left[
			\phi_4(y)
			\int_0^y
			\frac{\phi_3(z)}{ z\, \mathcal{W}(z)}\,dz
			+
			\phi_3(y)
			\int_y^\infty
			\frac{\phi_4(z)}{z\, \mathcal{W}(z)}\,dz
			\right],
			\qquad \text{for }y>0.
			\label{eq:heston_green_solution}
		\end{equation}
		The verification for $h$ can be proved similarly to \Cref{cor explicit sol solves the control problem} by using the boundedness and positivity of $h$. We compute this explicit solution $h$ directly and compute the solution $h_{1/R,R}$ of ODE in \eqref{eq:eps_BVP} on the domain $[1/R,R]$. If we choose $R=10^{5}$, the maximum relative error $\max_{[10^{-4},10^{4}]}\frac{|h_{1/R,R}-h|}{h}$ is less than $3 \times 10^{-5}$, for some suitable parameters. This justifies the accuracy of our numerical scheme.
	\end{remark}

	Following \cite[Tables~1,6]{PAN20023}, we use the following parameter values: 
	\begin{table}[H]
		\centering
		\label{tab:heston_parameters}
		\begin{tabular}{lccccccc}
			\hline
			\textbf{Parameter} & $\delta$ & $k$ & $\mu_0$ & $r$ & $\rho$ & $\theta$ & $\zeta$ \\
			\hline
			\textbf{Value} & $0.0624$ & $7.1$ & $8.6$ & $0.058$ & $-0.53$ & $0.0137$ & $0.32$ \\
			\hline
		\end{tabular}
		\caption{Parameters based on the SV estimates in \cite[Tables~1, 6]{PAN20023}.}
	\end{table}

	\noindent
	Specifically, \cite{PAN20023} estimates these values using an implied-state generalised method of moments (IS-GMM) procedure applied to the joint weekly time series of the S\&P 500 index and near-the-money short-dated option prices from January 1989 to December 1996. 
	The impatience rate \(\delta=0.0624\) is chosen separately. We also set $\beta=1$ and the domain $[1/R,R]=[10^{-5},10^{5}]$. 
	
	\begin{figure}[H]
		\centering
		\begin{subfigure}[t]{0.48\linewidth}
			\centering
			\includegraphics[width=\linewidth]{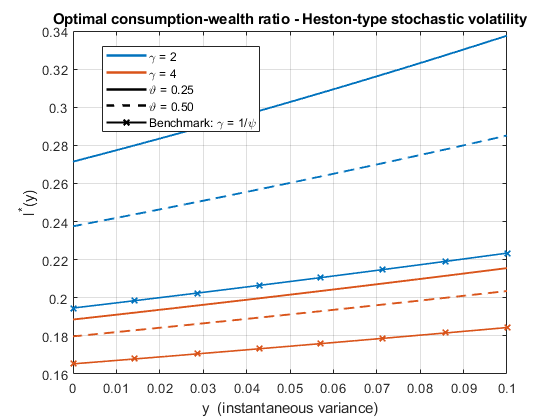} 
			\label{fig:Heston_consumption}
		\end{subfigure}
		\hspace{-0.01\linewidth}
		\begin{subfigure}[t]{0.48\linewidth}
			\centering
			\includegraphics[width=\linewidth]{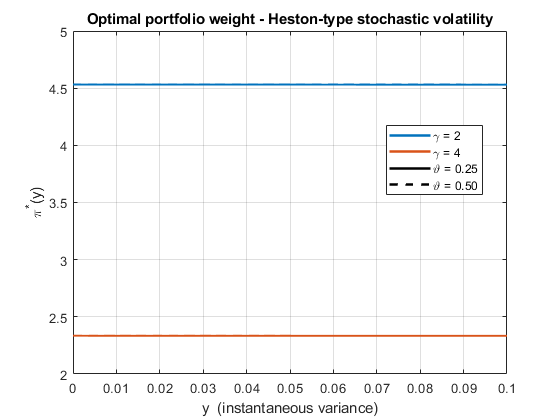} 
			\label{fig:Heston_portfolio}
		\end{subfigure}
		\caption{Optimal consumption-wealth ratio $l^*$ (left) and portfolio weight $\pi^*$ (right) against state variable $y$}
		\label{fig:H_consumption_portfolio}
	\end{figure}

	Under the Heston-type specification, the instantaneous Sharpe
	ratio is \(\mu_0\sqrt{y}\). A higher value of \(y\) raises the current consumption in the left panel. For a fixed value of \(\gamma\), increasing \(\vartheta\) also increases the EIS. A higher EIS lowers current consumption, particularly
	when \(y\) is large. Economically, a greater willingness to substitute
	consumption across time makes the investor more willing to postpone consumption
	when the Sharpe ratio is high. The near-flat portfolio rules in the right panel follow from
	\[
	\pi^*(y)
	=\frac{\mu_0}{\gamma}
	+
	\frac{\chi\zeta\rho}{\gamma}
	\frac{h'(y)}{h(y)}.
	\]
	The myopic demand is independent of the state variable. The remaining
	state dependence comes entirely from the hedging term, which is small and
	nearly constant over the plotted region. This also explains why changing
	\(\vartheta\) has little visible effect on portfolio choice, whereas increasing
	\(\gamma\) from \(2\) to \(4\) substantially reduces the risky-asset weight.\smallskip
	
	Moreover, the consumption curves imply \(h'(y)/h(y)<0\) in the plotted region due to \eqref{eq:feedback_controls}. Since \(\rho<0\), the intertemporal hedging demand is positive, so the optimal portfolio lies slightly above the myopic benchmark \(\mu_0/\gamma\). A negative stock-return shock tends to be accompanied by a rise in variance and therefore an improvement in future investment opportunities; this continuation-value effect partially insures the investor against the contemporaneous loss and supports a larger risky position.
	
	\subsection{CIR stochastic interest rates}\label{sec. CIR stochastic interest rates} Let \(E=\mathbb{R}_{+}\). We consider a market with a CIR stochastic short rate:
	\[
	\frac{dS_t^0}{S_t^0}
	=
	\underbrace{Y_t}_{r(\cdot)}\,dt,
	\qquad
	\frac{dS_t^1}{S_t^1}
	=
	\left(Y_t+\mu_0\right)dt+\sigma_0\,dZ_t,
	\qquad
	dY_t
	=
	\underbrace{k(\theta-Y_t)}_{b(\cdot)}\,dt+\underbrace{\zeta\sqrt{Y_t}}_{a(\cdot)}\,dW_t.
	\] 
	Here $k, \theta, \zeta, \sigma_0$, and $\mu_0$ are all positive, and $ 2k\theta\geq \zeta^2$. In this situation, the functions \(\widetilde b\) and \(\kappa\) are given by $\widetilde b(y) := k(\theta-y) - \left(1-\frac{1}{\gamma}\right) \frac{\mu_0\zeta\rho}{\sigma_0}\sqrt{y}$ and $\kappa(y) := \frac{1-\gamma}{\chi} \left( \frac{\delta\psi}{\psi-1}  -\frac{\mu_0^2}{2\gamma\sigma_0^2} -y\right)$. For the numerical experiment, we use the following parameter values 
	\begin{table}[H]
		\centering
		\label{tab:kim_cir_parameters}
		\begin{tabular}{lccccccc}
			\hline
			\textbf{Parameter} & $\delta$ & $k$ & $\mu_0$ & $\rho$ & $\sigma_0$ & $\theta$ & $\zeta$ \\
			\hline
			\textbf{Value} & $0.0624$ & $1.2188$ & $0.0417$ & $-0.789$ & $0.2438$ & $0.0183$ & $0.0289$ \\
			\hline
		\end{tabular}
	\end{table}

	\noindent We also set $\beta=1$ and the domain $[1/R,R]=[10^{-5},10^{5}]$. 
	
	\begin{figure}[H]
		\centering
		\begin{subfigure}[t]{0.49\linewidth}
			\centering
			\includegraphics[width=\linewidth]{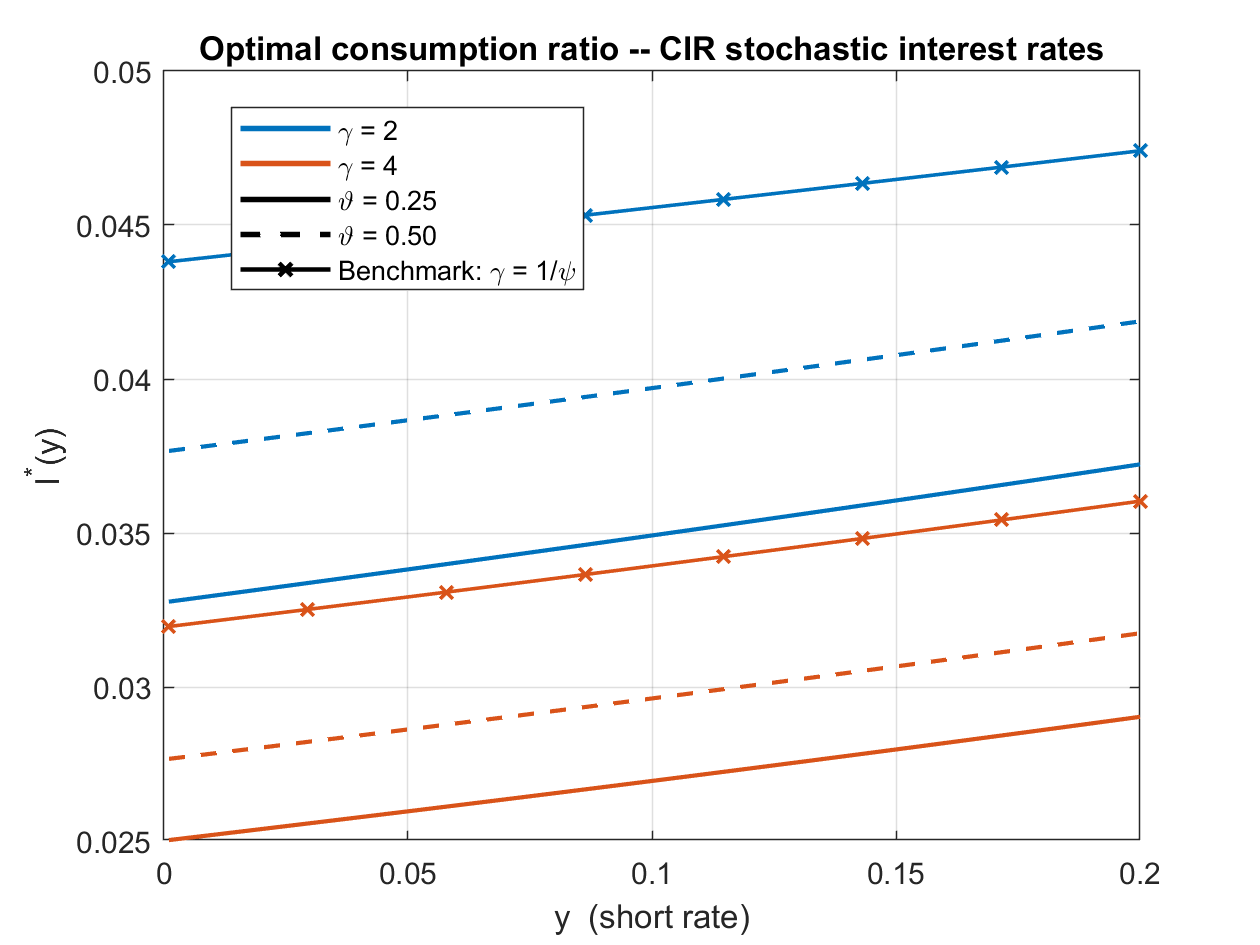} 
			\label{fig:CIR_consumption}
		\end{subfigure}
		\hspace{-0.01\linewidth}
		\begin{subfigure}[t]{0.49\linewidth}
			\centering
			\includegraphics[width=\linewidth]{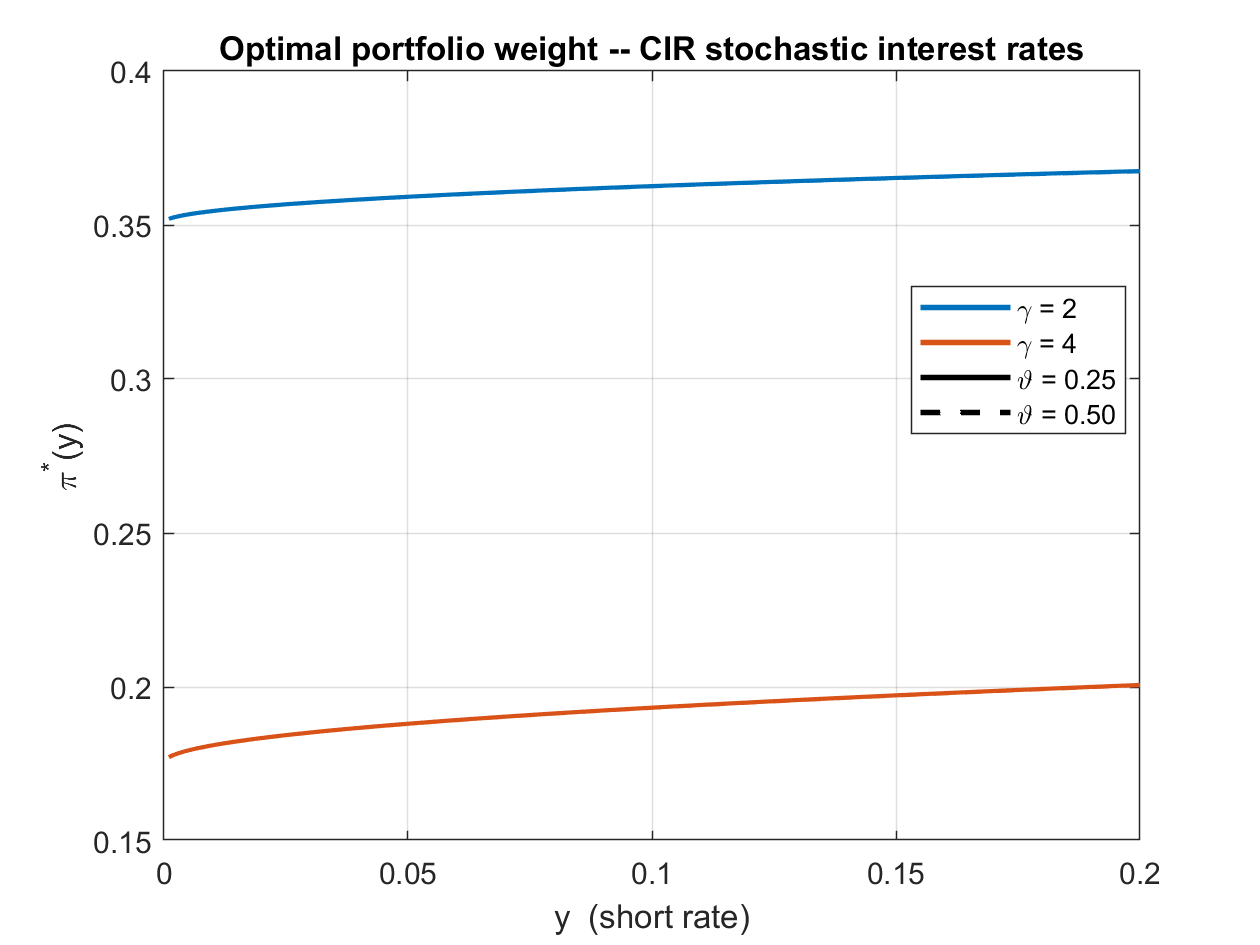} 
			\label{fig:CIR_portfolio}
		\end{subfigure}
		\caption{Optimal consumption-wealth ratio $l^*$ (left) and portfolio weight $\pi^*$ (right) against state variable $y$}
		\label{fig:CIR_consumption_portfolio}
	\end{figure}
	
	Under the CIR specification, the consumption--wealth ratio increases with \(y\), as observed in the left panel. For a fixed value of \(\gamma\), a larger \(\vartheta\) corresponds to a larger
	EIS. Under the present calibration, increasing \(\vartheta\) shifts the
	consumption--wealth ratio upward throughout the plotted region. As in the preceding examples, the effect of intertemporal substitution is considerably more visible in consumption than in portfolio choice.\smallskip
	
	The portfolio rule is
	\[
	\pi^*(y)
	=
	\frac{\mu_0}{\gamma\sigma_0^2}
	+
	\frac{\chi\zeta\rho}{\gamma\sigma_0}
	\sqrt{y}\,
	\frac{h'(y)}{h(y)}.
	\]
	The myopic demand is independent of \(y\). Thus, unlike the mean-reverting
	risk-premium model in \Cref{sec. Mean reverting risk premium}, all state dependence in the right panel is generated by intertemporal hedging. The increasing consumption rule implies that \(h'(y)/h(y)<0\) over the plotted region. Since \(\rho<0\), the hedging demand is positive, and the optimal risky-asset weight lies above its myopic benchmark.\smallskip
	
	The sign of this hedging demand has a natural interpretation. A fall in the
	short rate worsens future risk-free investment opportunities, but the negative
	correlation between stock and short-rate shocks means that the risky asset tends
	to perform well when the short rate falls. Holding the risky asset therefore
	provides insurance against an adverse change in the investment opportunity.
	The mild increase of the portfolio with \(y\) reflects the factor loading
	\(\sqrt{y}\) in the hedging term. Increasing \(\gamma\) from \(2\) to \(4\) substantially reduces the portfolio weight, primarily through the factor \(1/\gamma\). Similar to previous examples, the curves associated with different values of \(\vartheta\) are nearly indistinguishable.

	\paragraph{Acknowledgement.} H. M. Tai is partially supported by Australian Research Council Discovery Project DP240100781.

	\bibliographystyle{abbrv}
	\bibliography{bio}

@book{karatzas2014brownian,
   title     = "Brownian Motion and Stochastic Calculus",
   author    = "Karatzas, Ioannis and Shreve, Steven E",
   publisher = "Springer",
   edition   =  "2nd",
   year      =  "1998",
  }

@article{herdegen2023infinite_I,
   title={{The infinite-horizon investment--consumption problem for Epstein--Zin stochastic differential utility. I: Foundations}},
   author={Herdegen, Martin and Hobson, David and Jerome, Joseph},
   journal={Finance and Stochastics},
   volume={27},
   number={1},
   pages={127--158},
   year={2023},
   publisher={Springer}
  }

@article{herdegen2023infinite_II,
   title = {The infinite-horizon investment--consumption problem for {Epstein--Zin} stochastic differential utility. {II}: Existence, uniqueness and verification for $\vartheta \in (0,1)$},
   author={Herdegen, Martin and Hobson, David and Jerome, Joseph},
   journal={Finance and Stochastics},
   volume={27},
   number={1},
   pages={159--188},
   year={2023},
   publisher={Springer}
  }

@article{xing2017consumption,
   title={{Consumption--investment optimization with Epstein--Zin utility in incomplete markets}},
   author={Xing, Hao},
   journal={Finance and Stochastics},
   volume={21},
   number={1},
   pages={227--262},
   year={2017},
   publisher={Springer}
  }

@article{MatoussiXing,
   author = {Matoussi, Anis and Xing, Hao},
   title = {{Convex duality for Epstein--Zin stochastic differential utility}},
   journal = {Mathematical Finance},
   volume = {28},
   number = {4},
   pages = {991-1019},
   year = {2018}
  }

@article{kraft2017optimal,
   title={{Optimal consumption and investment with Epstein--Zin recursive utility}},
   author={Kraft, Holger and Seiferling, Thomas and Seifried, Frank Thomas},
   journal={Finance and Stochastics},
   volume={21},
   number={1},
   pages={187--226},
   year={2017},
   publisher={Springer}
  }

@article{feng2025consumption,
   title = {{Consumption–investment optimization with Epstein--Zin utility in unbounded non-Markovian markets}},
   author={Feng, Zixin and Tian, Dejian and Zheng, Harry},
   journal = {Stochastic Processes and their Applications},
   volume = {192},
   pages = {104805},
   year = {2026}
  }

@book{revuz2013continuous,
   title={Continuous martingales and Brownian motion},
   author={Revuz, Daniel and Yor, Marc},
   edition   = {3rd},
   year={1999},
   publisher={Springer Science \& Business Media}
  }

@article{guasoni2020consumption,
   title={Consumption in incomplete markets},
   author={Guasoni, Paolo and Wang, Gu},
   journal={Finance and Stochastics},
   volume={24},
   number={2},
   pages={383--422},
   year={2020},
   publisher={Springer}
  }

@article{PaoloScott12,
   author = {Paolo Guasoni and Scott Robertson},
   title = {{Portfolios and risk premia for the long run}},
   volume = {22},
   journal = {The Annals of Applied Probability},
   number = {1},
   publisher = {Institute of Mathematical Statistics},
   pages = {239 -- 284},
   year = {2012}
  }

@article{guasoni2025variational,
   title={A Variational Approach to Portfolio Choice},
   author={Guasoni, Paolo and Lawless, Emmet and Tai, Ho Man},
   journal={Available at SSRN 5669613},
   year={2025}
  }

@article{SATTINGER,
   author = {D. H. Sattinger},
   journal = {Indiana University Mathematics Journal},
   number = {11},
   pages = {979--1000},
   publisher = {Indiana University Mathematics Department},
   title = {Monotone Methods in Nonlinear Elliptic and Parabolic Boundary Value Problems},
   volume = {21},
   year = {1972}
  }

@book{GT77,
   title = {Elliptic Partial Differential Equations of Second Order},
   author = {Gilbarg, D. and Trudinger, N. S.},
   series = {Classics in Mathematics},
   edition = {2nd},
   publisher = {Springer},
   address = {Berlin},
   year = {2001},
   doi = {10.1007/978-3-642-61798-0}
  }

@article{MawhinSchmitt1984,
   title={Upper and lower solutions and semilinear second order elliptic equations with non-linear boundary conditions}, volume={97},
   journal={Proceedings of the Royal Society of Edinburgh Section A: Mathematics},
   author={Mawhin, J. and Schmitt, K.},
   year={1984},
   pages={199–207}
  }

@article{SCHMITT1978263,
   title = {Boundary value problems for quasilinear second order elliptic equations},
   journal = {Nonlinear Analysis: Theory, Methods \& Applications},
   volume = {2},
   number = {3},
   pages = {263-309},
   year = {1978},
   author = {Klaus Schmitt}
  }

@article{herdegen2025proper,
   title={{Proper solutions for Epstein--Zin stochastic differential utility}},
   author={Herdegen, Martin and Hobson, David and Jerome, Joseph},
   journal={Finance and Stochastics},
   volume={29},
   number={3},
   pages={885--932},
   year={2025},
   publisher={Springer}
  }

@book{Protter1990, address={Berlin}, title={Stochastic Integration and Differential Equations: A New Approach},  publisher={Springer Berlin Heidelberg}, author={Protter, Philip}, year={1990} }

@article{KE96,
   title={Dynamic nonmyopic portfolio behavior},
   author={Kim, Tong Suk and Omberg, Edward},
   journal={The Review of Financial Studies},
   volume={9},
   number={1},
   pages={141--161},
   year={1996},
   publisher={Oxford University Press}
  }

@article{W02,
   title={Portfolio and consumption decisions under mean-reverting returns: An exact solution for complete markets},
   author={Wachter, Jessica A},
   journal={Journal of Financial and Quantitative Analysis},
   volume={37},
   number={1},
   pages={63--91},
   year={2002},
   publisher={Cambridge University Press}
  }

@article{barberis2000investing,
   title={Investing for the long run when returns are predictable},
   author={Barberis, Nicholas},
   journal={The Journal of Finance},
   volume={55},
   number={1},
   pages={225--264},
   year={2000},
   publisher={Wiley Online Library}
  }

@article{10.1093/rfs/6.2.327,
   author = {Heston, Steven L.},
   title = {A Closed-Form Solution for Options with Stochastic Volatility with Applications to Bond and Currency Options},
   journal = {The Review of Financial Studies},
   volume = {6},
   number = {2},
   pages = {327-343},
   year = {1993}
  }

@article{PAN20023,
   title = {The jump-risk premia implicit in options: evidence from an integrated time-series study},
   journal = {Journal of Financial Economics},
   volume = {63},
   number = {1},
   pages = {3-50},
   year = {2002},
   author = {Jun Pan}
  }

@book{HigherTranscendental,
   author    = {Bateman, Harry},
   title     = {Higher Transcendental Functions},
   volume    = {II},
   publisher = {McGraw-Hill Book Company},
   address   = {New York},
   year      = {1953}
  }

@article{temme2000numerical,
   title={Numerical and asymptotic aspects of parabolic cylinder functions},
   author={Temme, Nico M},
   journal={Journal of Computational and Applied Mathematics},
   volume={121},
   number={1-2},
   pages={221--246},
   year={2000},
   publisher={Elsevier}
  }

@article{KrepsPorteus1978,
   author = {David M. Kreps and Evan L. Porteus},
   journal = {Econometrica},
   number = {1},
   pages = {185--200},
   title = {Temporal Resolution of Uncertainty and Dynamic Choice Theory},
   urldate = {2026-07-11},
   volume = {46},
   year = {1978}
  }

@article{kraft2013consumption,
   title={Consumption-portfolio optimization with recursive utility in incomplete markets},
   author={Kraft, Holger and Seifried, Frank Thomas and Steffensen, Mogens},
   journal={Finance and Stochastics},
   volume={17},
   number={1},
   pages={161--196},
   year={2013},
   publisher={Springer}
  }

@article{DuffieEpstein1992,
   author = {Duffie, Darrell and Epstein, Larry G.},
   title = {Asset Pricing with Stochastic Differential Utility},
   journal = {The Review of Financial Studies},
   volume = {5},
   number = {3},
   pages = {411-436},
   year = {1992}
  }

@article{DuffieEpstein1992a,
   author = {Darrell Duffie and Larry G. Epstein},
   journal = {Econometrica},
   number = {2},
   pages = {353--394},
   title = {Stochastic Differential Utility},
   urldate = {2026-07-11},
   volume = {60},
   year = {1992}
  }

@article{epstein1989substitution,
   title={Substitution, Risk Aversion, and the Temporal Behavior of Consumption and Asset Returns: A Theoretical Framework},
   author={Epstein, Larry G and Zin, Stanley E},
   journal={Econometrica},
   volume={57},
   number={4},
   pages={937--969},
   year={1989}
  }

@article{bayraktar2026infinite,
   title={Infinite Horizon Optimal Consumption: Intertemporal Hedging under {Epstein--Zin} Preferences},
   author={Bayraktar, Erhan and Lawless, Emmet},
   journal={arXiv preprint arXiv:2606.02945},
   year={2026}
  }

@article{Shigeta2026,
   author  = {Shigeta, Yuki},
   title   = {An economic interpretation and mathematical analysis of
    {Epstein--Zin} stochastic differential utility for an
    infinite horizon when {$\theta<0$}},
   journal = {Finance and Stochastics},
   volume  = {30},
   number = {3},
   pages   = {765--819},
   year    = {2026}
  }

@article{BansalYaron2004,
   author = {Bansal, Ravi and Yaron, Amir},
   title = {Risks for the Long Run: A Potential Resolution of Asset Pricing Puzzles},
   journal = {The Journal of Finance},
   volume = {59},
   number = {4},
   pages = {1481-1509},
   year = {2004}
  }

@article{AurandHuang_random_horizon,
   author = {Aurand, Joshua and Huang, Yu-Jui},
   title = {{Epstein-Zin} utility maximization on a random horizon},
   journal = {Mathematical Finance},
   volume = {33},
   number = {4},
   pages = {1370-1411},
   year = {2023}
  }

@article{gutekunst2025optimal,
   title={Optimal Investment and Consumption in a Stochastic Factor Model},
   author={Gutekunst, Florian and Herdegen, Martin and Hobson, David},
   journal={arXiv preprint arXiv:2509.09452},
   year={2025}
  }

@article{SCHRODER2003155,
   title = {Optimal lifetime consumption-portfolio strategies under trading constraints and generalized recursive preferences},
   author={Schroder, Mark and Skiadas, Costis},
   journal = {Stochastic Processes and their Applications},
   volume = {108},
   number = {2},
   pages = {155-202},
   year = {2003}
  }

@article{SCHRODER20051,
   title = {Lifetime consumption-portfolio choice under trading constraints, recursive preferences, and nontradeable income},
   author={Schroder, Mark and Skiadas, Costis},
   journal = {Stochastic Processes and their Applications},
   volume = {115},
   number = {1},
   pages = {1-30},
   year = {2005}
  }

@article{EZ_transaction,
   author = {Melnyk, Yaroslav and Muhle-Karbe, Johannes and Seifried, Frank Thomas},
   title = {Lifetime investment and consumption with recursive preferences and small transaction costs},
   journal = {Mathematical Finance},
   volume = {30},
   number = {3},
   pages = {1135-1167},
   doi = {https://doi.org/10.1111/mafi.12245},
   url = {https://onlinelibrary.wiley.com/doi/abs/10.1111/mafi.12245},
   eprint = {https://onlinelibrary.wiley.com/doi/pdf/10.1111/mafi.12245},
   year = {2020}
  }

@article{KANG2025108675,
   author  = {Kang, Jianhao and Gou, Zhun and Huang, Nanjing},
   title   = {Consumption and Portfolio Optimization Solvable Problems
    with Recursive Preferences},
   journal = {Communications in Nonlinear Science and Numerical Simulation},
   volume  = {144},
   pages   = {108675},
   year    = {2025},
   issn    = {1007-5704},
   doi     = {10.1016/j.cnsns.2025.108675},
   url     = {https://www.sciencedirect.com/science/article/pii/S1007570425000863}
  }

@article{BenzoniCollinDufresneGoldstein2011,
   title = {Explaining asset pricing puzzles associated with the 1987 market crash},
   journal = {Journal of Financial Economics},
   volume = {101},
   number = {3},
   pages = {552-573},
   year = {2011},
   issn = {0304-405X},
   doi = {https://doi.org/10.1016/j.jfineco.2011.01.008},
   url = {https://www.sciencedirect.com/science/article/pii/S0304405X11000328},
   author = {Luca Benzoni and Pierre Collin-Dufresne and Robert S. Goldstein}
  }

@article{BhamraKuehnStrebulaev2010,
   author = {Bhamra, Harjoat S. and Kuehn, Lars-Alexander and Strebulaev, Ilya A.},
   title = {The Levered Equity Risk Premium and Credit Spreads: A Unified Framework},
   journal = {The Review of Financial Studies},
   volume = {23},
   number = {2},
   pages = {645-703},
   year = {2010}
  }

@article{Wachter2013,
   author={Wachter, Jessica A},
   title = {Can Time-Varying Risk of Rare Disasters Explain Aggregate Stock Market Volatility?},
   journal = {The Journal of Finance},
   volume = {68},
   number = {3},
   pages = {987-1035},
   year = {2013}
  }

@article{AiKiku2013,
   title = {Growth to value: Option exercise and the cross section of equity returns},
   author={Ai, Hengjie and Kiku, Dana},
   journal = {Journal of Financial Economics},
   volume = {107},
   number = {2},
   pages = {325-349},
   year = {2013}
  }

@article{DuffieSkiadas1994,
   title = {Continuous-time security pricing: A utility gradient approach},
   journal = {Journal of Mathematical Economics},
   volume = {23},
   number = {2},
   pages = {107-131},
   year = {1994},
   author = {Darrell Duffie and Costis Skiadas}
  }

@article{ElKarouiPengQuenez2001,
   author = {El Karoui, Nicole and Peng, Shige and Quenez, Marie-Claire},
   title = {{A dynamic maximum principle for the optimization of recursive utilities under constraints}},
   volume = {11},
   journal = {The Annals of Applied Probability},
   number = {3},
   publisher = {Institute of Mathematical Statistics},
   pages = {664 -- 693},
   year = {2001}
  }
\end{document}